\documentclass[11pt]{article}      
\usepackage[T1]{fontenc}           
\usepackage[utf8]{inputenc}        
\usepackage{graphicx}      
\usepackage[english]{babel} 
\usepackage{geometry}              
\usepackage{mathptmx}              
\usepackage{amsmath}               
\usepackage{caption}
\usepackage{subcaption}
\usepackage{verbatim} 
\usepackage{amssymb}
\usepackage{amsfonts}
\usepackage{amsthm}
\usepackage{cancel} 
\usepackage{comment}
\usepackage{hyperref} 
\usepackage{empheq} 

\usepackage{dsfont}
\usepackage{csquotes}
\usepackage{tikz}
\usepackage{tkz-tab} 
\usepackage{array}
\usepackage{float} 
\usepackage{comment} 
\usepackage{authblk}

\newcommand{\dt}{\partial_{t}}
\newcommand{\dx}{\partial_{x}}
\newcommand{\dy}{\partial_{y}}

\newcommand{\RR}{\mathbb{R}}
\newcommand{\NN}{\mathbb{N}}

\DeclareMathOperator{\dive}{div}
\newcommand{\ddx}{\text{dx}}
\newcommand{\ddy}{\text{dy}}
\newcommand{\ddt}{\text{dt}}
\DeclareMathOperator{\be}{\bar \epsilon^2}
\newcommand{\hh}{\bar h}

\newcommand{\bK}{\bar K}
\DeclareMathOperator{\ba}{\bar a_1}
\DeclareMathOperator{\baa}{\bar a_2}
\DeclareMathOperator{\bb}{\bar b_1}
\DeclareMathOperator{\bbb}{\bar b_2}

\newcommand{\hk}{H^k}
\newcommand{\hkun}{H^{k+1}}
\newcommand{\lhk}{L^2_T(H^k)}

\newcommand{\lhkde}{L^2_T(H^{k+2})}

\newtheorem{theorem}{Theorem}[section]
\newtheorem{proposition}{Proposition}[section]

\newtheorem{remark}{Remark}[section]

\title{Well-posedness and stability analysis\\ of a landscape evolution model}
\date{}
\author[1]{Julie Binard}
\author[2]{Pierre Degond}
\author[1]{Pascal Noble}

\affil[1]{Institut de Mathématiques de Toulouse; UMR 5219\\ 
Université de Toulouse; CNRS \\
INSA, F-31077 Toulouse, France\\
email JB: binard@insa-toulouse.fr\\
email PN: noble@insa-toulouse.fr
}
\affil[2]{Institut de Mathématiques de Toulouse; UMR 5219\\ 
Université de Toulouse; CNRS \\
UPS, F-31062 Toulouse Cedex 9, France\\
email: pierre.degond@math.univ-toulouse.fr
}

\begin{document}

\maketitle

\begin{abstract}
In this paper, we study a system of partial differential equations modeling the evolution of a landscape  in order to describe the mechanisms of pattern formations. A ground surface is eroded by the flow of water over it, either by sedimentation or dilution. We consider a  model, composed of three evolution equations, one on the elevation of the ground surface, one on the fluid height and one on the concentration of sediments in the fluid layer. We first establish the well-posedness of the system in short time, and under the assumption that the initial fluid height does not vanish. Then, we focus on pattern formation in the case of a film flow over an inclined erodible plane. For that purpose, we carry out a spectral stability analysis of constant state solutions in order to determine instability conditions and identify a mechanism for pattern formations. These patterns, which are rills and gullies, are the starting point of the formation of rivers and valleys in landscapes. Finally, we carry out some numerical simulations of the full system in order to validate the spectral instability scenario, and determine the resulting patterns.
\end{abstract}

\paragraph{Keywords}
Pattern formation, instabilities,
Landscape evolution model, erosion by water, stream incision law, sedimentation.

\paragraph{Mathematics Subject Classification}
35M30, 35Q86, 35B35, 35B36, 35A01, 86-10

\section{Introduction}

The modeling of landscape evolution under the effect of water flow has received increasing attention these last decades. Several physical phenomena must be considered in order to describe the erosion and sedimentation processes, which occur when water flows over an erodible surface. The erosion is the removal of sediments from the soil by a fluid, and the sedimentation is the inverse process, when the sediments settle on the surface. Heuristically, the intensity of the erosion process is strongly related to the flow rate of the fluid, the erosion rate is higher when the flow rate is higher. An other process that affects landscape evolution is the creep effect, which tends to smooth the bottom surface on time scales that are much larger than the ones involved in sediment transport. This effect has multiple causes, such as gravitational forces acting on the soil, and corrosion and dilatation caused by physical and chemical factors. It can be described as a simple diffusion process of the soil. The creep effect was introduced in 1892, in \cite{davis_convex_1892}. It was used in \cite{gilbert_convexity_1909} to explain the convexity of hilltops profile. Later, in 1963 a derivation of the soil creep effect as a limit of a stochastic process had been done in \cite{culling1963soil}. The stochastic effect models soil particles which follow a random motion, under the constrain of gravity. In this article, Culling writes "Soil creep \textemdash so gradual as to be imperceptible \textemdash appears to be the result of the persistent effect of molecular and macromolecular forces tending to displace the soil particles.", and he describes the soil flow as a quasi-viscous fluid.\\

The description of the geological processes that occur in landscape evolution was initiated in \cite{gilbert_report_1877}. In this book, the author gives the fundamental principles of landscape evolution and explains  why the profiles of stream beds are concave upwards. This concavity property had been illustrated later in \cite{culling1960analytical}, with a mathematical description of the erosion of a slope. The slope evolution is described with a reaction diffusion equation, the exponential source term representing the erosion rate. The landscapes evolve mainly in function of these two competitive factors: creep on the hilltops, and stream incision on the lower slopes. Indeed, in the upper slopes the water flow is weak and dispersive, thus the creep effect predominates. Downward, the effect of shear stress of the flow becomes dominant, and bedload transport process increases. This leads to the formation of stream beds and valley, the profiles being concave. Since then, the complexity of the landscape evolution models increased: see \cite{chen_equations_2014} and \cite{chen_landscape_2014} for a review of these models.

The process of erosion and transport of sediments can be described by two different laws, the alluvial transport law or the stream incision law, which depends on the nature of the soil. When the soil is covered by alluviums, the sediments are directly transported by water flux, and the amount of sediments moved by the water flux in a given time $q_s$ follows a law that depends on the water discharge $q$. 
This transport discharge law is described, for example, in \cite{smith_stability_1972} and is given by 
\begin{align}
     q_s = k q^m |\nabla z|^n
\end{align}
where $k$, $m$, $n$ are constants. This is called the transport limited case.

On the other side, the stream incision law is used when the surface is bedrock, because in this case the sediment transport is limited by the resistance of the bedrock to the shear stress caused by the water flux. This is modeled by an erosion source term in the equation of surface height evolution, as described in \cite{howard_channel_1983}. The sediments removed from the soil by the erosion are supposed to be dissolved in water, thus are transported by the water flux. This stream incision law is described below, in Section~\ref{sec_model}. This case, that we study in this paper is called the detachment limited case.\\

In this paper, we study a system of partial differential equations describing the erosion of the soil by water, which had been proposed in \cite{chen_equations_2014}, Section 4.
It includes the modeling of the water flow, the erosion of the surface by water, the transport and deposition of sediments, and the creep effect. The characteristic fluid velocity is usually much larger than the erosion rate: in order to describe pattern formation, we only consider large scale fluctuations of the fluid velocity. As a result, we assume that the fluid velocity is proportional to the gradient of the free surface elevation.The other physical principles taken in account are:
\begin{itemize}
     \item The conservation law for water and sediments dissolved in water,
     \item The stream incision law: the erosion grows with the water speed and the water height,
     \item The sedimentation rate is proportional to the concentration of sediments in water,
     \item The creep effect: the soil is subject to a diffusion process.
\end{itemize}
Other effects, such as infiltration, vegetation, wind, or ice formation are neglected.  Moreover, we suppose that the bottom surface is constituted by one type of sediments.
The system studied in this article models the time evolution of the soil and of the fluid. This system is composed by the following three partial differential equations:
\begin{equation}
\left\{ \begin{array}{lll}
\displaystyle
\dt h - \dive(h \nabla(h+z)) = r(t,x),\vspace{2mm}\\
\displaystyle
\dt z = K \Delta z + sc - e h^m |v|^n,\vspace{2mm}\\
\displaystyle
\dt (ch) + \dive(chv) = e h^m |v|^n - sc.
\end{array} \right.
\label{syst-intro}
\end{equation}

The first equation is a mass conservation law for the fluid and describes the evolution of the fluid height, denoted by $h$, and the fluid is transported at speed $v := -\nabla(h+z)$. The term $r$ is a source term representing an exterior source of water, like the rain. 
The second equation models the time evolution of the bottom topography (denoted by $z$). The constant $K$ is the constant of creep whereas $e$ and $s$ are respectively the constants for the incision law and the sedimentation rate. As in \cite{howard_channel_1983}, we will suppose that the incision law depends on a power of the norm of the water velocity, and on a power of the water height. Thanks to this hypothesis, as long as the water height and velocity do not vanish, the erosion rate is positive. However, in practice the erosion starts if the water velocity is high enough to break the cohesion of the soil. Thus a threshold effect could be introduced in the model. This effect could be easily added to the numerical scheme of the model, but it would increase significantly the difficulty of its mathematical study. Therefore we choose to ignore the threshold effect in the erosion process. This approximation is justified in the regime that we study, because the water height and water velocity are not close to zero.
The last equation is the conservation equation for the sediments dissolved in the fluid. This concentration, denoted by $c$ is the average of the concentration over the height of the fluid, thus is given in gram per square meter. The right hand side is the source term, which represents the exchange of sediments between the ground and the fluid, caused by the erosion and the sedimentation. 
The model is set in two dimensions, the variables are the time $t \in \RR^+$ and the position $(x,y)$, which belongs to a domain $\Omega \subset \RR^2$. \\

 In \cite{chen_equations_2014}, the authors reviewed various landscape evolution models and focused on system~\eqref{syst-intro}. 
They listed  some open mathematical problems like local existence in time, regularity of solutions or stability. They performed a numerical study on the erosion of a gaussian shaped hill to demonstrate the ability of the model to exhibit pattern formation and studied the influence of some parameters on the complexity of these patterns. In \cite{lebrun_numerical_2018}, the authors performed numerical simulations in a more practical context where the initial topography is a real one like the ones found in La Reunion or Madeira islands: they took a particular care of the visualization of landscape evolution through a particular colorization process. The simulations show the capacity of the model to describe a realistic evolution in time of the rivers and gullies on these landscapes, provided its parameters are properly chosen.
Note that the detachment-limit hypothesis is valid for bedrock rivers and must be appropriately modified for alluvial rivers, where sediments form a layer of loose material. In this latter case, other classes of models may be considered like the shallow water equations coupled with Exner equations for the transport of sediments: see  \cite{fernandez2017formal}  for a formal derivation of these models from bi-layer type models (one layer for the fluid and one layer for the sediment) and \cite{fernandez2014influence, escalante2021modelling} for various numerical strategies to perform simulations of bedload sediment transport.
The aim of this paper is to explore the mechanism of pattern formation in the soil caused by water flow, and to precise the form and the frequency of apparition of these patterns. In landscapes, these patterns are channels, bed rivers, valley. 
We will focus on a simple framework where the initial bottom surface is a tilted plane. This surface is covered by a layer of water, the water flow from the top of the plane. The situation is represented in Figure~\ref{experience}, where the tilted plane is seen from the side.

\begin{figure}[ht]
     \begin{tikzpicture}[scale = 0.87]
     \draw[black, ->] (0,0) -- (14.6,0);
     \draw[black] (14.6,-0.3) node {$x$};
     
     \draw[black, ->] (0,0) -- (0,4.5);
     \draw[black] (-0.3,4.3) node {$z$};
     
     \draw[dashed, ->] (0,0) -- (2,1);
     \draw[black] (1.8,1.2) node {$y$};
     
     \draw[blue] (0,4) -- (13.7,0.4);
     \draw[blue] (0,4) -- (3,5.5);
     \draw[blue] (13.7,0.4) -- (16.7,1.9);
     \draw[blue] (3,5.5) -- (16.7,1.9);
     
     \draw[brown] (0,3.6) -- (13.7,0);
     \draw[brown, dashed] (0,3.6) -- (3,5.1);
     \draw[brown, dashed] (13.7,0) -- (16.7,1.5);
     \draw[brown, dashed] (3,5.1) -- (16.7,1.5);
     
     \draw[red, <->] (14,0) -- (16.9,1.45);
     \draw[red] (16,0.5) node {$L_y$};
     
     \draw[red, <->] (0,-0.7) -- (13.7,-0.7);
     \draw[red] (7,-1) node {$L_x$};
     
     \draw[red, ->] (12.5,0) -- (12.55,0.258);
     \draw[red] (12,0.19) node {$\theta$};
          
     \draw[dashed] (5, 0.4) -- (5, 3.2);
     \draw (4.5,0.5) node {$(x,y)$};
     \draw[brown] (5.8, 2.8) node {$z(x,y)$};
     \draw[blue] (5.3, 3.6) node {$(z+h)(x,y)$};
     \fill (5,0.4) circle[radius=1pt];
     \fill[blue] (5,3.2) circle[radius=1pt];
     \fill[brown] (5,2.8) circle[radius=1pt];
     
     \draw[red, ->] (12.5,0) -- (12.55,0.258);
     \draw[red] (12,0.19) node {$\theta$};
     
     \draw[blue, ->] (0,3.8) -- (0.6,3.65);
     \draw[blue, ->] (7.7,1.75) -- (8.3,1.6);
     
     \end{tikzpicture}
     \caption{Sketch of the mathematical framework for stability analysis: the water is flowing on a tilted plane.}
     \label{experience}
\end{figure}
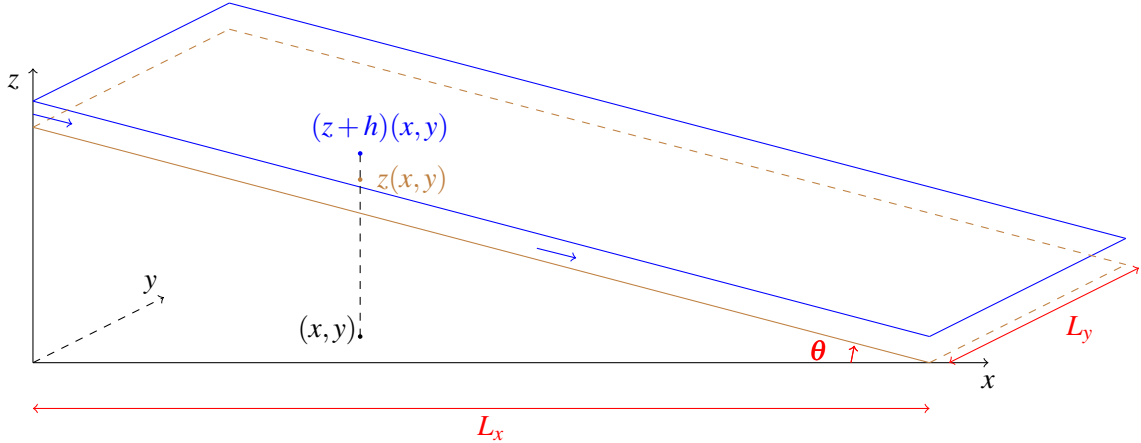

Although it is a less realistic framework than the landscape evolution described in \cite{lebrun_numerical_2018}  where real initial topographies are considered, it is sufficient to study the emergence of patterns on the surface. Note that experiments of erosion on a flat surfaces were conducted in \cite{guerin_streamwise_2020} with a thin film flow which erodes  blocks of plaster or salt. \\

The main outcomes of this paper are the following:
\begin{itemize}
\item We establish the well-posedness of System~\eqref{syst-intro} when the initial fluid height does not vanish.
\item We study the spectral stability of constant states when the bottom is an inclined plane and characterize rigorously the spectral instability through a non dimensional number, so called channelization index. 
\item We perform direct numerical simulations of System~\eqref{syst-intro} for realistic experimental data and on time scales associated to the erosion phenomenon. We tackle the problem of severe CFL restrictions due to very short time scales associated to the fluid flow.
\end{itemize}

We provide thereafter a short description of our main results.
First, we prove the well posedness character of System~\eqref{syst-intro} locally in time under some assumptions on the data and the parameters. We prove the existence of solutions in suitable Sobolev spaces. In order to simplify notations, we will denote for any $T>0$ and $\mathcal{H}$ a Hilbert space
$$
\displaystyle
L^2_T(\mathcal{H}):=L^2((0, T); \mathcal{H});\quad C_T(\mathcal{H}):=C((0, T);\mathcal{H})
$$
the space of functions $u:t\mapsto u(t,.)\in\mathcal{H}$ that are respectively $L^2$ integrable or continuous in time on the interval $(0, T)$. We prove the following result:
\begin{theorem}
\label{th}
Let $m>0$, $n>3$ or $n=2$, $K>0$, and $T_0>0$. Let us fix two constants fluid heights $h_{ref}>h_{min}>0$. Suppose that the initial data $h^0$, $z^0$, $c^0$ satisfy
$$h^0-h_{ref}, \, z^0 \in H^{k+1}(\RR^2), \quad c^0 \in \hk,\quad h^0(x) \geq 2\,h_{min} \, \forall x \in \RR^2.$$
with $k=3$. Suppose that $r \in L^2_{T_0}(\hk) $, $K\,h_{min} - ||h^0||_{L^{\infty}}^2 \geq  0$.
Then there exists $0 < T < T_0$ such that System~\eqref{syst-intro} admits a unique solution $(h,z,c)$ with 
$$h-h_{ref}, \, z \in \lhkde \cap C_T(\hkun), \quad c \in C_T(\hk).$$
\end{theorem}

The proof of Theorem~\ref{th} is based on energy estimates and on a fixed point argument. The problem of the well posedness of this system was raised in the conclusion of \cite{chen_equations_2014}. We prove that it is well-posed locally in time, and under the assumption that the fluid height does not vanish. With this assumption, the equations on water height and water concentration are parabolic, and this property is used in the proof, for the energy estimates.This validates the model in the situation where the fluid height is close to a constant, hence, far from vanishing. This will be the case in our stability study, and our numerical simulations. The case of vanishing fluid height is more involved as the model becomes degenerated parabolic and the regularizing effect on the fluid height, and thus on the fluid velocity, is lost: one has to consider weaker classes of solutions. More precisely, dropping the $z$ dependence, the equation on $h$ reads
$$
\displaystyle
\partial_t h=h\Delta h+|\nabla h|^2.
$$
In \cite{brandle2005viscosity}, the authors prove the well posedness for this equation within the class of bounded, continuous and non negative viscosity solutions.
This is beyond the scope of this paper to couple this type of solutions with the other equations of the model, and we leave this problem open for future studies. An other interesting question is the numerical simulation of the model in the presence of dry areas: this will be carried out in a forthcoming work.

The second part of the paper is dedicated to pattern formation when the initial topography is an inclined plane and the bottom surface is weakly eroded. Our aim is to identify instability mechanisms that could explain the formation of patterns. For that purpose, we linearize System~\eqref{syst-intro} around a constant state and we study the conditions of spectral instability and the nature of the instabilities. We expect that when the system is spectrally unstable for some parameters and wave vectors, a slight perturbation of the system with these unstable modes will grow up and lead to the formation of patterns in the soil.

Stability studies have been done previously, for other landscape evolution models. The papers \cite{smith_stability_1972} and \cite{loewenherz_stability_1991} analyse a model of two equations, where the water is supposed to be at equilibrium. This model has steady solutions for which the soil height can be concave in some areas and convex in other areas, depending on the sediment discharge law. They show that the linearised system is stable in the convex parts and unstable in the concave parts, with a stronger instability in the transverse direction. In their model they use the sediment transport law, thus its not the same framework as in the model~\eqref{syst-intro}.

Note that in a couple of recent papers \cite{anand2020linear} and \cite{bonetti2020channelization}, the following system of 2 PDEs was considered:
\begin{equation}\label{eq-zh}
\displaystyle
\partial_t z=K\Delta z-e(\frac{h}{H})^m|\nabla z|^n+U,\quad \partial_t h=\dive(hv_0\frac{\nabla z}{|\nabla z|})+R,
\end{equation}

where $R,U$ and $V_0$ are constants. A numerical scheme is designed in \cite{anand2020linear} for~\eqref{eq-zh} where the time derivative of the fluid height is neglected with test cases where the initial bottom topography is pyramidal. It is found that a channelization index 
$$\displaystyle \mathcal{C}_I=\frac{e\ell^{m+n}}{K^nU^{1-n}}$$
drives the formation of channels: the number of channels and their branching increase with $\mathcal{C}_I$. This analysis is completed by a spectral stability analysis of a spatially non homogeneous steady state where the topography is a hillslope which is divided in the middle. It is found numerically that there exists a critical $\mathcal{C}_I^0$ such that the steady state is stable if  $\mathcal{C}_I\leq \mathcal{C}_I^0$ and unstable otherwise.\\

The spectral study carried out in this paper is new and our analysis provides some explanations for the formation of patterns in landscapes. The appearance of channels on the flat plane is indeed the initial stage of development for the formation of valley and rivers in landscapes.  
The stability of the system depends on the parameters, in particular the constant of creep $K$ plays an important role in this study. We show that there is a critical value $\bar K$ such that if $K \geq \bar K$, and if another condition on parameters is satisfied, then the system is spectrally stable at all frequencies. If $K < \bar K$ then there exist some wave numbers and vectors for which the system is spectrally unstable. Moreover, the instabilities grow as the wave vectors of the perturbations points in the direction transverse to the flow, which explains the formation of gullies and channels aligned with the direction of the fluid flow. We then recover qualitatively the results of \cite{bonetti2020channelization}.\\

Finally, our stability study is completed by direct numerical simulations, which illustrate the appearance of patterns for the nonlinear system. 
The space and time scales of the model can take a large range of values, depending of the environment. On real landscapes, the domain size can be measured in kilometers, with a very slow erosion rate, in the order of $20-200$ millimeters per thousand year. In the experience on salt and plaster made in \cite{guerin_streamwise_2020}, the domain has a size of the order of ten centimeters whereas the erosion speed is around one millimeter per hour (so, much faster than in real landscape) and the fluid velocity is 1 meter per second. Parameters chosen in the numerical simulations are based on these experiments.
As in \cite{anand2020linear, bonetti2020channelization}, we have observed that decreasing $K$ (respectively increasing the channelization index $\mathcal{C}_I$) reinforces the channelization process. Note that we have focused here on the formation of channels: unlike simulations made in \cite{lebrun_numerical_2018} where the initial state is a matured landscape, we start from a simple state and the landscape evolves by himself in the simulations. \\


The paper is organised as follow: 
First, in Section~\ref{sec_model}, we describe the model and the associated system of equations. Then, Section~\ref{WP} is devoted to the proof of the well posedness character of the system in short time, (see Theorem~\ref{th}).
Next, in Section~\ref{stability} we carry out a spectral stability analysis of the  System~\eqref{syst-intro} linearized about a stationary solution. These stability results are compared to direct numerical simulations of the nonlinear system~\eqref{syst-intro} in Section~\ref{sec_cclstab}. Finally, Section~\ref{sec_ccl} draws a brief conclusion of the paper and provides some future perspectives.

\section{The landscape evolution model} \label{sec_model}

In this section, we introduce the landscape evolution model considered in this paper. This is a system of three partial differential equations for the fluid height $h$, the bottom topography $z$ and the sediment concentration $c$.

\paragraph{Evolution of topography. } \label{sec_evol_topo}

The evolution of the bottom topography $z$ is given by:
\begin{equation*}
	\dt z = K \Delta z - E + S.
\end{equation*}

The functions $E = E(t,x,y)$ and $S = S(t,x,y)$ represent the erosion speed of the soil and the sedimentation speed respectively, with $t \geq 0$, $(x,y) \in \RR^2$. \\

The parameter $K > 0$ is a constant, and the term $K \Delta z$ models the creep of the soil. This phenomenon is a slow diffusive movement of the soil which occurs at large time and space scales. This movement is caused by several processes, such as the gravitational flow of the soil, wind, rain splash, expansions and contractions of the soil due to freeze-thaw, wet-dry and hot-cold cycles, or biological activity. In sufficiently eroded landscape there are generally not many sharp edges, and this creep term, which tends to smooth the bottom surface models this phenomenon. We shall see that this term plays a significant role in the well-posedness of the model. However, the creep effect is not supposed to be relevant in the formation of patterns as it is a short time effect and should be supposed to be small in comparison to the erosion and sedimentation terms.\\ 

The erosion of the surface is caused by the shear stress and the friction of the water flow. Assuming that the fluid velocity is constant across the fluid layer, this amounts to consider that the erosion increases with the (norm of the) water velocity, and with the water discharge $Q$, as in \cite{howard_channel_1983}:
$$\displaystyle E(t,x,y) = \alpha Q^\mu v^\nu.$$
Consequently, as $Q = hv$, we will suppose that the erosion speed depends on a power of the norm of the water velocity, and of a power of the water height. Thus we set 
$$
\displaystyle
E(t,x,y)= e \left( \frac{h(t,x,y)}{H} \right)^m \left( \frac{|v(t,x,y)|}{V} \right)^n,
$$

where $e$ is the erosion speed in the conditions $h=H$ and $|v|=V$ with $H,V>0$ that respectively represent a reference fluid height and fluid velocity. \\

The sedimentation occurs when the concentration of sediments in water is high enough. The sedimentation speed increases with the concentration of sediment in the fluid. For the sake of simplicity, we  suppose that this speed is proportional to the concentration and we set: 
$$
\displaystyle
S(t,x,y) = s \, \frac{c(t,x,y)}{c_{sat}},
$$
with $s$ the speed of sedimentation in the reference condition $c=c_{sat}$. Therefore, the soil elevation evolves according to the equation :
\begin{equation}\label{evol-z}
	\dt z = K \Delta z - e \left( \frac{h(t,x,y)}{H} \right)^m \left( \frac{|v(t,x,y)|}{V} \right)^n + s \, \frac{c(t,x,y)}{c_{sat}}.
\end{equation}

\paragraph{The landscape evolution model.}

We complete Equation~\eqref{evol-z} with two evolution equations for the fluid height $h$ and sediment concentration $c$. The mass conservation law for the fluid reads
\begin{equation}\label{evol-h}
\displaystyle
\dt h+\dive(hv)=r,
\end{equation}

where $r$ is a source term, modeling an incoming flow in a channel or the rain over the bottom.\smallskip

On the other hand, the mass conservation law for the sediment reads
\begin{equation}\label{evol-c}
\dt (hc)+\dive(chv)=\rho_s(E-S),    
\end{equation}
where $\rho_s$ is the volumetric mass density of the sediments, and is constant. In order to close System~\eqref{evol-z},~\eqref{evol-h},~\eqref{evol-c}, we need to write an equation for the fluid velocity. One possibility would be to write a shallow water type model with an evolution equation for the momentum $hv$. We rather choose the simpler closure 
\begin{equation}\label{evol-v}
\displaystyle
v=-\mu\nabla (h+z),
\end{equation}

where $\mu>0$ is some characteristic fluid velocity and $\nabla(h+z)$ is the gradient of the fluid surface elevation. System~\eqref{evol-z},~\eqref{evol-h},~\eqref{evol-c},~\eqref{evol-v} is closed and we shall consider its well-posedness in Section~\ref{WP}.\smallskip

We are also interested in the pattern formation at the surface of the soil. For that purpose, we have chosen to explore the case of water flowing down an inclined plane. This situation was considered experimentally in \cite{guerin_streamwise_2020}. The domain $\Omega\subset\mathbb{R}^2$ has length $L_x$ and width $L_y$: $\Omega = [0,L_x] \times [0,L_y]$. Denote $\theta$ the inclination of the plane. We can decompose the bottom topography $z$ as $z(t,x,y) = (L_{x}-x) \tan \theta + \tilde{z}(t,x,y)$ where $\tilde{z}$ is the eroded height of the soil. Thus the fluid velocity is written as $v(t,x,y) = \mu (\tan \theta,0) - \mu \nabla (\tilde{z}+h)$.  Consequently, omitting the tilde over $z$,
System~\eqref{evol-z},~\eqref{evol-h},~\eqref{evol-c},~\eqref{evol-v} admits the new form:
\begin{subequations}
\begin{empheq}[left= \empheqlbrace]{align}
& \displaystyle
\dt h + \mu \tan \theta \dx h = \mu \dive (h \nabla (h+z)), \vspace{2mm} \label{eq_systcomplet_h} \\
&\displaystyle
h \dt c + \mu h \tan \theta \dx c = \mu h \nabla (h+z).\nabla c + \rho_s \, e \left( \frac{h(t,x)}{H} \right)^m \left( \frac{|v(t,x)|}{V} \right)^n - \rho_s \, s \, \frac{c(t,x)}{c_{sat}},\vspace{2mm} \label{eq_systcomplet_c} \\
&\displaystyle
\dt z = K \Delta z - e \left( \frac{h(t,x)}{H} \right)^m \left( \frac{|v(t,x)|}{V} \right)^n + s \, \frac{c(t,x)}{c_{sat}}. \label{eq_systcomplet_z} 
\end{empheq}
\label{syst_complet}
\end{subequations}

\section{Well-posedness of the landscape evolution model} \label{WP}

In this section we study the existence and uniqueness of solutions of the system~\eqref{evol-z},~\eqref{evol-h},~\eqref{evol-c},~\eqref{evol-v}, locally in time.

\subsection{Hypothesis on the system of equations} \label{WP_syst}


We consider System~\eqref{evol-z},~\eqref{evol-h},~\eqref{evol-c},~\eqref{evol-v} where we set, for simplicity, $\mu = 1$, $\rho_s = 1, c_{sat} = 1$ and $H = V = 1$. The choice of these constants does not change anything in the proof Theorem~\ref{th}, we fix them to simplify the notations. Provided that $\forall (t,x) \in \RR^+ \times \RR^2$, $h(t,x) \geq h_{min} > 0$, the equations on $h$ and on $z$ are parabolic equations.
As long as $h$ does not vanish, the equation on $c$ can be written as 
$$
\dt c + v.\nabla c = e h^{m-1} |v|^m - sc/h - rc/h.
$$

Therefore, in order to prove Theorem~\ref{th}, we consider the equivalent following system, composed of two parabolic equations and one transport equation:
\begin{subequations}
\begin{empheq}[left= \empheqlbrace]{align}
\hspace{1mm} & \displaystyle\dt h -\dive(h\nabla h) - \dive(h \nabla z) = r(t,x), \label{eq_h} \\[0.2em]
& \displaystyle\dt z = K \Delta z + sc - e h^m |v|^n, \label{eq_z} \\[0.2em]
& \displaystyle\dt c + v.\nabla c = e h^{m-1} |v|^n - sc/h - rc/h, \label{eq_c} 
\end{empheq}
\label{eq_WPsyst}
\end{subequations}
with the initial conditions $h(0,x) = h^0(x)$, $z(0,x) = z^0(x)$ and $c(0,x) = c^0(x)$, and where $t \in \RR^+$, $x \in \RR^2$.\\

In Subsection~\ref{useful properties}, we recall some results concerning Sobolev spaces that will be used to prove Theorem~\ref{th}. The well posedness of System~\eqref{eq_WPsyst} is proved in Subsections~\ref{a priori estimate}~-~\ref{uniqueness} by using a fixed point argument and energy inequalities. We first establish an a priori estimate on the solutions of~\eqref{eq_WPsyst} in Subsection~\ref{a priori estimate}. Then, we build a sequence of approximate solutions in
Subsection~\ref{approximate system} and provide uniform estimates on these solutions. We show that it forms a Cauchy sequence and converges to a solution of System~\eqref{eq_WPsyst}. Finally, we prove the uniqueness of  solutions in Subsection~\ref{uniqueness}. In these sections, we denote by $\nabla^p f$ the vector made by the partial derivatives of order $p$ of the function $f$. We also denote $(\nabla f)^2 = (\partial_{x_i} f \, \partial_{x_j} f, 1 \leq i \leq 2, 1 \leq j \leq 2)$.

\subsection{Sobolev injections} \label{useful properties}

In this section we recall some properties of Sobolev spaces, that will be used in the next sections. The proofs of these properties can be found in \cite{adams_sobolev_2003}, and in  \cite{evans_partial_2010} for Proposition ~\ref{prop_Sob_reg_derivative}.
The first proposition concerns the imbedding of the Hilbert space $H^{k}(\RR^2)$ into Sobolev spaces with smaller derivation index. The symbol $\hookrightarrow$ indicates that the injection is continuous.
\begin{proposition}
If $2 < q < +\infty$ and $k \in \NN$ then $H^{k+1}(\RR^2) \hookrightarrow W^{k,q}(\RR^2)$.
\end{proposition}

The following proposition gives the imbedding of the Hilbert space $H^k(\RR^2)$ into a space of smooth functions. 
\begin{proposition}
If $k \geq 2$ then $H^{k}(\RR^2) \hookrightarrow C^{k-2, \alpha}(\RR^2), \, \forall \, 0<\alpha <1$. In particular, $H^{2}(\RR^2) \subset C^0_b(\RR^2)$.
\end{proposition}

\begin{proposition}
   If $k>1$ then  $H^{k}(\RR^2) \hookrightarrow W^{k-2,\infty}(\RR^2)$. 
\end{proposition}

Finally, Proposition ~\ref{prop_algebra} provides a bound on the norm of a power of functions, that will be useful to control non linear terms in the equations.

\begin{proposition}
$H^k(\RR^2)$ is an algebra for $k > 1$. Thus if $q\in\mathbb{N} ^*$ then $||f^q||_{H^{k}(\RR^2)} \leq C ||f||_{H^k(\RR^2)}^q$. If $q\in\RR\setminus\mathbb{N}$ and $q>k$, one has $||f^q||_{H^{k}(\RR^2)}\leq C ||f||_{H^k(\RR^2)}^q$.
\label{prop_algebra}
\end{proposition}

Finally, Proposition ~\ref{fdtf} concerns spaces involving time, and gives the continuity in time of a function provided this function and its derivative in time have enough regularity.
\begin{proposition}
\label{fdtf}
If $f \in L^2_T(H^{k+2}(\RR^2))$ and $\dt f \in L^2_T(H^{k}(\RR^2))$ then $f \in C^0_T(H^{k+1}(\RR^2))$.
\label{prop_Sob_reg_derivative}
\end{proposition}

In what follows, we will denote $H^{k}(\RR^2)=H^k$ for the sake of simplicity.

\subsection{An a priori estimate} \label{a priori estimate}

In this section we give an a priori estimate on solutions of System~\eqref{eq_WPsyst}, which will be used in the Section~\ref{convergence}. We fix $k=3$, and assume $m>0$, $n>3$ or $n=2$ as in the hypothesis of Theorem~\ref{th}.

\begin{proposition}
Let $(h-h_{ref},z,c) \in \left(\lhkde \cap C_T(H^{k+1}) \right)^2 \times C_T(H^k)$ be a solution of System~\eqref{eq_WPsyst}. Assume that the hypothesis of Theorem~\ref{th} are satisfied. Then there exists $T_1 \leq T$ such that $h-h_{ref}$ and $z$ are bounded in $C_{T_1}(H^{k+1})\cap L^2_{T_1}(H^{k+2})$, $c$ is bounded in $C_{T_1}(H^k)$, and they satisfy the estimate:
$$
\displaystyle
\mathcal{E}(t)\leq e^{Ct}\left(\mathcal{E}(0)+C\int_0^t\|r\|^2_{H^k}\right),\quad
\displaystyle
\int_0^t \left( \|\nabla h\|_{H^{k+1}}^2+\|\nabla z\|_{H^{k+1}}^2 \right) \leq C\,e^{Ct}\left(\mathcal{E}(0)+C\int_0^t\|r\|^2_{H^k}\right),
$$
where $C$ is a constant depending on $h^0, z^0, c^0, h_{min}$ and
$\mathcal{E}$ is defined as:
$$
\displaystyle
\mathcal{E}(t)=\frac{1}{2}\left(\|h-h_{ref}\|^2+\|\nabla h\|_{H^k}^2+\|z\|_{H^{k+1}}^2+\|c\|_{H^k}^2\right).
$$
\label{prop_apriori_estim}
\end{proposition}

\begin{proof}
We first provide Sobolev estimates on the fluid height $h$. We multiply the equation~\eqref{eq_h} by $h-h_{ref}$ and integrate it over $\RR^2$. One obtains: 
\begin{equation}\label{h1}
\displaystyle
 \frac{1}{2} \frac{d}{dt} \|h-h_{ref}\|^2_{L^2} =
 \int_{\RR^2} (h-h_{ref}) \dive \left(h \nabla (h+z) \right)+\int_{\RR^2} (h-h_{ref}) r.
\end{equation}
By integrating by part~\eqref{h1}, and under the assumption that $h \geq h_{min}$, we obtain for any $t\geq 0$ fixed:
\begin{equation}\label{h11}
\frac{1}{2} \frac{d}{dt} \|h-h_{ref}\|^2_{L^2} + h_{min} \|\nabla h\|^2_{L^2}  \leq \|r\|_{L^ 2} \|h-h_{ref}\|_{L^2} + \|h\|_{L^\infty} \|\nabla z\|_{L^2}  \|\nabla h\|_{L^2}.
\end{equation}

Then for all $p \in \{1, \dots, k+1 \}$, we differentiate $p$ times Equation~\eqref{eq_h}, multiply it by $\nabla^p h$ and integrate it with respect to the space variable:
\begin{equation}
\displaystyle
\frac{1}{2} \frac{d}{dt} \|\nabla^p h\|^2_{L^2} + \int_{\RR^2} \nabla^{p-1} \left[ \dive(h\nabla (h+z)) \right] \nabla^{p+1} h
= -\int_{\RR^2} \nabla^{p-1} r \, \nabla^{p+1} h
 \label{h2}
\end{equation}
We estimate the second term in~\eqref{h2}. As $H^1 \hookrightarrow L^{4}$ and $H^2 \hookrightarrow L^{\infty}$, we find:
{\setlength\arraycolsep{1pt}
\begin{eqnarray*}
\displaystyle
\int_{\RR^2} \nabla^{p-1} \left[ \dive(h\nabla h) \right] \nabla^{p+1} h 
&=& \int_{\RR^2} h |\nabla^{p+1} h|^2+ \int_{\RR^2} \sum\limits_{i=1}^{p-1} \binom{p}{i} \nabla^i h \nabla^{p-i+1} h \nabla^{p+1} h + \int_{\RR^2} \nabla h \nabla^p h \nabla^{p+1} h \\ 
\displaystyle
&\geq& h_{min} ||\nabla^{p+1} h||_{L^2}^2-C\|\nabla h\|_{H^{p}}^2\|\nabla h\|_{H^{k+1}}.
\end{eqnarray*}
}
For the third term in~\eqref{h2}, we proceed similarly: 
\begin{align*}
\displaystyle
\left|\int_{\RR^2} \nabla^{p-1} \dive(h \nabla z) \nabla^{p+1} h\right| &= \left|\int_{\RR^2} h \nabla^{p+1} z \nabla^{p+1} h + \int_{\RR^2} \sum\limits_{i=1}^{p} \binom{p}{i} \nabla^i h \nabla^{p-i+1} z \nabla^{p+1} h\right| \\
\displaystyle
& \leq ||h||_{L^\infty} ||\nabla z||_{H^{k+1}} ||\nabla^{p+1} h||_{L^2} + C ||\nabla z||_{H^{p}}||\nabla h||_{H^{p}} ||\nabla h||_{H^{k+1}}.
\end{align*}
By inserting these two estimates into~\eqref{h2}, one obtains:
{\setlength\arraycolsep{1pt}
\begin{eqnarray}
\displaystyle
\frac{1}{2} \frac{d}{dt} \|\nabla h\|_{H^{k}}^2 + h_{min} \|\nabla h\|_{H^{k+1}}^2 
&\leq & \|r\|_{H^k} \|\nabla h\|_{H^{k+1}}+ \|h\|_{L^\infty} \|\nabla z\|_{H^{k+1}} \|\nabla h\|_{H^{k+1}}\nonumber\\
&&+ C \left(\|\nabla z\|_{H^{k}}+ \|\nabla h\|_{H^{k}}\right)^2\|\nabla h\|_{H^{k+1}}.
\label{eq_in_energy_h}
\end{eqnarray}
}

Next, we derive an estimate on the bottom topography $z$. We multiply equation~\eqref{eq_z} by $z$ and integrate over space $\mathbb{R}^2$. We get:
{\setlength\arraycolsep{1pt}
\begin{eqnarray}
\displaystyle
\frac{1}{2} \frac{d}{dt} ||z||^2_{L^2} + K ||\nabla z||^2_{L^2} &=& s \int_{\RR^2} c \, z - e \int_{\RR^2} h^m |v|^n z\leq s ||c||_{L^2} ||z||_{L^2} + e ||h||_{L^{\infty}}^m ||v||_{L^{2n}}^n ||z||_{L^2} \vspace{3mm}\nonumber\\
\displaystyle
&\leq & s ||c||_{L^2} ||z||_{L^2} + C e\|h\|_{L^\infty}^m \left( ||\nabla h||_{H^1} + ||\nabla z||_{H^1} \right)^n ||z||_{L^2}.
\label{z0}
\end{eqnarray}
}
The last inequality is a consequence of the injection $H^1 \hookrightarrow L^q$, with $q = 2n \geq 2$. Let us now estimate the derivatives of $z$ of order $p \in \{1, \dots, k+1 \}$. We differentiate Equation~\eqref{eq_z} $p$ times and multiply by $\nabla^p z$. 
{\setlength\arraycolsep{1pt}
\begin{eqnarray}
\displaystyle \frac{1}{2} \frac{d}{dt} \|\nabla^p z\|^2_{L^2} + K \|\nabla^{p+1} z\|^2_{L^2} &=& s \int_{\RR^2} \nabla^p c \nabla^p z - e \int_{\RR^2} \nabla^p \left( h^m |v|^n \right) \nabla^p z\nonumber\\
\displaystyle 
&=& -s \int_{\RR^2} \nabla^{p-1} c \nabla^{p+1} z + e \int_{\RR^2} \nabla^{p-1} \left( h^m |v|^n \right) \nabla^{p+1} z\nonumber\\
\displaystyle
&\leq& s ||\nabla^{p-1} c||_{L^2} ||\nabla^{p+1} z||_{L^2} + e ||\nabla^{p-1} (h^m |v|^n)||_{L^2} ||\nabla^{p+1} z||_{L^2}.\label{z1}
\end{eqnarray}
}

By adding the estimates~\eqref{z0} and ~\eqref{z1} for $p={1,\dots,k+1}$, we obtain:

\begin{equation}
\frac{1}{2} \frac{d}{dt} \|\nabla z\|_{H^{k}}^2 + K \|\nabla z\|_{H^{k+1}}^2 
\leq s\, \|c\|_{H^k} \|\nabla z\|_{H^{k+1}} + e \|h^m |v|^n\|_{H^k} \|\nabla z\|_{H^{k+1}}.
\label{z2}
\end{equation}

There remains to estimate the erosion term $h^m|v|^n$ in $H^k$ norm. We assumed that $n>3$ or $n=2$, which implies that $v\mapsto |v|^n$ is $\mathcal{C}^3(\RR^2; \RR)$. We estimate successively $\nabla^p(h^m|v|^n)$ for $p=1,2,3=k$. By using successively the injections $L^q\hookrightarrow H^1$ for $q\geq 2$ and $H^2\hookrightarrow L^\infty$,  one finds:
$$
\begin{array}{lll}
\displaystyle
\|\nabla(h^m |v|^n)\|_{L^2}\leq \|h^{m-1}\|_{L^\infty}\left(\|\nabla h\|_{H^2}+\|\nabla z\|_{H^2}\right)^n(\|h\|_{L^\infty}+\|\nabla h\|_{L^2}), \vspace{3mm}\\
\displaystyle
\|\nabla^2(h^m |v|^n)\|_{L^2}\leq \|h^{m-2}\|_{L^\infty}\left(\|\nabla h\|_{H^2}+\|\nabla z\|_{H^2}\right)^n(\|h\|_{L^\infty}+\|\nabla h\|_{H^1})^2, \vspace{3mm}\\
\displaystyle
\|\nabla^3(h^m |v|^n)\|_{L^2}\leq \|h^{m-3}\|_{L^\infty}\left(\|\nabla h\|_{H^3}+\|\nabla z\|_{H^3}\right)^n(\|h\|_{L^\infty}+\|\nabla h\|_{H^1})^2\left(\|h\|_{L^\infty}+\|\nabla h\|_{H^3}+\|\nabla z\|_{H^3}\right).
\end{array}
$$




We deduce the following estimate on the bottom topography:
{\setlength\arraycolsep{1pt}
\begin{eqnarray}
\displaystyle
\frac{1}{2} \frac{d}{dt} \|\nabla z\|_{H^{k}}^2 + K \|\nabla z\|_{H^{k+1}}^2&\leq& s \|c\|_{H^k} \|\nabla z\|_{H^{k+1}}\nonumber\\ 
\displaystyle
&&+C\|h^{m-3}\|_{L^\infty}\left(\|h\|_{L^\infty}+\|\nabla h\|_{H^3}+\|\nabla z\|_{H^3}\right)^{n+3}\|\nabla z\|_{H^{k+1}}.
\label{eq_in_energy_z}
\end{eqnarray}
}

Finally, we derive a priori estimate for the sediment concentration $c$: the basic energy estimates reads:
\begin{equation}
\displaystyle    
\frac{1}{2}\frac{d}{dt}\|c\|_{L^2}^2\leq \left(\|\nabla h\|_{H^3}+\|\nabla z\|_{H^3}\right)\|c\|_{L^2}^2+e\|h^{m-1}\|_{L^\infty}\left(\|\nabla h\|_{H^1}+\|\nabla z\|_{H^1}\right)^n\|c\|_{L^2}+(s+\|r\|_{L^\infty})\frac{\|c\|^2_{L^2 }}{h_{min}}.
\label{c1}
\end{equation}
The first term on the right hand side of~\eqref{c1} is related to the advection, the second term to the erosion and the third one to the sedimentation and the source term. Now, for all $p \in \{1, \dots, k \}$, we differentiate $p$ times Equation~\eqref{eq_c} and multiply it by $\nabla^p c$. There is no additional issue with respect to the former computations on $z$ and one finds:
\vspace{-4mm}
{\setlength\arraycolsep{1pt}
\begin{eqnarray}
\displaystyle 
\frac{1}{2} \frac{d}{dt} ||c||^2_{H^k} &\leq& C\left(\|\nabla h\|_{H^{k+1}}+\|\nabla z\|_{H^{k+1}}\right)\|c\|^2_{H^k}+(s+\|r\|_{H^k})\frac{\|c\|_{H^k}^2}{h_{min}}P\left(\frac{\|\nabla h\|_{H^{k+1}}}{h_{min}}\right)\nonumber\\
\displaystyle
&&+eC\|h^{m-4}\|_{L^\infty}\left(\|h\|_{L^\infty}+\|\nabla h\|_{H^3}+\|\nabla z\|_{H^3}\right)^{n+3}.
\label{eq_in_energy_c}
\end{eqnarray}
}
Here $P$ is some polynomial of degree three with positive coefficients. Let us denote
$$
\mathcal{Q}_h(a,b)=Ka^2+h_{min}b^2-\|h\|_{L^\infty}ab.
$$

By combining the estimates~\eqref{h11},~\eqref{eq_in_energy_h},~\eqref{z0},~\eqref{eq_in_energy_z} and~\eqref{eq_in_energy_c} and the Young inequality on products:
\begin{align*}
    \forall (a,b)\in\RR^2,\forall\varepsilon>0, \quad ab \leq \frac{\varepsilon}{2} a^2 + \frac{1}{2 \varepsilon} b^2,
\end{align*}
one finds that for any $\varepsilon>0$, there exists $C(\varepsilon)$ such that
\begin{equation}
\displaystyle 
\frac{d}{dt}\mathcal{E}(t)+\mathcal{Q}_{h}\left(\|\nabla h\|_{H^{k+1}}, \|\nabla z\|_{H^{k+1}} \right)\leq \varepsilon\left(\|\nabla h\|_{H^{k+1}}^2+\|\nabla z\|_{H^{k+1}}^2\right)+C(\varepsilon)\left(\|r\|_{H^k}^2+\mathcal{F}(h,z,c)(t)\right),
\label{eq-E}
\end{equation}
where 
{\setlength\arraycolsep{1pt}
\begin{eqnarray}
\displaystyle
\mathcal{F}(h,z,c)&=&\|h-h_{ref}\|^2+\|h\|_{L^\infty}\|\nabla z\|_{L^2}\|\nabla h\|_{L^2}+\|c\|_{L^2}\|z\|_{L^2}\nonumber\\
&&+\|c\|_{H^k}^2+\|c\|_{H^k}^4\left(1+(h_{min})^{-2}P^2\left(\frac{\|\nabla h\|_{H^{k+1}}}{h_{min}}\right)\right)\nonumber\\
\displaystyle
&&+\|h\|_{L^\infty}^m\left(\|\nabla h\|_{H^1}+\|\nabla z\|_{H^1}\right)^n\|z\|_{L^2}+\left(\|\nabla h\|_{H^k}+\|\nabla z\|_{H}^k\right)^4\nonumber\\
\displaystyle
&&+\|h^{m-3}\|_{L^\infty}^2\left(\|h\|_{L^\infty}+\|\nabla h\|_{H^k}+\|\nabla z\|_{H^k}\right)^{2(n+1)}\nonumber\\
&&+\|h^{m-4}\|_{L^\infty}^2\left(\|h\|_{L^\infty}+\|\nabla h\|_{H^k}+\|\nabla z\|_{H^k}\right)^{(n+1)}\|c\|_{H^k}.
\label{def-F}
\end{eqnarray}
}
From~\eqref{def-F}, one easily proves that there exists a constant $C$ depending only on $h_{min}$ and $\alpha>0$  such that 
$$
\displaystyle
\mathcal{F}(h,z,c)(t)\leq C\left(1+C\mathcal{E}(t)^\alpha\right)\mathcal{E}(t)
$$
Note that we have used the estimate:
$$
\displaystyle
\|h\|_{L^\infty}\leq h_{ref}+\|h-h_{ref}\|_{L^2}+C\|\nabla h\|_{H^1},\quad \|h^{m-i}\|_{L^\infty}\leq \|h\|_{L^\infty}^m(h_{min})^{-i}, i=3,4.
$$

Now, the quadratic form $\mathcal{Q}_h$ is positive semi-definite provided that $Kh_{min}>\|h\|_{L^\infty}^2/4$. Under the assumption that $\|h^0\|_{L^\infty}<\sqrt{Kk_{min}}$, there exists a time $T^*$ such that $\sup_{(0, T^*)}\|h\|_{L^\infty}\leq \sqrt{3Kk_{min}}$. Thus there exists $\varepsilon>0$ such that $\mathcal{Q}_{h}(a, b)\geq 2\varepsilon(a^2+b^2)$. We deduce from~\eqref{eq-E} that for all $t\in (0, T^*)$, one has:
\begin{equation}
\displaystyle 
\frac{d}{dt}\mathcal{E}(t)+\varepsilon\left(\|\nabla h\|_{H^{k+1}}^2+\|\nabla z\|_{H^{k+1}}^2\right)\leq C(\varepsilon)\|r\|_{H^k}^2+C\left(1+C\mathcal{E}(t)^\alpha\right)\mathcal{E}(t),
\label{eq-final}
\end{equation}

Now, we integrate  Equation~\eqref{eq-final} with respect to time: one finds
\begin{equation}
\displaystyle
\mathcal{E}(t)+\varepsilon\int_0^t \left( \|\nabla h\|_{H^{k+1}}^2+\|\nabla z\|_{H^{k+1}}^2 \right) \leq \mathcal{E}(0)+C\int_0^t\|r\|_{H^k}^2+C\int_0^t(1+C\mathcal{E}(s)^\alpha)\mathcal{E}(s)ds.
\end{equation}

By apply one more time a continuity argument, there exists $T_1\leq T^*$ ,  such that $\mathcal{E}(t)\leq 2\mathcal{E}(0)$ for all $t\in[0, T_1]$. This implies that

\begin{equation}
\displaystyle
\mathcal{E}(t)+\varepsilon\int_0^t \left( \|\nabla h\|_{H^{k+1}}^2+\|\nabla z\|_{H^{k+1}}^2 \right) \leq \mathcal{E}(0)+C\int_0^t\|r\|_{H^k}^2+C(1+2^\alpha C\mathcal{E}(0)^\alpha)\int_0^t\mathcal{E}(s)ds.
\end{equation}

By applying Gronwall lemma, one finds
$$
\begin{array}{ll}
\displaystyle
\mathcal{E}(t)\leq e^{Ct}\left(\mathcal{E}(0)+C\int_0^t\|r\|^2_{H^k}\right),\\
\displaystyle
\int_0^t \left( \|\nabla h\|_{H^{k+1}}^2+\|\nabla z\|_{H^{k+1}}^2 \right) \leq C\,e^{Ct}\left(\mathcal{E}(0)+C\int_0^t\|r\|^2_{H^k}\right).
\end{array}
$$
This concludes the proof of the energy estimate.

\end{proof}

\subsection{The approximate system} \label{approximate system}

In this section, we prove the existence of a sequence of solutions $(h_i, z_i, c_i)_{i \in \NN}$ to the following linear system:
\begin{subequations}
\begin{empheq}[left= \empheqlbrace]{align}
&\displaystyle \dt h_{i+1} -\dive(h_i \nabla h_{i+1}) - \dive(h_{i+1} \nabla z_i) = r(t,x), \label{eq_hi} \\
&\displaystyle \dt z_{i+1} = K \Delta z_{i+1} + sc_i - e h_i^m |v_i|^n, \label{eq_zi}\\
&\displaystyle \dt c_{i+1} + v_i.\nabla c_{i+1} = e h_i^{m-1} |v_i|^n - sc_i/h_i - rc_i/h_i, \label{eq_ci}\\
&\displaystyle h_i(0,x) = h(0,x), \, z_i(0,x) = z(0,x), \, c_i(0,x) = c(0,x). \nonumber
\end{empheq}
\label{eq_systapprox}
\end{subequations}

with initial conditions $h_0(t,x) = h^0(x)$, $z_0(t,x) = z^0(x)$ and $c_0(t,x) = c^0(x)$. Then we show that these solutions are uniformly bounded with respect to $i\in\mathbb{N}$ for a suitable Sobolev norm. In what follows, we will denote $h_{min}=inf_{x\in\\R^2}(h_0(x))/2$.\smallskip

The classical theory of parabolic equations and of linear transport equations provides conditions to obtain a well-posed system of equations. For the two parabolic equations~\eqref{eq_hi} and~\eqref{eq_zi}, we state the following result (see \cite{cherrier_linear_2012} for more details):

\begin{proposition}
\label{parab}
Assume that $h^0-h_{ref}, z^0 \in H^{k+1}(\RR^2)$, $r$, $sc_i - e h_i^m |v_i|^n \in \lhk$, $h_i \in L^\infty([0, T]\times \RR^2)$ with $h_i\geq h_{min}$ and $\nabla h_i, \nabla z_i \in L^1_T(H^k)$. Then there exists a unique solution $(h_{i+1}, z_{i+1})$ to the equations~\eqref{eq_hi},~\eqref{eq_zi}, and $h_{i+1}-h_{ref}, \, z_{i+1} \in \lhkde \cap C_T(H^{k+1}(\RR^2))$.
\end{proposition}

For the transport equation~\eqref{eq_ci}, we use theorem $7.2.2$ in \cite{metivier_para-differential_2008} to state the result:
\begin{proposition}
\label{transport}
Assume that $c^0 \in H^{k}(\RR^2)$, $\, \nabla v_i \in L^\infty_T(H^{k-1})$, $e h_i^{m-1} |v_i|^n - sc_i/h_i - rc_i/h_i \in L^{1}_T(H^k)$ and $v_i \in L^1_T(W^{1,\infty})$. Then there exists a unique solution $c_{i} \in C_T(\hk)$ to the equation~\eqref{eq_ci}.
\end{proposition}

We prove by induction that for all $i \in \NN$, System~\eqref{eq_systapprox} is well-posed.
First, for $i=0$, the functions of $(h_i,z_i,z_i)$ are time independent and $h^0\geq h_{min}$. Thus, as $T_0=T$ is finite, we have $sc_0-eh_0^m|v_0|^n\in L^2_{T_0}(H^k)$, $h_0\in L_{T_0}^\infty(\RR^2)$
and $\nabla h_0,\nabla z_0\in L^1_{T_0}(H^k)$. Then there exists a unique solution $(h_1, z_1)\in  L^2_{T_1}(H^{k+2}) \cap C_{T_1}(H^{k+1})$ of~\eqref{eq_hi} and~\eqref{eq_zi} with $0<T_1\leq T_0$ such that $h_1\geq h_{min}$. Similarly, there exists $c_1\in C_{T_1}(H^k)$ solution of~\eqref{eq_ci}. Now, if we assume that $h_{i}-h_{ref}, z_{i} \in \lhkde \cap C_{T_i}(H^{k+1})$ and $c_i\in C_{T_i}(H^k)$ then, one has $h_i\in h_{ref}+C_{T_i}(H^2)\subset L^\infty([0,T]\times\RR^2)$. Moreover, we have $\nabla h_i, \nabla z_i\in C_{T_i}(H^{k})\subset L^1_{T_i}(H^{k})$ and $c_i\in C_{T_i}(H^k)\subset L^2_{T_i}(H^k)$. Finally, one has
$$
\displaystyle
\|h_i^m|v_i|^n\|_{H^k}\leq C\|h_i^{m-3}\|_{L^\infty}\left(\|h_i\|_{L^\infty}+\|\nabla h_i\|_{H^k}+\|\nabla z_i\|_{H^k}\right)^{n+3}\in L^\infty_{T_i}(H^k)\subset L^2_{T_i}(H^k).
$$
Thus, the assumptions of Proposition ~\ref{parab} are satisfied and there exists a unique solution $h_{i+1}, z_{i+1}$ of (\ref{eq_hi},\ref{eq_zi}) such that $h_{i+1}-h_{ref}, \, z_{i+1} \in L^2_{T_{i+1}}(H^{k+2}) \cap C_{T_{i+1}}(H^{k+1}(\RR^2))$ with $0<T_{i+1}\leq T_i$ such that $h_{i+1}\geq h_{min}$. The existence of a solution $c_{i+1}\in C_{T_{i+1}}(H^k)$ follows similarly.


\begin{proposition}[Uniform bounds]
Denote 
$$
\displaystyle
\mathcal{E}_i(t)=\|h_i-h_{ref}\|_{L^2}+\|\nabla h_i\|_{H^k}+\|z_i\|_{H^{k+1}}+\|c_i\|_{H^k}.
$$
There exists $B>0$ and $0<T^*\leq T$ independent of $i\in\mathbb{N}$ such that, for all $i \in \NN$
$$
\displaystyle
\sup_{t\in[0, T^*]}\mathcal{E}_i(t)\leq B,\qquad {\rm and}\qquad h_i(t,x)\geq h_{min},\quad \forall x\in\RR^2. 
$$
Moreover, 
$$
\displaystyle
\int_0^{T^*} \left( \|\nabla h_ i\|^2_{H^{k+1}}+\|\nabla z_i\|_{H^{k+1}}^2 \right) \leq C(T^*)B,\quad \forall i\in\mathbb{N}^*.
$$
\label{uniformBounds}
\end{proposition}

\begin{proof}
We proceed by induction. One has easily $\displaystyle\mathcal{E}_0(t)=\mathcal{E}(0)\leq B$ and $h_0\geq h_{min}$ since $h_0,z_0,c_0$ are time independent. Following the strategy used to derive a priori estimates, one can prove that for some $\varepsilon>0$, for all $t\in[0, T]$,
\begin{equation}\label{eq-E1}
\displaystyle
\frac{d}{dt}\mathcal{E}_1(t)+\varepsilon\left(\|\nabla h_1\|^2_{H^{k+1}}+\|\nabla z_1\|^2_{H^{k+1}}\right)\leq C(\varepsilon)(\|r\|^2_{H^k}+\mathcal{E}_0(t)^\alpha)\mathcal{E}_{1}(t).
\end{equation}
for some constant $C(\varepsilon)$ depending only on $\varepsilon$ and $h^0$. By integrating Equation~\eqref{eq-E1} with respect to time, one finds:
$$
\displaystyle
\mathcal{E}_1(t)+\varepsilon\int_0^t \|\nabla h_1\|^2_{H^{k+1}}+\|\nabla z_1\|^2_{H^{k+1}}\leq \mathcal{E}(0)+C(\varepsilon)\int_0^t\|r\|_{H^k}^2+\int_0^tC(\varepsilon)\mathcal{E}_0(s)^\alpha\mathcal{E}_1(s)ds.
$$
Denote 
$$
\displaystyle
B=2\left(\mathcal{E}(0)+C(\varepsilon)\int_0^T\|r\|^2_{H^k}\right).
$$
We choose $\tilde{T}^*$ such that 
$$\displaystyle e^{\tilde{T}^*C(\varepsilon)B^\alpha}\leq 2.
$$
Then, by applying Gronwall's lemma, one obtains:
$$
\begin{array}{ll}
\displaystyle
\mathcal{E}_1(t)\leq B,\qquad \forall t\in[0, \tilde{T}^*],\\
\displaystyle
\int_0^{\tilde{T}^*}\|\nabla h_1\|^2_{H^{k+1}}+\|\nabla z_1\|^2_{H^{k+1}}\leq C(\varepsilon)\frac{1+\ln(2)}{\varepsilon}B.
\end{array}
$$
Next, we have 
$$
h_1=h^0+\int_0^t\dive(h_0\nabla (h_1)+h_1\nabla z_0))+r.
$$
Thus 
$\displaystyle \|h_1-h^0\|_{L^\infty}\leq C(\tilde{B}t+\sqrt{Bt})$
for some constant $C$ independent of the problem and related to Sobolev injections whereas $\tilde B=B+h_{ref}\sqrt{B}$. Then, there exists $0<T^*\leq\tilde{T}^*$ such that $C(\tilde{B}t+\sqrt{Bt})\leq h_{min}$ and we deduce that 
$$
\displaystyle
h_1\geq h_0-h_{min}\geq h_{min}\qquad \forall t\in[0, T^*].
$$
This proves the initial step for $i=1$. Now assume that $\mathcal{E}_i(t)\leq B$ and $h_i\geq h_{min}$ for all $t\in[0, T^*]$. The estimates on $h_{i+1}, z_{i+1}$ and $c_{i+1}$ are a direct consequence of the energy estimate:
\begin{equation}\label{eq-Ei}
\displaystyle
\frac{d}{dt}\mathcal{E}_{i+1}(t)+\varepsilon\left(\|\nabla h_{i+1}\|^2_{H^{k+1}}+\|\nabla z_{i+1}\|^2_{H^{k+1}}\right)\leq C(\varepsilon)(\|r\|^2_{H^k}+\mathcal{E}_{i}(t)^\alpha)\mathcal{E}_{i+1}(t)
\end{equation}
which is proved by following the strategy used to derive the a priori estimates. This completes the proof of the proposition.

\end{proof}

\begin{proposition}
The sequences $(\dt h_i)$ et $(\dt z_i)$ are uniformly bounded in $L^2_{T^*}(H^k)$. The sequence $(\dt c_i)$ is uniformly bounded on $L^2_{T^*}(H^{k-1})$.
\label{bounds_dt}
\end{proposition}

\begin{proof}

For $p \in \{0, \dots, k\}$ we differentiate $p$ times the equations~\eqref{eq_hi},~\eqref{eq_zi}~\eqref{eq_zi} and multiply it by $\nabla^p \dt h_i$, $\nabla^p \dt z_i$ and $\nabla^p \dt c_i$ respectively, and integrate in space. Using Proposition ~\ref{uniformBounds}, we obtain bounds for $\|\dt h_i\|_{L^2_{T^*}(H^k)}$, $\|\dt z_i\|_{L^2_{T^*}(H^k)}$ and $\|\dt c_i\|_{L^2_{T^*}(H^{k-1})}$.

\end{proof}

\subsection{Convergence of the sequences} \label{convergence}

In what follows, we denote $T^*=T$ in order to simplify the notations.

\begin{proposition}
The sequences $(h_i)_{i\in\mathbb{N}}$, $(\dt h_i)_{i\in\mathbb{N}}$ and $(z_i)_{i\in\mathbb{N}}$, $(\dt z_i)_{i\in\mathbb{N}}$ are Cauchy sequences in $L^2_{T}(L^2)$. The sequences $(c_i)_{i\in\mathbb{N}}$ and $(\dt c_i)_{i\in\mathbb{N}}$ are Cauchy sequences in $C_{T}(L^2)$.
\label{cauchy}
\end{proposition}

\begin{proof}
In the following we denote the quantities of the form $f_{i+1}-f_i$ by $\delta f_i$. For all $i\geq 1$, the equations for $\delta h_i$, $\delta z_i$, $\delta c_i$ are written as:
\begin{subequations}
\begin{empheq}[left= \empheqlbrace]{align}
& \displaystyle \dt \delta h_i - \dive (h_i \nabla \delta h_i + \delta h_{i-1}  \nabla h_i) - \dive (h_{i+1} \nabla \delta z_i + \delta h_i \nabla z_{i-1}) = 0, \label{eq_delta_hi} \\[0.3em]
& \displaystyle \dt \delta z_i - K \delta \Delta z_i = s \delta c_i - e \left(h_i^m (|v_i|^n - |v_{i-1}|^n) + |v_{i-1}|^n (h_i^m - h_{i-1}^m) \right), \label{eq_delta_zi} \\[0.3em]
& \begin{array}{ll}
    \displaystyle \hspace{-1,8mm} \dt \delta c_i + \delta v_{i-1} \nabla c_{i+1} + v_{i-1} \delta \nabla c_i & \displaystyle \hspace{-2,5mm} = e \left(h_i^{m-1} (|v_i|^n - |v_{i-1}|^n) + |v_{i-1}|^n (h_i^{m-1} - h_{i-1}^{m-1}) \right) \smallskip \\ 
    & \displaystyle -(s+r) \left(\displaystyle \frac{c_i}{h_i} - \frac{c_{i-1}}{h_{i-1}} \right). \end{array} \label{eq_delta_ci}
\end{empheq}
\label{eq_delta}
\end{subequations}

\paragraph{Bounds on $\delta h_i$, $\delta z_i$ and $\delta c_i$: }
We multiply the equation~\eqref{eq_delta_hi} by $\delta h_i$, and integrate over space: 
\begin{align*}
\frac{1}{2} \frac{d}{dt} \|\delta h_i\|_{L^2}^2 + \int_{\RR^2} (h_i \nabla h_{i+1} - h_{i-1} \nabla h_i) \nabla \delta h_i + \int_{\RR^2} (\nabla z_i h_{i+1} - \nabla z_{i-1} h_i)\nabla \delta h_i = 0.
\end{align*}
First, on the one hand, we have 
$$
\displaystyle
\int_{\RR^2} (h_i \nabla h_{i+1} - h_{i-1} \nabla h_i) \nabla \delta h_{i} = \int_{\RR^2} h_i (\nabla \delta h_i)^2 + \delta h_{i-1}  \nabla h_i \nabla \delta h_i \geq h_{min} \|\nabla \delta h_i\|_{L^2}^2 + \int_{\RR^2} \delta h_{i-1}  \nabla h_i \nabla \delta h_i.
$$
On the other hand, we have the estimate:
\begin{align*}
\displaystyle \int_{\RR^2} (\nabla z_i h_{i+1} - \nabla z_{i-1} h_i)\nabla \delta h_i &= \frac{1}{2} \int_{\RR^2} \nabla z_i \nabla (\delta h_i)^2 + \int_{\RR^2} \delta z_{i-1} h_i \nabla \delta  h_i = -\frac{1}{2} \int_{\RR^2} \Delta z_i (\delta h_i)^2 + \int_{\RR^2} \delta z_{i-1} h_i \nabla \delta  h_i	\\
& \displaystyle \leq C \|\delta h_i\|^2_{L^2} + \| h_i\|_{L^\infty} \|\delta z_{i-1}\|_{L^2} \|\nabla \delta h_i\|_{L^2}\\
& \displaystyle \leq  C \|\delta h_i\|^2_{L^2} + \frac{1}{2 h_{min}} \|\delta z_{i-1}\|^2_{L^2} + \frac{h_{min}}{2} \|\nabla \delta h_i\|^2_{L^2}.
\end{align*}
Consequently, we obtain:
\begin{align}
&\frac{d}{dt} \|\delta h_i\|_{L^2}^2 + h_{min} \|\nabla \delta h_i\|^2_{L^2} \leq C \|\delta h_i\|_{L^2}^2 + C \|\delta h_{i-1}\|_{L^2}^2 + C \|\delta z_{i-1}\|_{L^2}^2.
\label{eq_bound_delta_hi}
\end{align}

We proceed similarly for $\delta z_i$. By using Equation~\eqref{eq_delta_zi}, one finds:
\begin{align*}
& \displaystyle \frac{d}{dt} \|\delta z_i\|_{L^2}^2 + 2K \|\nabla \delta z_i\|^2 \leq s (\|\delta c_{i}\|^2_{L^2} + \|\delta z_i\|^2_{L^2}) + 2e\left| \int_{\RR^2} (h_i^m |v_i|^n - h_{i-1}^m |v_{i-1}|^n) \delta z_i\right|.
\end{align*}
In order to bound the right-hand term, we use the inequality $||x|^n-|y|^n| \leq n |x-y| \max(|x|,|y|)^{n-1}$:
\begin{align*}
h_i^m |v_i|^n - h_{i-1}^m |v_{i-1}|^n &= h_i^m (|v_i|^n-|v_{i-1}|^n) + |v_{i-1}|^n (h_i^m - h_{i-1}^m)\\
&\leq n h_i^m |v_i-v_{i-1}| \, \max(\|v_{i-1}\|_{L^\infty},\|v_i\|_{L^\infty})^{n-1}\\ 
&+ \frac{m}{h_{min}} |v_{i-1}^n |h_i - h_{i-1}| \, \max(\|h_i\|_{L^\infty},\|h_{i-1}\|_{L^\infty})^{m}.
\end{align*}
Consequently, as $h_i, h_{i-1}$ and $v_i, v_{i-1}$ are uniformly bounded in $L^\infty([0,T] \times \RR^2)$:
\begin{align*}
\displaystyle \int_{\RR^2} (h_i^m |v_i|^n - h_{i-1}^m |v_{i-1}|^n) \delta z_i &\leq
C (\|v_i - v_{i-1}\|_{L^2} + \|h_i-h_{i-1}\|_{L^2})\|\delta z_i\|_{L^2}\\
& \displaystyle \leq C (\|\nabla \delta h_{i-1}\|_{L^2} + \|\nabla \delta z_{i-1}\|_{L^2} + \|\delta h_{i-1}\|_{L^2})\|\delta z_i\|_{L^2}\\
& \displaystyle \leq C (\|\nabla \delta h_{i-1}\|_{L^2}^2 + \|\nabla \delta z_{i-1}\|_{L^2}^2) + \|\delta h_{i-1}\|_{L^2}^2 + \|\delta z_i\|_{L^2}^2 ).
\end{align*}
Thus,
\begin{align}
&\frac{d}{dt} \|\delta z_i\|_{L^2}^2 + 2K \|\nabla \delta z_i\|^2_{L^2} \leq C \|\delta z_i\|_{L^2}^2 + C (\|\nabla \delta h_{i-1}\|_{L^2}^2 + \|\nabla \delta z_{i-1}\|_{L^2}^2 + \|\delta h_{i-1}\|_{L^2}^2).
\label{eq_bound_delta_zi}
\end{align}

Finally, for $\delta c_i$, we use the same method as before with the equation~\eqref{eq_delta_ci} and we obtain:
\begin{align}
\displaystyle \frac{1}{2} & \frac{d}{dt} \|\delta c_i\|_{L^2}^2 = -\int \left( \delta v_{i-1} \nabla c_{i+1} + v_i \nabla \delta c_i \right)\delta c_i + e \int (h_i^m |v_i|^n - h_{i-1}^m |v_{i-1}|^n) \delta c_i - (s+r) \int \left(\frac{c_i}{h_i} - \frac{c_{i-1}}{h_{i-1}} \right) \delta c_i \nonumber \\
& \displaystyle= -\int \left( \delta v_{i-1}\nabla c_{i+1}\delta c_i - \frac{1}{2} \dive(v_i) (\delta c_i)^2 \right) + e \int (h_i^m |v_i|^n - h_{i-1}^m |v_{i-1}|^n) \delta c_i - (s+r) \int \left(\frac{c_i}{h_i} - \frac{c_{i-1}}{h_{i-1}} \right) \delta c_i \nonumber \\
& \displaystyle \leq \|\delta v_{i-1}\|_{L^2} \|\nabla c_{i+1}\|_{L^\infty} \|\delta c_i\|_{L^2} + \|\nabla v_i\|_{L^\infty} \|\delta c_{i}\|_{L^2}^2 + C (\|\nabla \delta h_{i-1}\|_{L^2}^2 + \|\nabla \delta z_{i-1}\|_{L^2}^2 + \|\delta h_{i-1}\|_{L^2}^2) \nonumber \\
& \displaystyle \qquad + \frac{s+\|r\|_{H^2}}{h_{min}} \left(\|c_i\|_{L^2} + \|c_{i-1}\|_{L^2}\right) \|\delta c_i\|_{L^2}.
\label{eq_bound_delta_ci}
\end{align}

Then we add the inequalities~\eqref{eq_bound_delta_hi},~\eqref{eq_bound_delta_zi} and~\eqref{eq_bound_delta_ci} and integrate on $[0,t]$ with $0\leq t\leq T$:
\begin{align*}
    \displaystyle \|\delta h_i\|_{L^2}^2 & + \|\delta z_i\|_{L^2}^2 + \|\delta c_i\|_{L^2}^2
    + h_{min} \int_0^t \|\nabla \delta h_i\|^2_{L^2} dt + 2K \int_0^t \|\nabla \delta z_i\|^2_{L^2} dt\\
    & \displaystyle \leq \|\delta h_0\|_{L^2}^2 + \|\delta z_0\|_{L^2}^2 + \|\delta c_0\|_{L^2}^2
    + C \int_0^t \left(\|\delta h_i\|_{L^2}^2 + \|\delta z_i\|_{L^2}^2+ \|\delta c_i\|_{L^2}^2 \right) dt\\
    & \displaystyle + C \int_0^t \left(\|\delta h_{i-1}\|_{L^2}^2 + \|\delta z_{i-1}\|_{L^2}^2 + \|\delta c_{i-1}\|_{L^2}^2 \right) dt + C \int_0^t (\|\nabla \delta h_{i-1}\|_{L^2}^2 + \|\nabla \delta z_{i-1}\|_{L^2}^2) dt.
\end{align*}

We apply the Gronwall lemma and obtain, for all $t\in [0, T]$,
\begin{align*}
    & \displaystyle\|\delta h_i\|_{L^2}^2 + \|\delta z_i\|_{L^2}^2 + \|\delta c_i\|_{L^2}^2
    + h_{min} \int_0^t \|\nabla \delta h_i\|^2_{L^2} dt + 2K \int_0^t \|\nabla \delta z_i\|^2_{L^2} dt
    \leq C \biggl[\|\delta h_0\|_{L^2}^2 + \|\delta z_0\|_{L^2}^2 + \|\delta c_0\|_{L^2}^2 \\
    & \displaystyle+ \int_0^t \left(\|\delta h_{i-1}\|_{L^2}^2 + \|\delta z_{i-1}\|_{L^2}^2 + \|\delta c_{i-1}\|_{L^2}^2 \right) dt + \int_0^t (\|\nabla \delta h_{i-1}\|_{L^2}^2 + \|\nabla \delta z_{i-1}\|_{L^2}^2) dt \biggr] e^{Ct}.
\end{align*}

With this inequality, we deduce by induction on $i$ that $\forall i \in  \NN$,
\begin{align*}
\|\delta h_i\|_{L^2_T(L^2)}^2 + \|\delta z_i\|_{L^2_T(L^2)}^2 + \|\delta c_i\|_{L^\infty_T(L^2)}^2 &\leq C \frac{\left(Te^{CT} \right)^i}{i!} \left(\|\delta h_0\|_{H^1}^2 + \|\delta z_0\|_{H^1}^2 + \|\delta c_0\|_{L^2}^2 \right).
\end{align*}

Consequently the series $\sum\|\delta h_i\|_{L^2_T(L^2)}^2$, $\sum \|\delta z_i\|_{L^2_T(L^2)}^2$ and $\sum\|\delta c_i\|_{L^\infty_T(L^2)}^2$ converge, thus the sequences $(h_i)_{i\in\mathbb{N}}$, $(z_i)_{i\in\mathbb{N}}$, $(c_i)_{i\in\mathbb{N}}$  are Cauchy sequences in the required spaces.

\paragraph{Bounds on $\dt \delta h_i$, $\dt \delta z_i$ and $\dt \delta c_i$:}
Like the estimates in Proposition ~\ref{bounds_dt}, we use the system~\eqref{eq_delta} and Proposition ~\ref{uniformBounds} to obtain the bounds.
\end{proof}

By Proposition ~\ref{cauchy} there exists $h-h_{ref}, z \in L^2_T(L^2)$ such that $h_i-h_{ref}$ converges to $h-h_{ref}$ and $z_i$ converges to $z$ in $L^2_T(L^2)$. As $(h_i-h_{ref})$ and $(z_i)$ are uniformly bounded in $\lhkde$, we obtain by interpolation that $\forall \, 1/2 < \theta < 1$, $\forall i>j \in \NN$,
\begin{align*}
\displaystyle \int_0^T \|h_i - h_j\|_{H^{\theta(k+2)}}^2 dt &\leq \int_0^T \left( \|h_i-h_j\|_{L^2}^{1-\theta} \|h_i-h_j\|_{H^{k+2}}^{\theta} \right) dt \\
& \displaystyle\leq \|h_i-h_j\|_{C_T(L^2)}^{1-\theta} \int_0^T \|h_i-h_j\|_{H^{k+2}}^{\theta} dt\\
& \displaystyle\leq \|h_i-h_j\|_{C_T(L^2)}^{1-\theta} T^{2/(2-\theta)} \|h_i-h_j\|_{L^2_T(H^{k+2})}^{\theta}\\
& \displaystyle\leq \|h_i-h_j\|_{C_T(L^2)}^{1-\theta} T^{2/(2-\theta)} (2C_1)^{\theta} \underset{n \to +\infty}{\longrightarrow} 0.
\end{align*}
Therefore $\forall k/2 < s < k$, the sequences $(h_i-h_{ref})$ and $(z_i)$ are Cauchy sequences in $L^2_T(H^{s+2})$, thus $h-h_{ref}$, $z \in L^2_T(H^{s+2})$. Moreover $\dt h_i$ converges to $\dt h$ in $L^2_T(L^2)$ and $(\dt h_i)$ is uniformly bounded in $\lhk$, similarly for $\dt z$. Consequently $\forall s <k$, $\dt h$, $\dt z \in L^2_T(H^s)$, thus by Proposition ~\ref{fdtf}, $h-h_{ref}$, $z \in C_T(H^{s+1})$. Finally the a priori estimates on $h$ and $z$ allow to conclude that $h-h_{ref}$, $z \in \lhkde \cap C_T(H^{k+1})$.
In particular, as $k+1 =4$, $(h_i-h_{ref})$, $(\nabla h_i)$, $(\nabla^2 h_i)$, $(\dt h_i)$ and $(z_i)$, $(\nabla z_i)$, $(\nabla^2 h_i)$, $(\dt z_i)$ converges in $C([0,T] \times \RR^2)$.

Now we consider $(c_i)_{i\in\mathbb{N}}$. By Proposition ~\ref{cauchy}, there exists $c \in C_T(L^2)$ limit of $(c_i)_{i\in\mathbb{N}}$ in this space. We know that $(c_i)_{i\in\mathbb{N}}$ is uniformly bounded in $C_T(H^k)$, so by interpolation: $\forall \, 0 < \theta < 1$, $\forall i>j \in \NN$,
\begin{align*}
\underset{t \in [0,T]}{\text{sup}} \|c_i - c_j\|_{H^{\theta(k)}} &\leq \underset{t \in [0,T]}{\text{sup}} \|c_i-c_j\|_{L^2}^{1-\theta} \underset{t \in [0,T]}{\text{sup}} \|c_i-c_j\|_{H^{k}}^{\theta} \\
&\leq \|c_i-c_j\|_{C_T(L^2)}^{1-\theta} (\|c_i\|_{C_T(H^{k})}+\|c_j\|_{C_T(H^{k})})^{\theta}\\
&\leq C \|c_i-c_j\|_{C_T(L^2)}^{1-\theta} \underset{n,\,m \to +\infty}{\longrightarrow} 0.
\end{align*}
Therefore $(c_i)_{i\in\mathbb{N}}$ is a Cauchy sequence, and thus converges to $c$ in the space $C_T(H^s)$, for all $s<k$. Moreover $(\dt c_i)$ converges to $\dt c$ in $C_T(L^2)$ and is uniformly bounded in $C_T(H^{k-1})$, so it converges to $c$ in $C_T(H^{s-1})$. And we conclude by the a priori estimate on $c$ that $c \in C_T(H^k)$.
To conclude, $(c_i)_{i\in\mathbb{N}}$, $(\dt c_i)_{i\in\mathbb{N}}$ and $(\nabla c_i)_{i\in\mathbb{N}}$ converge in $C_T(\RR^2)$ thus we can take the limit in the equations, and $(h,z,c)$ is solutions of System~\eqref{eq_WPsyst}. This concludes the proof of the existence.

\subsection{Uniqueness} \label{uniqueness}

\begin{proposition}
Let $(h_1, z_1, c_1)$ and $(h_2, z_2, c_2)$ be two solutions of System~\eqref{eq_WPsyst} satisfying the hypotheses of Theorem~\ref{th}. Then $\forall t \in [0,T]$ :
\begin{align*}
\|h_1(t)-h_2(t)\|^2_{L^2} + \|z_1(t)-z_2(t)\|^2_{L^2} + \|c_1(t)-c_2(t)\|^2_{L^2}
\leq \left(\|h^0_1 - h^0_2\|_{L^2}^2 + \|z^0_1 - z^0_2\|_{L^2}^2 + \|c_1^0 - c^0_2\|_{L^2}^2 \right) e^{CT}
\end{align*}
\end{proposition}

In particular when the initial conditions are the same for both solutions, these solutions are the same. Consequently this proposition shows the uniqueness of the solution $(h,z,c)$ of the theorem.

\begin{proof}
We first write the equations verified by $\delta h := h_1-h_2$, $\delta z := z_1-z_1$ and $\delta c := c_1-c_2$ :
\begin{subequations}
\begin{empheq}[left= \empheqlbrace]{align}
& \displaystyle\dt \delta h - \dive (h_1 \nabla \delta h + \delta h  \nabla h_2) - \dive (h_1 \nabla \delta z + \delta h \nabla z_2) = 0\label{eq_delta_h} \\[0.3em]
& \displaystyle\dt \delta z - K \delta \Delta z = s \delta c - e \left(h_1^m (|v_1|^n - |v_2|^n) + |v_2|^n (h_1^m - h_{2}^m) \right) \label{eq_delta_z} \\
& \displaystyle\dt \delta c + \delta v \nabla c_2 + v_2 \delta \nabla c =  e \left(h_1^{m-1} (|v_1|^n - |v_{2}|^n) + |v_{2}|^n (h_1^{m-1} - h_{2}^{m-1}) \right) -(s+r) \left(\frac{c_1}{h_1} - \frac{c_{2}}{h_{2}} \right) \label{eq_delta_c}
\end{empheq}
\end{subequations}

Then the bound is obtained with a similar process as in the proof of Proposition ~\ref{cauchy}.
\end{proof}

 \section{Spectral stability of constant states} \label{stability}

In this section, we consider the flow over a topography that is an inclined plane at time $t=0$. We assume that $r=0$ and we study the spectral stability of constant states. We expect that instability will provide a mechanism for pattern formation. We first write System~\eqref{syst_complet} in a non-dimensional form and then linearize this system around constant states. Then we explore numerically the stability of the system. Finally, we carry out the spectral stability analysis by using Routh-Hurwitz theorem: this provides necessary and sufficient conditions for constant states to be spectrally stable. However, these conditions do not provide any insight on the nature of the instabilities. We complete this analysis by an asymptotic expansion of the spectrum around the origin and in the high frequency regime.


\subsection{Non-dimensionalization and linearization of the system} \label{adimensionalisation}

We write System~\eqref{syst_complet} in a non-dimensional form in order to identify the important parameters. We introduce several characteristic quantities : $Z$ is a characteristic eroded height, $H$ a characteristic water height, $L$ a characteristic wavelength and $T$ a characteristic time. We chose $T$ to be the necessary time to erode the soil of a height $Z$, with an erosion speed $e$. Thus $T$ verifies $e\,T = Z$. As $v = \mu \tan \theta e_1 - \mu \nabla (z+h)$, we fix the characteristic water velocity $V = \mu$. We introduce the dimensionless variables :
$$
h' := \frac{h}{H}, \quad z' = \frac{z}{Z}, \quad v' := \frac{v}{V}, \quad c' := \frac{c}{c_{sat}}\quad
x' := \frac{x}{L}, \quad t' := \frac{e}{Z} t.
$$
In order to simplify the notations, we will assume $H = Z$. Then, dropping the primes, System~\eqref{syst_complet} is written as:
\begin{align*}
\left\{ \begin{array}{lll}
\displaystyle
\dt h + \frac{Z V \tan \theta}{eL}\dx h = \frac{Z^2 V}{e L^2} \dive (h \nabla (h+z)), \vspace{2mm} \\
\displaystyle
h\dt c + \frac{Z V \tan \theta}{eL} h \dx c = \frac{Z V}{eL^2} h \nabla (h+z).\nabla c + \frac{\rho_s}{c_{sat}} h^m |\tan \theta e_1 - \frac{Z}{L} \nabla(h+z)|^n - \frac{s}{e} \frac{\rho_s}{c_{sat}} c, \vspace{2mm}\\
\displaystyle
\dt z = \frac{ZK}{eL^2} \Delta z - h^m |\tan \theta e_1 - \frac{Z}{L} \nabla(h+z)|^n + \frac{s}{e} c.
\end{array} \right.
\end{align*}
To simplify the equations, we set $\frac{Z}{L} = \frac{e}{V}$, and we define $\alpha := \frac{Z}{L} = \frac{e}{V}$, $K := \frac{ZK}{eL^2} = \frac{K}{LV}$. The system reads :
\begin{align}
\left\{ \begin{array}{lll}
\displaystyle
\dt h + \tan \theta \dx h = \alpha \dive (h \nabla (h+z)), \vspace{1mm}\\
\displaystyle
h\dt c +  \tan \theta h \dx c = \alpha{h}  \nabla (h+z).\nabla c + \frac{\rho_s}{c_{sat}} h^m |\tan \theta e_1 - \alpha \nabla(h+z)|^n - \frac{\rho_s s}{c_{sat} e} c,\\
\displaystyle
\dt z = K \Delta z - h^m |\tan \theta e_1 - \alpha \nabla(h+z))|^n + \frac{s}{e} c.
\end{array} \right.
\label{systadim}
\end{align}
The stationary states of Equation~\eqref{systadim} for a flat surface, denoted by $(\underline{h}, \underline{c}, \underline{z})$ verify:
\begin{align*}
     \forall t \in \RR^+, \forall (x,y) \in \Omega, \quad 
     \displaystyle
     \left\{ \begin{array}{lll}
          h(t,x,y) &= \underline{h}>0,\\
          c(t,x,y) &= \displaystyle \underline{c} = \frac{e}{s} \underline{h}^m \tan^n \theta >0,\\
          z(t,x,y) &= 0. 
     \end{array} \right.
\end{align*}
This means that the erosion and the deposition process equilibrate each other and the bottom is not eroded, whereas the fluid height is a constant.\\

Let $(\underline{h} + h, \underline{c} + c, z)$, with $|h|,|c|,|z|\ll 1$, be a small perturbation of the constant state, and solution of System~\eqref{systadim}. Then, at first order, this solution verifies the following linear system:
\begin{align}
\left\{ \begin{array}{lll}
\displaystyle
\dt h + \tan \theta \dx h = \alpha \underline{h} \, \Delta (h+z) \vspace{1mm}\\
\displaystyle
\dt c + \tan \theta \dx c = \frac{\rho_s}{c_{sat}} \frac{s}{e} \left( \frac{m \underline{c}}{\underline{h}^2} h - \frac{\alpha n \underline{c}}{\underline{h} \tan \theta} \dx (h+z) - \frac{c}{\underline{h}} \right) \vspace{2mm}\\
\displaystyle
\dt z = K \Delta z - \frac{s}{e} \left( \frac{m \underline{c}}{\underline{h}} h - \frac{\alpha n \underline{c}}{\tan \theta} \dx (h+z) - c \right)
\end{array} \right.
\label{eq_systlin}
\end{align}
We denote $f = (h,c,z)^T$. Then $f$ verifies the equation $\dt f = A_0 f + A_1 \dx f + A_2 \Delta f$ with:
\begin{align*}
A_0 =
\left[ \begin{array}{ccc}
\displaystyle
0 & 0 & 0\\
\vspace{0.2cm} \displaystyle
\frac{a m \underline{c}}{\underline{h}} & - a & 0\\
\vspace{0.2cm} \displaystyle
- \frac{a m \underline{c}}{\bar \rho_s} & \displaystyle \frac{a \underline{h}}{\bar \rho_s} & 0
\end{array} \right],
& \quad A_1 =
\left[ \begin{array}{ccc}
\displaystyle -\tan \theta & 0 & 0\\
\vspace{0.2cm} 
\displaystyle - \frac{\alpha a n \underline{c}}{\tan \theta} & \displaystyle -\tan \theta & \displaystyle - \frac{\alpha a n \underline{c}}{\tan \theta}\\
\vspace{0.2cm} 
\displaystyle \frac{\alpha a n \underline{h} \, \underline{c}}{\bar \rho_s \tan \theta} & 0 & \displaystyle \frac{\alpha a n \underline{h} \, \underline{c}}{\bar \rho_s \tan \theta}
\end{array} \right],
&  A_2 = 
\left[ \begin{array}{ccc}
\displaystyle \alpha \underline{h} & 0 & \displaystyle \alpha \underline{h}\\
\displaystyle
0 & 0 & 0\\
0 & 0 & K
\end{array} \right],
\end{align*}
where we have denoted $\displaystyle a = \frac{s \rho_s}{e \underline{h} c_{sat}}$ and $\bar \rho_s = \frac{\rho_s}{c_{sat}}$.\\

We apply the Fourier transform in space and the equation verified by $\hat{f}$, the Fourier transform in space of $f$, is:
$$\dt \hat f = \left(A_0 + i \xi A_1 - (\xi^2 + \eta^2)A_2 \right) \hat{f} := A(\xi, \eta) \hat{f}.$$

Consequently, in order to study the stability of the system, we have to determine the sign of the real part of the eigenvalues of the matrix $A(\xi,\eta)$. These eigenvalues are denoted by $\lambda^i (\xi, \eta)$ with $i \in \{1,2,3\}$ and the associated eigenvectors are denoted by $V^i (\xi, \eta)$. The expressions for $\lambda^i$ are not explicit: in the next Subsection, we compute numerically the stability of the system. Then Subsection~\ref{stab RH} gives a stability result on the domain. This result is completed  with the asymptotic study of the eigenvalues at low ($|\xi|^2+|\eta|^2\ll 1$) and high ($|\xi|^2+|\eta|^2\ll 1$) frequencies.

\subsection{Numerical exploration of stability} \label{sec_cclstab}

In this section, we explore numerically the stability of the system~\eqref{eq_systlin}. For that purpose, we have computed numerically the three eigenvalues of the matrix $A(\xi,\eta)$, and the system is stable if and only if the real parts of these eigenvalues are negative. Since $A(\xi,-\eta)=A(\xi,\eta)$ and $A(-\xi,\eta)=\overline{A(\xi,\eta)}$, the real part of the spectrum remains unchanged under the transformation $\xi\mapsto-\xi$ and $\eta\mapsto-\eta$. Thus, we examine the behaviour of the system in the top right quarter of the plane.
The choice of parameters for these computations is the same as in Section~\ref{sec_simu} (unless otherwise specified), see Table~\ref{tab_param} for their values. 

In Figure~\ref{stabtr}, we have represented an illustration of the stability of the system when $K>0$, where
$$ \displaystyle K_e=\frac{5\times 10^{-4}}{3600}m^2s^{-1}. $$
The domain represented is a bounded subset of the plane $(\xi,\eta) \subset \RR^2$, and the color represent the stability: the system is stable in the green area, and unstable in the red area. We clearly see the stabilizing effect of the creep effect: when $K$ is higher, the stable area is larger. It seems that the unstable area is bounded. This will be confirmed by Proposition ~\ref{prop_stab_high}. 

\begin{figure}[h]
	\centering
	\begin{subfigure}[b]{0.32\textwidth}
          \centering
          \includegraphics[width=\textwidth]{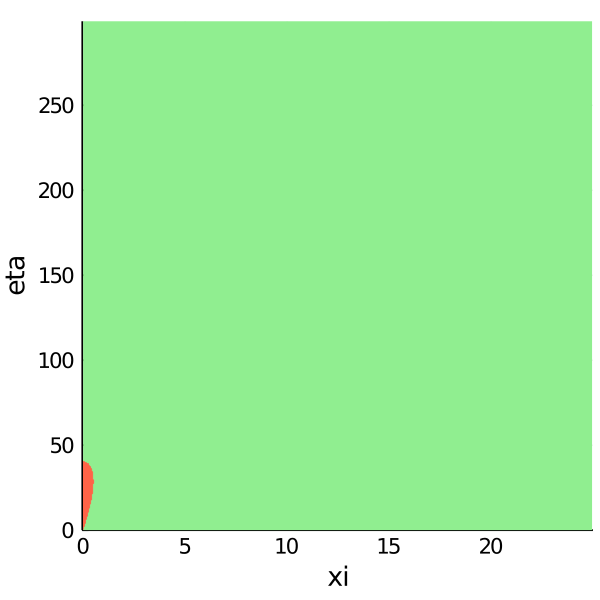}
          \caption{$K = K_e$}
          \label{stabtr_e}
     \end{subfigure}
     \begin{subfigure}[b]{0.32\textwidth}
          \centering
          \includegraphics[width=\textwidth]{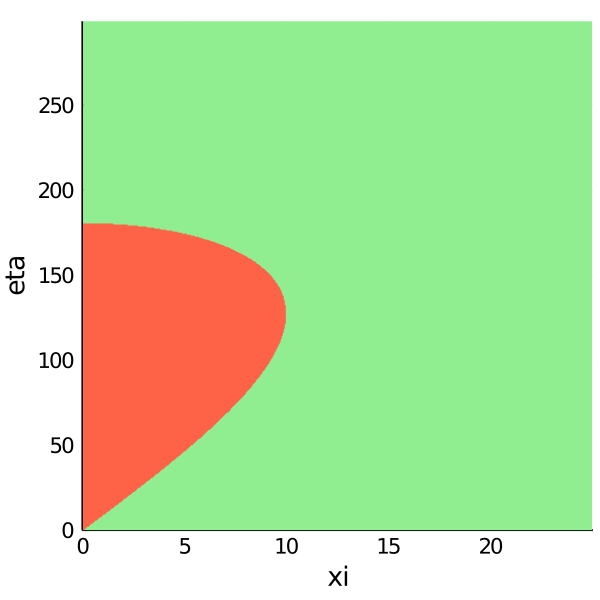}
          \caption{$K = K_e /20$}
          \label{stabtr_e/20}
     \end{subfigure}
     \begin{subfigure}[b]{0.32\textwidth}
          \centering
          \includegraphics[width=\textwidth]{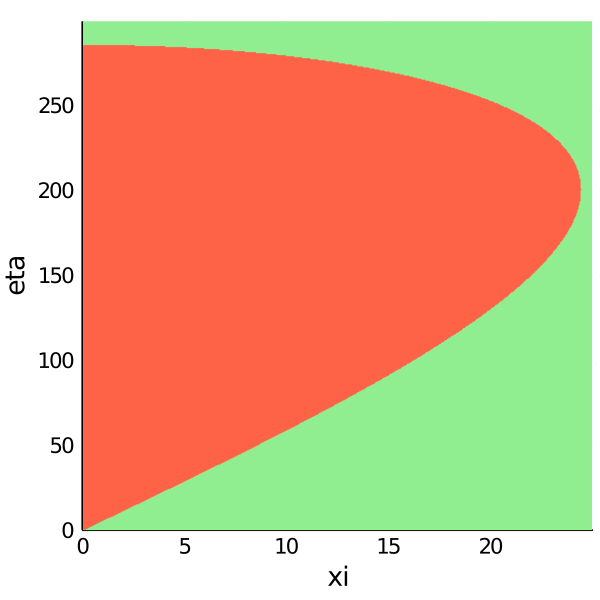}
          \caption{$K = K_e / 50$}
          \label{stabtr_e/50}
     \end{subfigure}
     \caption{Stability diagrams in the plane $(\xi,\eta) \subset \RR^2$, for $K>0$. The system is unstable in the red area, and stable in the green area.}
     \label{stabtr}
\end{figure}

Then, in Figure~\ref{stabK=0} there is no creep effect: $K=0$, and we represent various stability diagrams for several values of the ratio $n/m$. The value of $m$ is the same as in Table~\ref{tab_param}: $m = 1.6$, and $n$ vary between $m/2$ and $10m$. We observe that the behaviour of the system changes with the ratio $m/n$, but the unstable area always seems to be unbounded.

\begin{figure}[h!]
	\centering
     \begin{subfigure}[b]{0.24\textwidth}
          \centering
          \includegraphics[width=\textwidth]{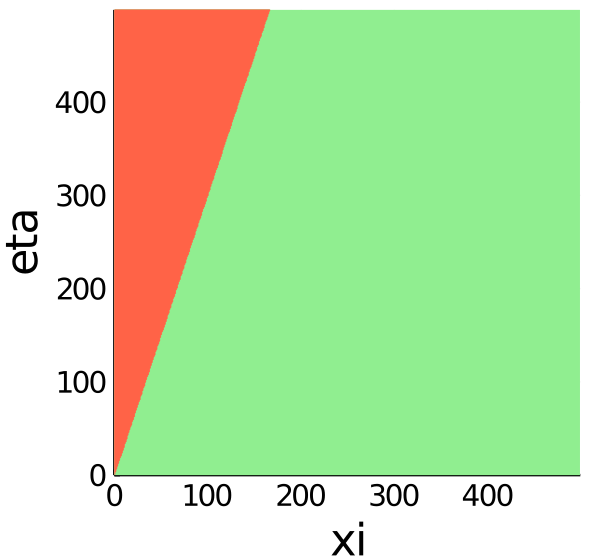}
          \caption{$n = 10m$}
          \label{stabK=0_n=10m}
     \end{subfigure}
     \begin{subfigure}[b]{0.24\textwidth}
          \centering
          \includegraphics[width=\textwidth]{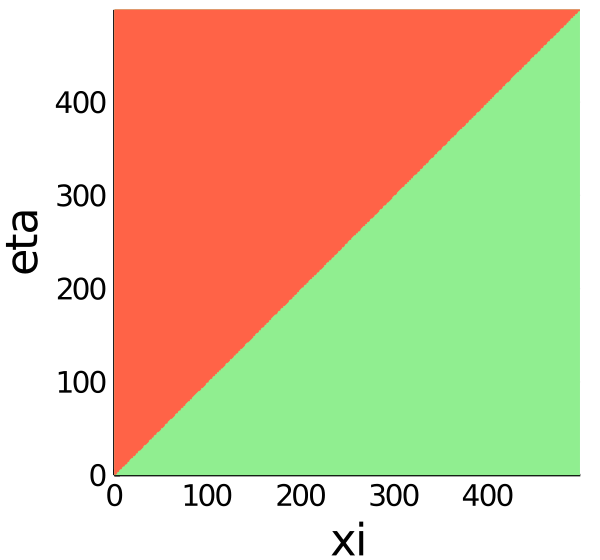}
          \caption{$n = 2m$}
          \label{stabK=0_n=2m}
     \end{subfigure}
     \begin{subfigure}[b]{0.24\textwidth}
          \centering
          \includegraphics[width=\textwidth]{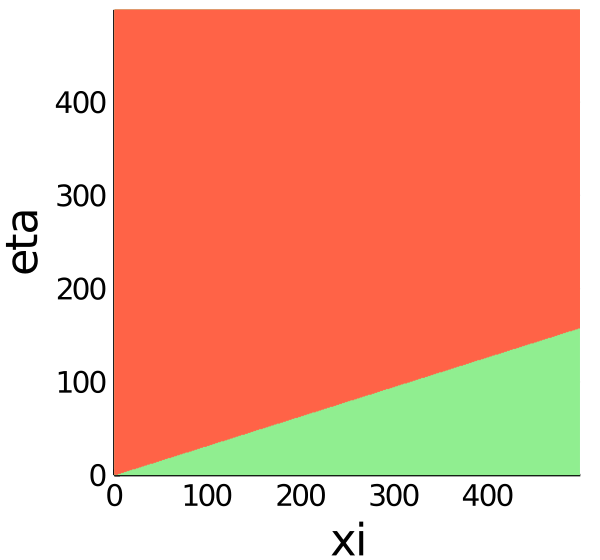}
          \caption{$n = 1.1 m$}
          \label{stabK=0_n=1.1m}
     \end{subfigure}
     \begin{subfigure}[b]{0.24\textwidth}
          \centering
          \includegraphics[width=\textwidth]{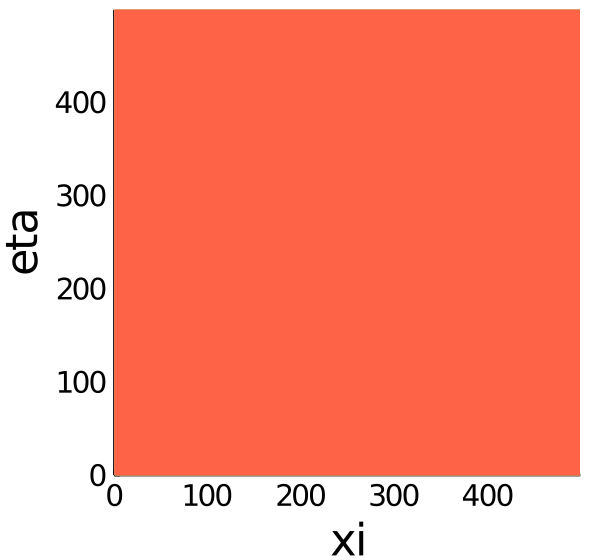}
          \caption{$n = m/2$}
          \label{stabK=0_n=m/2}
     \end{subfigure}
     \caption{Stability diagrams in the plane $(\xi,\eta) \subset \RR^2$, for $K=0$ and various values of $n/m$. The system is stable in the green area and unstable in the red area.}
     \label{stabK=0}
\end{figure}

When $n\gg m$ as in Figure~\ref{stabK=0_n=10m}, the stable area is bigger than the unstable area, and the system is unstable only for perturbations of transverse dominant direction. When $n\approx m$, as in Figure~\ref{stabK=0_n=1.1m}, the system is stable only for longitudinal perturbations. Finally, when $n<m$ the system seems to be unstable at all frequencies, as in Figure~\ref{stabK=0_n=m/2}.\\

The stability exploration of the system~\eqref{eq_systlin} should be quantified by theoretical results, in order to identify the types of instabilities and the transitions between stability and instability. This is done in the following Section.

\subsection{Stability analysis} \label{stab RH}

The following theorem provides a necessary and sufficient condition for the stability of the system, depending on the parameters of the model.
\begin{theorem}
     If $K >0$, there exists a constant $\gamma > 0$ independent of $n$ such that System~\eqref{eq_systlin} is spectrally stable at all frequencies $(\xi,\eta)  \in \RR_*^2$ if and only if 
     $$K  \geq \alpha m \underline{h} \underline{c}/ \bar \rho_s \quad \text{and} \quad n < \gamma.$$
     If $K = 0$, System~\eqref{eq_systlin} is spectrally unstable. More precisely, for $(\xi,\eta)  \in \RR_*^2 := \RR^2 \setminus \{(0,0)\}$, the system is stable if and only if 
     $$m < \frac{\bar \rho_s}{\underline{c}} \quad \text{and} \quad \eta^2 m < \xi^2 (n-m).$$ In particular, if $n \leq m$ then the system is always unstable.
     \label{th_stabDomain}
\end{theorem}

The proof of Theorem~\ref{th_stabDomain} is given in Appendix~\ref{app_proof th stab} by using Routh-Hurwitz criterion.\\

This theorem confirms the prominent role of the creep effect that was already observed for the well-posedness property of System~\eqref{eq_WPsyst} since Theorem~\ref{th} required $K$ to be sufficiently large. Here, it plays a crucial role in the stability of the system. Indeed, the condition $K\geq \alpha m \underline{h} \underline{c}/ \bar \rho_s$ is verified only if the creep effect on the soil is large enough. Therefore, this effect has a stabilising effect on the system, as observed in the previous section where the unstable region shrinks to $0$ as the constant $K>0$ increases. 

Conversely, the two conditions of Theorem~\ref{th_stabDomain} mean that the constants $\alpha$, $m$ and $n$ have a destabilising effect on the system. Recall that $\alpha$ measures the ratio between the erosion speed and the water speed. When $\alpha$ is large, the erosion speed is large compared to the fluid velocity which increases the instability of the bottom surface.

When $K=0$, the constant $m$ still has a destabilising effect. The other condition implies that arbitrary high frequencies are unstable.
When $n<m$ then the system is unstable on the whole domain $\RR^2$. When $n \geq m$, the system is stable in an unbounded area of the spectrum  : 
$$\left\{(\xi,\eta) \in \RR^* \times \RR \, ; \, m \leq \frac{\xi^2}{\xi^2+\eta^2}n \right\}.$$
This area is delimited by the lines of equation 
$$
\displaystyle
\eta = \pm \sqrt{\frac{n-m}{m}} \xi,
$$ 
thus increases with $n$: this confirm the observation made in Figure~\ref{stabK=0}. 
In particular, if $n \gg m$ then the system destabilises only if $\xi \ll \eta$, that is for perturbations transverse to the water flux. Moreover, as long as $n \geq m$, the system is unstable in the transverse direction ($\xi = 0$). This may lead to the formation of rills in the direction of the water flux. \\

Theorem~\ref{th_stabDomain} provides a stability criterion but does not give any insight on the nature of the instabilities. The following propositions determine more precisely the spectrum in the small wavenumber regime $\xi^2+\eta^2\ll 1$ and in the large wavenumber regime $\xi^2+\eta^2\gg 1$.
Since the system is spectrally stable if $K \geq \alpha m \underline{h} \underline{c} / \bar \rho_s$, we focus on the case 
$$
\displaystyle
K< \frac{\alpha m \underline{h} \underline{c}}{\bar \rho_s}
$$
in order to detect low frequency instabilities. The following proposition provides a more complete picture of the stability diagram when $K < \alpha m \underline{h} \underline{c} / \rho_s$ in the limit $\xi^2+\eta^2\to 0$. The expansion of the eigenvalues depends on the relative position of $\xi$ and $\eta^2$. We split the analysis into two cases: $\xi = O(\eta^2)$ and $\eta^2 = o(\xi)$.

\begin{proposition}[Instabilities at low frequencies]
     Assume $K < \alpha m \underline{h} \underline{c} / \bar \rho_s$. In the limit $|\xi|+|\eta|\to 0$, System~\eqref{eq_systlin} has the following stability properties:
     \begin{itemize}
          \item When $\xi = O(\eta^2)$ then:
          \begin{itemize}
               \item under the assumption $\alpha m \underline{h} \underline{c} / \bar \rho_s - \alpha \underline{h} < K < \alpha m \underline{h} \underline{c} / \bar \rho_s$, the system is stable if and only if $$\xi^2 < \frac{K}{\bar \rho_s \tan^2 \theta} \frac{(K + \alpha \underline{h} - \alpha m \underline{h} \underline{c})^2}{\alpha m \underline{h} \underline{c} - K} \eta^4.$$
               \item If $K \leq \alpha m \underline{h} \underline{c} / \bar \rho_s - \alpha \underline{h}$, the system is unstable for all these frequencies.
          \end{itemize}
          \item When $\eta^2 = o(\xi)$ then the system is stable if and only if $$\left( K \bar \rho_s + \alpha \underline{h} \underline{c} (n-m) \right) \xi^2 > \left( \alpha m \underline{h} \underline{c} - K \bar \rho_s \right) \eta^2.$$
          In particular, if $K < \alpha (m-n) \underline{h} \underline{c} / \rho_s$ the system is unstable at these frequencies. Moreover the most unstable eigenvalue expands as $$\lambda^1 \underset{\xi^2 + \eta^2 \to 0}{=} - \left(K+\frac{c_{sat}}{\rho_s} \alpha n\underline{h}\underline{c} - \frac{c_{sat}}{\rho_s} \alpha m\underline{h}\underline{c} \right) \xi^2 - \left(K-\frac{c_{sat}}{\rho_s}  \alpha m\underline{h}\underline{c} \right) \eta^2 + o \left(\xi^2 + \eta^2 \right).$$
     \end{itemize}
     \label{prop_lowfr2}
\end{proposition}

\begin{remark}
    This proposition means that the system is stable in all directions at low frequencies if and only if $K \geq \alpha m \underline{h} \underline{c} / \bar \rho_s$, and we recover the global criterion.
\end{remark}

The consequences of Proposition ~\ref{prop_lowfr2} are fully discussed in Section~\ref{sec_ccl_stablow}. However, we can make some preliminary comments. We first note that when $K \to \alpha m \underline{h} \underline{c} / \bar \rho_s$ then 
          $$\frac{K}{\bar \rho_s \tan^2 \theta} \frac{(K + \alpha \underline{h} - \alpha m \underline{h} \underline{c})^2}{\alpha m \underline{h} \underline{c} - K} \to +\infty \quad \text{and} \quad  \frac{\alpha m \underline{h} \underline{c} - K \bar \rho_s}{K \bar \rho_s + \alpha \underline{h} \underline{c} (n-m)} \to 0.$$
          
Consequently, the stable area (the set of frequencies $(\xi,\eta)$ such that the system is stable at these frequencies) increases with $K$, until filling all the low frequencies.
Conversely, when $K \to 0$ the stable area decreases, up to the area described in the second part of Theorem~\ref{th_stabDomain}, when $K=0$.

In the unstable regime $K<\alpha m\underline{h}\underline{c}/\bar{\rho}_s$, we can precise the instability scenario. In the  case $\xi = O(\eta^2)$, the behaviour of the function $f$ defined as
$$
\displaystyle
f(K) := K \frac{(\alpha \underline{h}+ K-\alpha m \underline{h} \underline{c}/ \bar \rho_s)^2}{\alpha m \underline{h} \underline{c}/ \bar \rho_s- K}
$$ 
is described in Figure~\ref{fig_f}. On the interval $\displaystyle\left[m\underline{h}\underline{c} / \bar \rho_s - \alpha \underline{h}, \alpha m\underline{h}\underline{c} / \bar \rho_s \right]$, the function $f$ is increasing from $0$ to $+\infty$, thus for $\xi$, $\eta$ fixed there exists a unique $\bar K(\xi,\eta)$ such that $$\xi^2 = \frac{1}{\bar \rho_s \tan^2 \theta} f(\bar K(\xi,\eta)) \eta^4.$$ If $K \leq \bar K(\xi,\eta)$ then the system is unstable at this frequency $(\xi,\eta)$, and if $K > \bar K(\xi,\eta)$ then the system is stable at this frequency.
  
\begin{figure}[ht]
	\center
	\begin{subfigure}{0.32 \textwidth}
		\center
		\begin{tikzpicture}
			\tkzTabInit[lgt = 1.2, espcl=3]{$K$ / 1, $f(K)$ / 2}{$0$, $\alpha \frac{m\underline{h}\underline{c}}{\bar \rho_s}$}
			\tkzTabVar{-/ $0$, +/ $+\infty$}
		\end{tikzpicture}
		\caption{$m\underline{c} \leq \rho_s$}
		\label{mc<rho_s}
	\end{subfigure} \quad
	\begin{subfigure}{0.57 \textwidth}
		\center
		\begin{tikzpicture}
			\tkzTabInit[lgt = 1.2, espcl=2.5]{$K$ / 1, $f(K)$ / 2}{$0$, $K_0$, $\alpha \frac{m\underline{h}\underline{c}}{\bar \rho_s} - \alpha \underline{h}$, $\alpha \frac{m\underline{h}\underline{c}}{\bar \rho_s}$}
			\tkzTabVar{-/ $0$, +/ $f(K_0)$, -/ $0$, +/ $+\infty$}
		\end{tikzpicture}
		\caption{$m\underline{c} > \rho_s$}
		\label{mc>rho_s}
	\end{subfigure}
	\caption{Variations of the function $f$}
	\label{fig_f}
\end{figure}
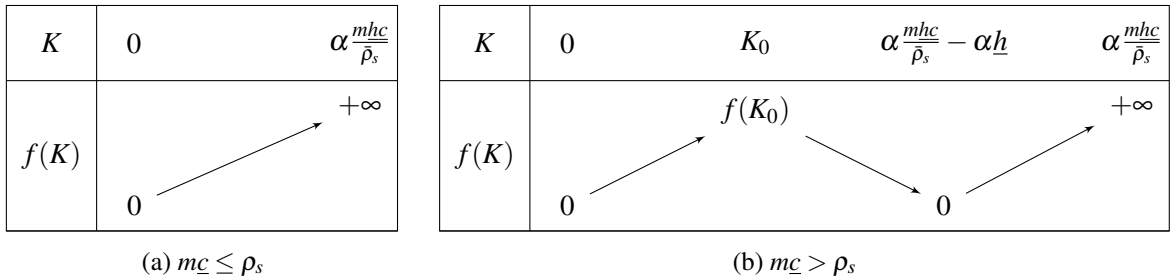

In order to prove Proposition ~\ref{prop_lowfr2}, we compute an asymptotic expansion of the eigenvalues $\lambda^i(\xi, \eta), \, i=1,2$ of the matrix $A(\xi,\eta)$, as $|\xi|+|\eta|$ goes to $0$. Then, we study the sign of their real parts. Note that $A(0,0)=A_0$, which corresponds to homogeneous in space perturbations, admits $\lambda^1 = \lambda^2 = 0$ and $\lambda^3 = -a$ as eigenvalues. We focus on the expansion of the eigenvalues $\lambda^1,\lambda^2$ bifurcating from $0$ as the third one bifurcates from $\lambda^3(0,0)=-a<0$ and its real part remains negative for $|\xi|+|\eta|$ small enough. The matrix $A(0,0)$ is diagonalizable and the eigenvalue $0$ is semi-simple (its algebraic multiplicity, $2$, is equal to its geometric multiplicity). Therefore the eigenvalues $\lambda^i(\xi,\eta)$ admit a Taylor expansion with respect to the perturbation parameters $\xi$ and $\eta^2$ (see \cite{andrew1993derivatives}). These expansions depends heavily on the ratio $\xi/\eta^2$: we split the analysis between the case $|\xi|/\eta^2$ bounded by a constant ,which is studied in Section~\ref{stab low1}, and the case $|\xi|^2/\eta^2 \to +\infty$, which is studied in Section~\ref{second case}.
The details of the proof of Proposition~\ref{prop_lowfr2} can be found in Appendix~\ref{app_stab_low}. \\

The following proposition determine the stability of System~\eqref{eq_systlin} in the  limit $\xi^2 + \eta^2 \to +\infty$. 

\begin{proposition}[Stability analysis at high frequencies]
The stability results for the system at high frequencies are divided in three cases:
\begin{itemize}
    \item When $K > 0$ then System~\eqref{eq_systlin} is stable at high frequencies. The eigenvalues expand as
     \begin{align*}
          \left\{ \begin{array}{l}
          \displaystyle
          \lambda_1(\xi,\eta)=-K(\xi^2+\eta^2)+o(\xi^2+\eta^2)\\
          \displaystyle
          \lambda_2(\xi,\eta)=-\alpha\underline{h}(\xi^2+\eta^2)+o(\xi^2+\eta^2)\\
          \displaystyle
          \lambda_3(\xi,\eta)=-i\xi\tan(\theta)-\frac{s\bar{\rho}_s}{e\underline{h}}+o(1).
          \end{array}\right.
     \end{align*}
    \item If $K=0$ and $\xi \neq 0$, the Taylor expansion of the eigenvalues when $\xi^2 + \eta^2 \to + \infty$ is given by:
     \begin{align*}
          \left\{ \begin{array}{l}
          \displaystyle
          \lambda_1(\xi,\eta) = \frac{s \underline{c}}{e \underline{h}} \left(m - \frac{\xi^2}{\epsilon^2}n \right) + o(1), \vspace{0.2cm}\\
          \displaystyle
          \lambda_2(\xi,\eta) = -\alpha \underline{h}(\xi^2+\eta^2) + o(\xi^2 + \eta^2), \vspace{0.3cm}\\
          \displaystyle
          \lambda_3(\xi,\eta) =  -i\xi \tan \theta - \frac{s \bar \rho_s}{e \underline{h}} + o(1).
          \end{array} \right.
     \end{align*}
     \item If $K=0$ and $\xi = 0$, the Taylor development of the eigenvalues when $\eta \to + \infty$ is:
     \begin{align*}
          \left\{ \begin{array}{l}
          \displaystyle
          \lambda_1(0,\eta) = 0, \vspace{0.3cm}\\
          \displaystyle
          \lambda_2(0,\eta) = -\alpha \underline{h} \eta^2 + o(\eta^2), \vspace{0.1cm}\\
          \displaystyle
          \lambda_3(0,\eta) =  \frac{s \underline{c}}{e \underline{h}} \left( m- \frac{\bar \rho_s}{\underline{c}} \right) + o(1).
          \end{array} \right.
     \end{align*}
\end{itemize}
\label{prop_stab_high}
\end{proposition}

The proof of this proposition is postponed in Appendix~\ref{app_stab_high}.\\

As a consequence of Proposition \ref{prop_stab_high}, one finds that when $K >0$, the unstable domain in the spectrum is bounded, as observed in Section~\ref{sec_ccl}. Thus, there exists a wave vector associated to an unstable eigenvalue with a maximum real part which may provide a description of the pattern geometry. When $K=0$, the unstable region is unbounded but there is also a most unstable eigenvalue with its real part bounded by $\frac{s \underline{c}}{e \underline{h}}m$: this is obtained when $\xi=0$ for pure transverse perturbations.

\subsection{Discussion on the low frequencies stability analysis: form of the spectrum} \label{sec_ccl_stablow}


The stability result on the whole domain Theorem~\ref{th_stabDomain} does not give an explicit formula for the localisation of the stable and unstable areas when $K>0$. Therefore in this section we interpret the stability analysis at low frequencies Proposition ~\ref{prop_lowfr2} to study the form of the limit between stable and unstable areas in the spectrum, at these frequencies, for $K>0$. As the real part of the spectrum remains unchanged under the transformations $\xi\to-\xi$ and $\eta\mapsto-\xi$, we examine the behaviour of the system in the top right quarter of the plane. When $K \geq \alpha m \underline{h} \underline{c} / \bar \rho_s$, the system is stable at all low frequencies, thus we focus on the case $K < \alpha m \underline{h} \underline{c} / \bar \rho_s$.\\ 

We first examine the case of the frequencies which verify $\eta^2 = o(\xi)$. If $K \bar \rho_s \leq \alpha (m-n) \underline{h} \underline{c}$, the system is unstable at these frequencies. When $\alpha (m-n) \underline{h} \underline{c} / \bar \rho_s \leq K \leq \alpha m \underline{h} \underline{c} / \bar \rho_s$, the instability comes from the term in $\eta$ of the first eigenvalue, and:
\begin{align*}
     \displaystyle
     \lambda^1 & \geq 0
     \quad \Longleftrightarrow \quad \frac{\eta^2}{\xi^2} \geq \frac{K \bar \rho_s - \alpha (m-n) \underline{h} \underline{c}}{\alpha m \underline{h} \underline{c} - \bar K \rho_s} + o \left(\xi^2 + \eta^2 \right).
\end{align*}

\begin{figure}[h!]
	\centering
     \begin{subfigure}[b]{0.3\textwidth}
          \centering
          \includegraphics[width=\textwidth]{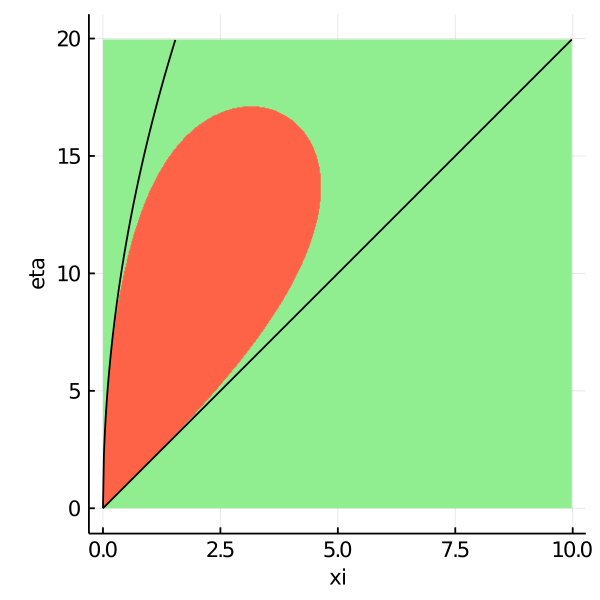}
          \caption{$K = 0.6 \alpha m \underline{h} \underline{c}/\rho_s$}
          \label{stablow21}
     \end{subfigure}
     \begin{subfigure}[b]{0.3\textwidth}
          \centering
          \includegraphics[width=\textwidth]{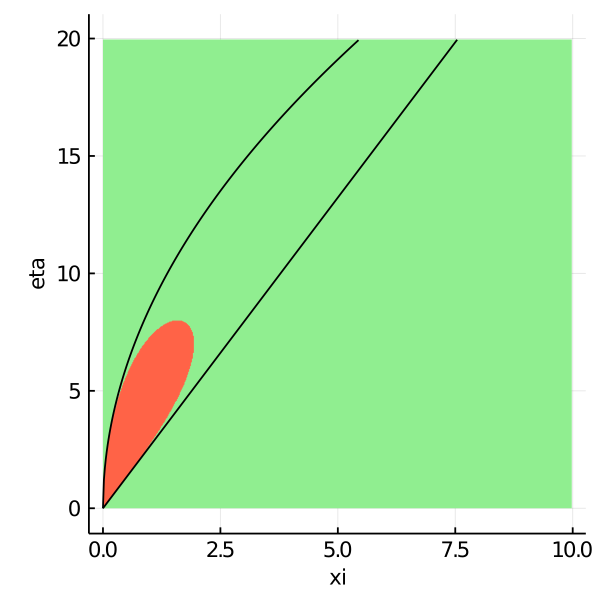}
          \caption{$K = 0.75 \alpha m \underline{h} \underline{c}/\rho_s$}
          \label{stablow22}
     \end{subfigure}
     \begin{subfigure}[b]{0.3\textwidth}
          \centering
          \includegraphics[width=\textwidth]{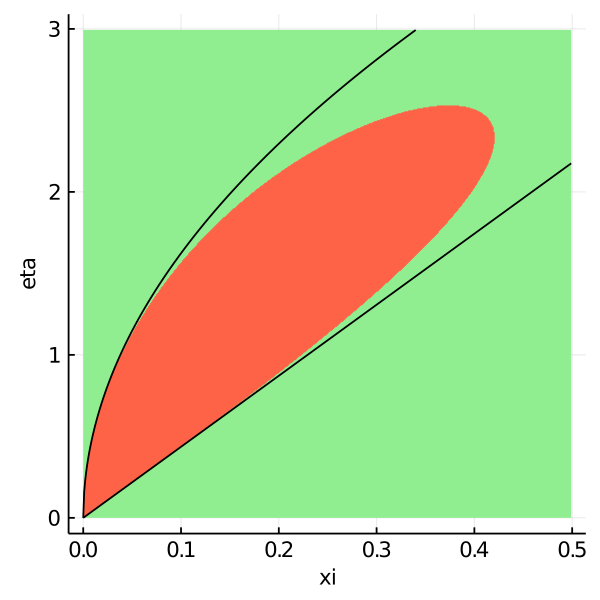}
          \caption{$K = 0.9 \alpha m \underline{h} \underline{c}/\rho_s$, lower scale}
          \label{stablow23}
     \end{subfigure}
     \caption{Stability of the system at low frequencies in the plane $(\xi,\eta)$, when $K > \alpha m \underline{h} \underline{c} / \bar \rho_s - \alpha \underline{h}$. The system is stable in the green area and unstable in the red area. The black curves are the line of slope~\eqref{eq_slope} and the curve of Equation~\eqref{eq_curve}}.
     \label{fig_stablow1}
\end{figure}

Consequently the boundary between the stable and unstable areas of the system is close to a straight line, of slope 
\begin{align}
\frac{\eta}{\xi}=\sqrt{\frac{K \bar \rho_s - \alpha (m-n) \underline{h} \underline{c}}{\alpha m \underline{h} \underline{c} -K \bar \rho_s}}.
\label{eq_slope}
\end{align}
The system is stable below this line, and unstable above. The other case corresponds to the frequency domain $\xi = O(\eta^2)$. If $K \leq \alpha m \underline{h} \underline{c} / \bar \rho_s - \alpha \underline{h}$ the system is unstable. When $\alpha (m-n) \underline{h} \underline{c} / \bar \rho_s \leq K \leq \alpha m \underline{h} \underline{c} / \bar \rho_s$ the destabilizing effect is given by the first eigenvalue, and
\begin{align*}
     \displaystyle
     \lambda^1 & \geq o \left(\xi^2 + \eta^2 \right)
     \Longleftrightarrow \frac{\xi^2 \tan^2 \theta}{\eta^4} \leq \frac{K}{\bar \rho_s}  \frac{\left( K \bar \rho_s + \alpha \underline{h} \bar \rho_s - \alpha m \underline{h} \underline{c} \right)^2}{\alpha m \underline{h} \underline{c} -K \bar \rho_s } + o \left(\xi^2 + \eta^2 \right).
\end{align*}

Therefore the boundary between stable and unstable area in this case is close to the curve of equation 
\begin{align}
     \eta = \left( \frac{\bar \rho_s}{K}  \frac{\alpha m \underline{h} \underline{c} -K \bar \rho_s }{\left( K \bar \rho_s + \alpha \underline{h} \rho_s - \alpha m \underline{h} \underline{c} \right)^2}\right)^{1/4} \sqrt{\tan \theta \, \xi}.
     \label{eq_curve}
\end{align}

Consequently, when $$\min (\alpha m \underline{h} \underline{c} / \bar \rho_s - \alpha \underline{h}, \alpha (m-n) \underline{h} \underline{c} / \bar \rho_s) \leq K \leq \alpha m \underline{h} \underline{c} / \bar \rho_s,$$ the system is stable below the line of slope given by~\eqref{eq_slope} and above the curve of Equation~\eqref{eq_curve}, and unstable between these curves. This situation is illustrated in Figure~\ref{fig_stablow1}, for various values of $K$. The black curves are the boundary curves between stable and unstable areas calculated above, we can see that they fit the calculations. The unstable area are  bounded, and this is confirmed by the stability analysis at high frequencies below. Moreover, as $K$ increases, the unstable area decreases.

\begin{figure}[h!]
	\centering
     \begin{subfigure}[b]{0.3\textwidth}
          \centering
          \includegraphics[width=\textwidth]{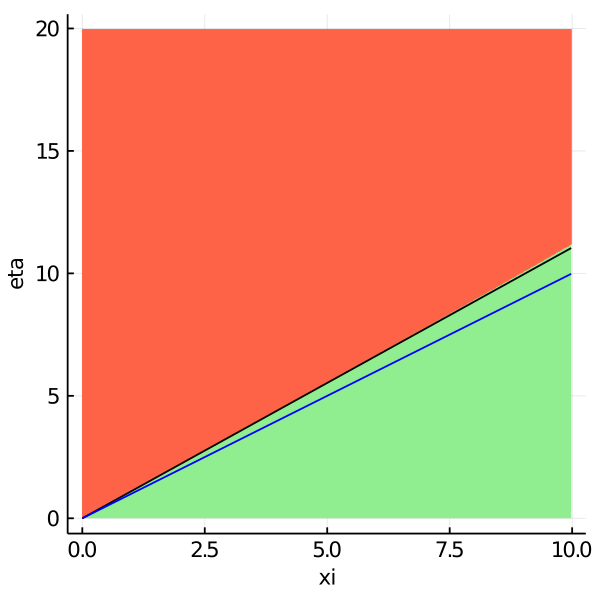}
          \caption{$K = 0.1 \alpha m \underline{h} \underline{c}/\rho_s$}
          \label{stablow11}
     \end{subfigure}
     \begin{subfigure}[b]{0.3\textwidth}
          \centering
          \includegraphics[width=\textwidth]{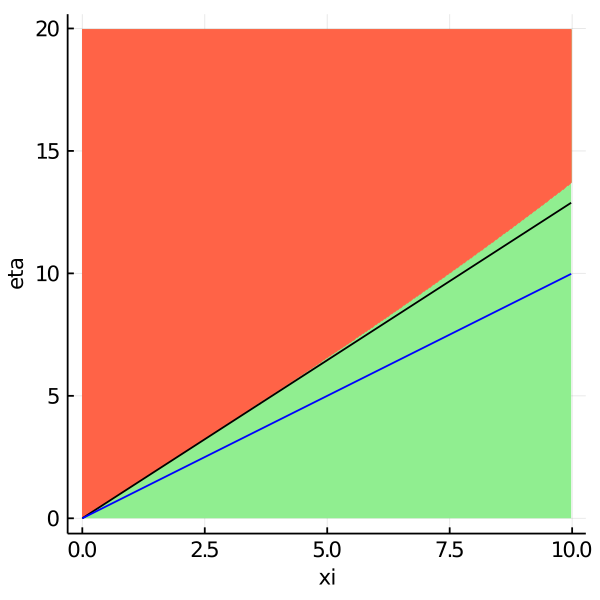}
          \caption{$K = 0.25 \alpha m \underline{h} \underline{c}/\rho_s$}
          \label{stablow12}
     \end{subfigure}
     \begin{subfigure}[b]{0.3\textwidth}
          \centering
          \includegraphics[width=\textwidth]{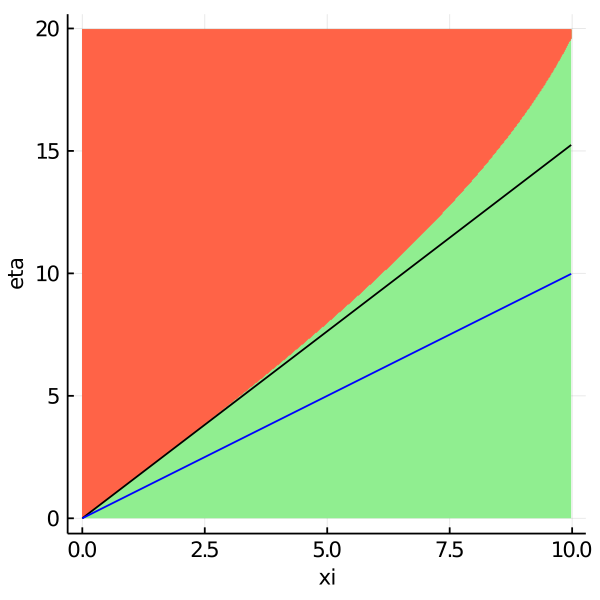}
          \caption{$K = 0.4 \alpha m \underline{h} \underline{c}/\rho_s$}
          \label{stablow13}
     \end{subfigure}
     \caption{Stability of the system at low frequencies in the plane $(\xi,\eta)$, when $K < \alpha m \underline{h} \underline{c} / \bar \rho_s - \alpha \underline{h}$. The system is stable in the green area and unstable in the red area. The black line is the line of slope~\eqref{eq_slope}, the blue line is the line of slope $\sqrt{\frac{n-m}{m}}$ (the boundary line when $K=0$).}
     \label{fig_stablow2}
\end{figure}

If $K$ is smaller than $\alpha m \underline{h} \underline{c} / \bar \rho_s - \alpha \underline{h}$ or $\alpha (m-n) \underline{h} \underline{c} / \bar \rho_s$ then one of the curves (or both) of equation~\eqref{eq_slope} and~\eqref{eq_curve} disappear, and the instability area increases. Figure~\ref{fig_stablow2} illustrates the case $\alpha (m-n) \underline{h} \underline{c} / \bar \rho_s < K < \alpha m \underline{h} \underline{c} / \bar \rho_s - \alpha \underline{h}$, for various values of $K$. The only boundary between the stable and unstable areas is the line of slope~\eqref{eq_slope}. The system is stable below this line and unstable above this line. As $K$ decreases the unstable area increases, and comes closer to the line of slope $\sqrt{\frac{n-m}{m}}$ which is the boundary when $K=0$, given in Theorem~\ref{th_stabDomain}.

\section{Direct numerical simulations} \label{sec_simu}

In this section, we present some numerical experiments of the erosion of a tilted plane. The parameters of these experiments come from physical data. The quantities $L_x$, $L_y$, $V$, $\underline{h}$, $\rho_s$, $c_{sat}$, $e$ and $\theta$ are chosen according to the laboratory experiment \cite{guerin_streamwise_2020}, which erodes a block of salt. The choice of values for the exponents $m$ and $n$ has been investigated many times in the literature. The values are chosen between $0$ and $3$, with an additional relation $n = 2m$, as  explained in \cite{chen_landscape_2014} (here the constant $n$ corresponds to the constant $m+n$ in the literature). We choose $n = 2m$, with $n$ sufficiently large in order to observe the formation of channels in the simulations. Indeed, we found that the effect of digging in depressions is reinforced when these exponents are larger.
The choice of parameters is given in Table~\ref{tab_param}.\\

\begin{table}[ht]
\center 
\begin{tabular}{| l| l | l |}
     \hline
     Length of the domain & $L_x$ & $40$ $cm$\\
     Width of the domain & $L_y$ & $10$ $cm$\\
     Characteristic water speed & $ V = \mu \tan \theta$ & $1$ $m/s$\\	
     Initial water height & $h_0$ & $0.5$ $mm$\\
     Initial sediment concentration & $c_0$ & $317$ $g/m^3$\\
     Exponent of friction over $h$ & $m$ & $1,6$\\
     Exponent of friction over $v$ & $n$ & $3,2$\\
    Density of the sediments & $\rho_s$ & $2.17 \times 10^6$ $g/m^3$ \\
    Concentration of saturation & $c_{sat}$ & $3.17 \times 10^5$ $g/m^3$ \\
     Erosion speed & $e$ & $0.5$ mm/hour\\
     Sedimentation speed & $s$ & $e/2000$ \\
     Angle of the plane & $\theta$ & $39^\circ$\\
     \hline  
\end{tabular}
\caption{Parameters of numerical simulations}
\label{tab_param}
\end{table}

Once the ratio between the erosion speed and the water speed has been fixed, we assume that the time variations of the water height and concentration are small, consequently we neglect these variations in the simulation. Indeed, in the simulation, the eroded height of the soil is of the order of a millimeter, thus the characteristic time $T = Z/e$ is of the order of an hour. The water crosses the domain in $0.4$ seconds, thus there are four orders of magnitude between the characteristic time of the water flow and that of erosion. A direct numerical simulation of the full system, with time derivatives, would impose a severe CFL restriction: indeed, the fluid velocity is about $1m.s^{-1}$ whereas the erosion rate is around $1mm.h^{-1}$. Since we are interested in the erosion process, the natural time scale is one hour and a typical time step would be a minute. However, the numerical time step $\delta t$ is driven by a CFL: $\delta t\leq \delta x/v_{water}\approx 0.001 s$ if one consider a  typical mesh size $\delta x=1mm$ (for a channel of length $400mm$). This increases the numerical cost of the scheme. Instead, we solve stationary problems at each time steps, where the time step is determined by the erosion time scale. A comparison between the two resolution methods is made in Appendix \ref{appendix_scheme}, it shows that taking the stationary equation does not affect the results.

The stationary equations for the water height and concentration of sediments in water are discretised with a finite volume method. The scheme is given in Appendix~\ref{appendix_scheme}. The equation~\eqref{eq_systcomplet_c} on $c$ is a linear equation, and we discretise it with an explicit Euler scheme by considering it as an evolution equation with respect to the variable $x$, as the speed in this direction does not vanish:
\begin{align}
     \dx c + \frac{v_y}{v_x} \dy c = \frac{\rho_s}{h} (E-S).
     \label{eq_c_lin}
\end{align}
The equation~\eqref{eq_systcomplet_h} on $h$ is non linear, thus it is harder to discretise it. In order to avoid an implicit discretisation of this equation, we made the choice to linearise the equation. Denoting by $h^n$ the solution of the equation at time $t^n$, we approximate: 
$$\dive(h\nabla h)(t_n) + \dive (h\nabla z)(t_n) \approx \dive(h^{n-1} \nabla h^n) + \dive (h^n \nabla z^n).$$

As the solution at the previous time step is known, the right hand term is linear in $h^n$. Thus we can discretise it with a finite volume scheme in two dimensions. The justification for the quasi-stationary model and the numerical scheme are given in Appendix~\ref{appendix_scheme}.\\

We have chosen periodic boundary conditions in the transverse direction, therefore in the stability analysis there are only countable frequencies in this direction. The frequencies are the $\eta_n = 2 \pi n / L_y$ where $L_y$ is the width of the domain.  The boundary conditions at the top of the domain are Dirichlet condition for $h$, $c$ and $z$. At the bottom of the domain, z is prescribed by a Dirichlet condition and we suppose that water flows freely. Thus we fix Neumann condition for $h$ and $c$.\\

In the simulations, the initial surface is a flat tilted plane with a small random perturbation. This surface is represented in Figure~\ref{sol_init}, and it shows a flat view of the two dimensional plane. The color scale show the height difference between the actual soil, and the flat plane. Thus the yellow areas are the less dug parts.\\ 

\begin{figure}[h!]
     \centering
     \includegraphics[width=0.49\textwidth]{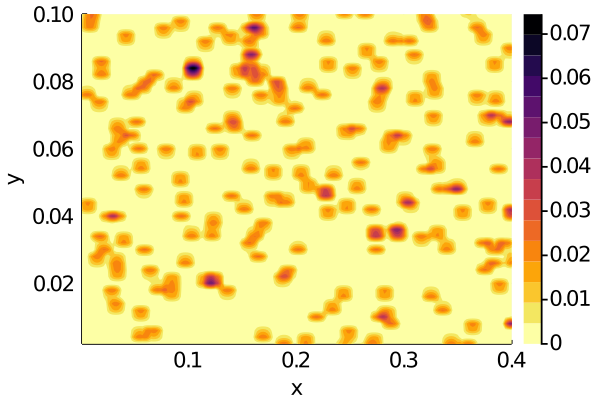}
     \caption{Perturbation of the initial surface, on the domain $[0,L_x] \times [0,L_y]$. The represented quantity is the difference between the flat surface and the perturbed soil, in millimeters. This quantity is zero at yellow points, and increase when color is darker.}
     \label{sol_init}
\end{figure}

\subsection{Simulations without source term for water} 

In this part, we present some results of simulations when $r=0$, as in the spectral stability analysis. These simulations are compared to the theoretical results of stability.

\begin{figure}[h!]
	\centering
     \begin{subfigure}[b]{0.49\textwidth}
          \centering
          \includegraphics[width=\textwidth]{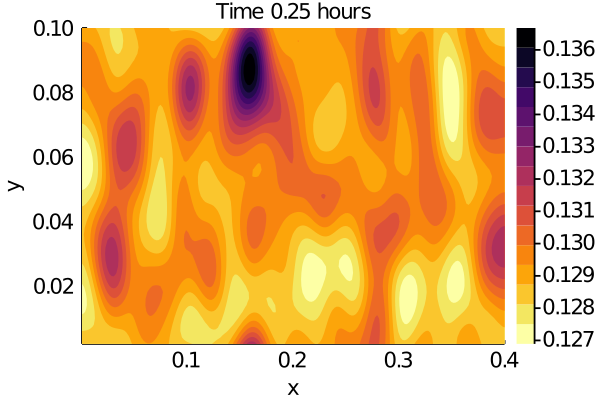}
          \caption{$K=K_e$, $T=0.25$}
          \label{e1}
     \end{subfigure}
     \begin{subfigure}[b]{0.49\textwidth}
          \centering
          \includegraphics[width=\textwidth]{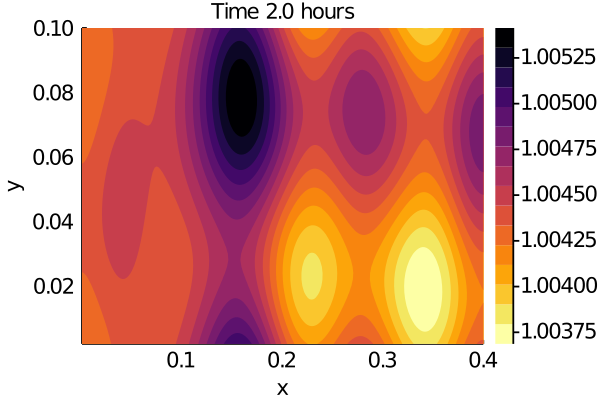}
          \caption{$K=K_e$, $T=2$}
          \label{e2}
     \end{subfigure}
     \caption{Eroded height of the soil in $mm$. The darkest areas are the most eroded areas.}
     \label{simu_stable}
\end{figure}

First, in Figure~\ref{simu_stable} we presents results of the simulation of the system, when 
$$
\displaystyle
K=K_e=\frac{5\times 10^{-4}}{3600}m^2s^{-1}.
$$ 
In this case, the system is spectrally stable at all frequencies $\{ (\xi,\eta_n); \, \xi \in \RR, \, n \in \NN \}$. The pictures represent the eroded height $z(T)-z_0$ at time $T=0.25$ hour and $T=2$ hours. We observe that surface perturbations are quickly smoothed, and tends to disappear. After $15$ minutes, we see in Figure~\ref{e1} that the amplitude of perturbation has not decreased yet, but these perturbations are smoother than initially. After $2$ hours, we see in Figure~\ref{e2} that the amplitude of the perturbations has decreased. \\

\begin{figure}[h!]
	\centering
     \begin{subfigure}[b]{0.49\textwidth}
          \centering
          \includegraphics[width=\textwidth]{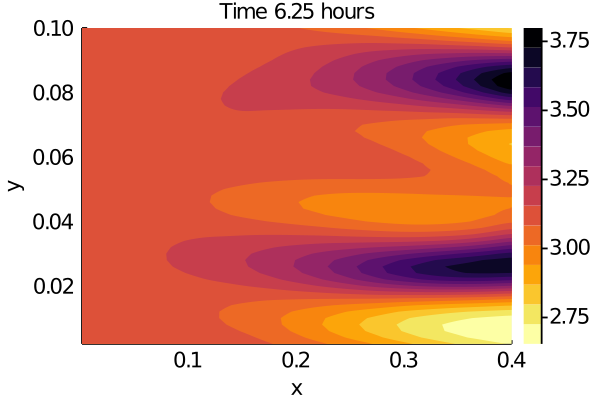}
          \caption{$\displaystyle K = \frac{K_e}{20}$}
          \label{e/20}
     \end{subfigure}
     \begin{subfigure}[b]{0.49\textwidth}
          \centering
          \includegraphics[width=\textwidth]{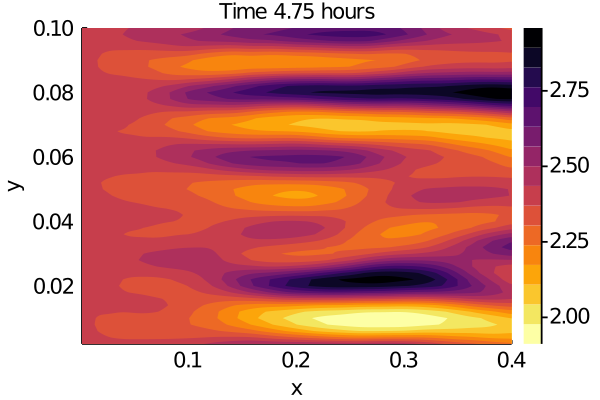}
          \caption{$\displaystyle K = \frac{K_e}{50}$}
          \label{e/50}
     \end{subfigure}
     \caption{Eroded height of the soil in $mm$ with respect to the flat surface. The darkest areas are the most eroded areas. The width of the channel are larger for larger values of $K$}
     \label{simu}
\end{figure}

\begin{figure}[h!]
	\centering
     \begin{subfigure}[b]{0.49\textwidth}
          \centering
          \includegraphics[width=\textwidth]{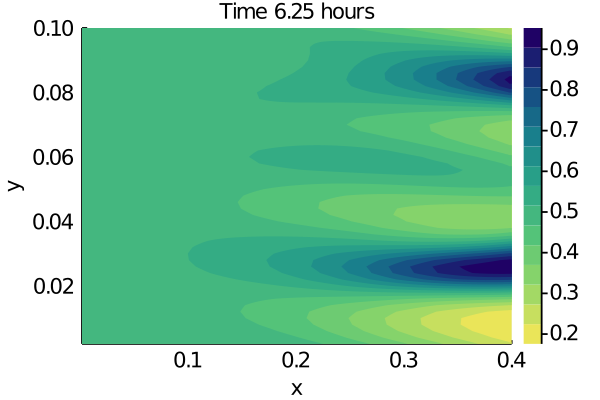}
          \caption{$\displaystyle K = \frac{K_e}{20}$}
          \label{eau_e/20}
     \end{subfigure}
     \begin{subfigure}[b]{0.49\textwidth}
          \centering
          \includegraphics[width=\textwidth]{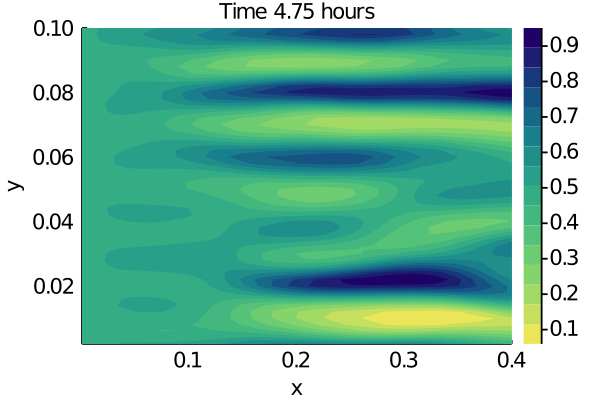}
          \caption{$\displaystyle K = \frac{K_e}{50}$}
          \label{eau_e/50}
     \end{subfigure}
     \caption{Water height on the domain, in millimeters. The water depth is greater in darker areas.}
     \label{simu_water}
\end{figure}

In Figures~\ref{simu} and \ref{simu_water}, we represent respectively the eroded surface and the fluid height when the system is unstable at some frequencies. 
In Figure~\ref{simu}, we can observe the formation of channels in the soil, in the flow direction. The width of these channels are larger when $K$ is higher, and they take more time to appear. This can be explained by the stability analysis, as discussed below. 
Figure~\ref{simu_water} have some similarities with Figure~\ref{simu} since the water tends to fill the eroded channels ; the water depth is larger in the channels, and smaller between them. \\

\begin{figure}[h!]
	\centering
    \includegraphics[width=0.7\textwidth]{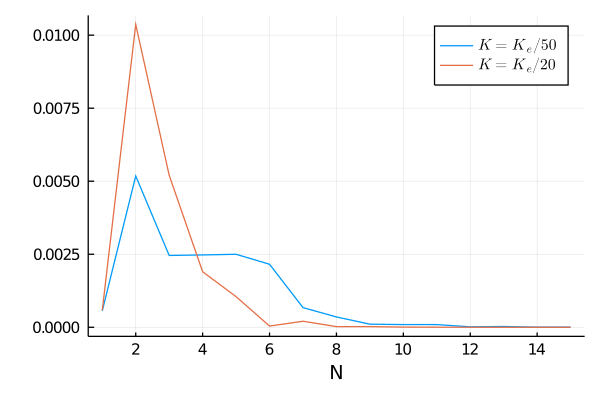}
     \caption{Discrete frequency analysis of the result of the simulation, when $K = K_e/20$ (Figure~\ref{e/20}), and $K=K_e/50$ (Figure~\ref{e/50}). The quantity represented is the norm of the discrete Fourier transform of the signal in the transverse direction.}
     \label{trFourier}
\end{figure}

We quantify the number of channels generated by the simulations in order to validate the stability analysis. For that purpose, we compute the discrete Fourier transform of the eroded surface at the end of the simulation. As the main direction of perturbations is transverse to the slope, we have calculated this Fourier transform in this direction, at fixed $x$.
In Figure~\ref{trFourier}, we represent the norm of the discrete Fourier transform (DFT) of the result of the simulations \ref{e/20},~\ref{e/50} in the transverse direction, calculated at the bottom of the plane (for $x = L_x$). The frequencies are the $2 \pi N / L_y$, for $N \in \{0, \dots, Ny/2 \}$. When $K = K_e/20$, we can see in Figure~\ref{trFourier} that the dominant frequency is reached when $N = 2$ at the bottom of the tilted plane. When $K=K_e/50$, we have contribution between $N=2$ and $N=5$ frequency

\begin{figure}[h!]
	\centering
     \begin{subfigure}[b]{0.45\textwidth}
          \centering
          \includegraphics[width=\textwidth]{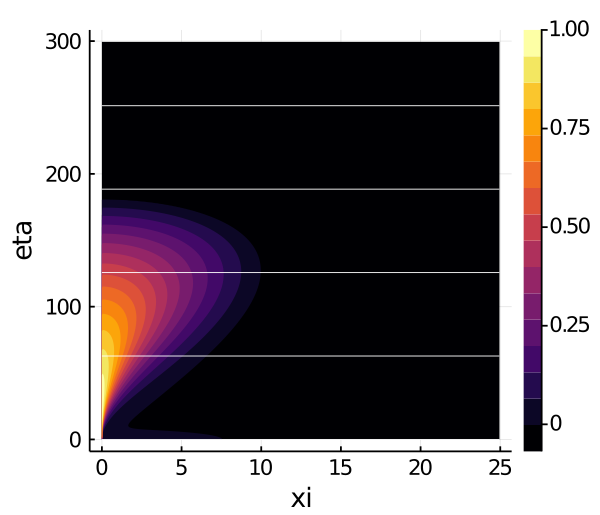}
          \caption{$K = \frac{K_e}{20}$}
          \label{stab_e/20}
     \end{subfigure}
     \begin{subfigure}[b]{0.45\textwidth}
          \centering
          \includegraphics[width=\textwidth]{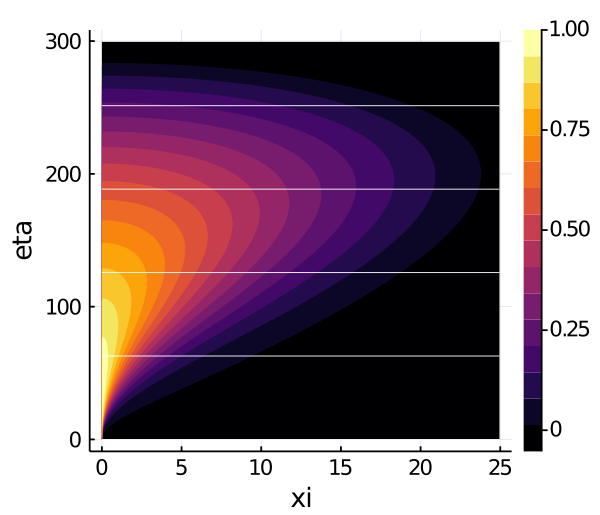}
          \caption{$K = \frac{K_e}{50}$}
          \label{stab_e/50}
     \end{subfigure}
     \caption{Maximum of the real part of the three eigenvalues, normalised by the higher one, for $(\xi,\eta) \in [0,20] \times [0,300]$. The most unstable areas are in yellow, and in the black parts the three eigenvalues have a negative real part, thus the system is stable.}
     \label{fig_stab}
\end{figure}

Then, in order to compare this observation with the stability analysis, we compute numerically the frequencies which create the most unstable modes.
In Figure~\ref{fig_stab}, we present numerical computations of the eigenvalues of the linearised system, depending on the frequencies. The largest real part of the three eigenvalues is represented, for each frequencies $\xi$ and $\eta$ in a bounded domain. This quantity control the stability, the system is stable if and only if it is negative. Due to the periodic boundary conditions, the frequencies allowed in the transverse direction are the $\eta_n = 2 \pi n / L_y$, where $L_y$ is the width of the domain. In Figure~\ref{fig_stab}, the white lines shows these frequencies. For both cases Figure~\ref{stab_e/20} and \ref{stab_e/50}, the system is unstable and the maximum of instability is reached at $\xi = 0$, $\eta_1 = 2 \pi / L_y$. Thus the system destabilises in the transverse direction, and the most destabilising frequency has period one. This is of the same order of magnitude as computed by the DFT of the simulations, where this frequency is $(\xi, \eta) = (0,\, 2 \times 2 \pi / L_y)$. The difference between the period of $1$ predicted by the stability analysis and the periods of $2-5$ obtained in the simulations could come from the non linear effects of the model, that are not taken in account in the stability analysis.

\subsection{Simulations including a source term for water}

In this part, we present some results of simulations when the source term $r$ is a positive function, to observe the effect of rain in the model. In the following simulations, $K=K_e/20$.

\begin{figure}[h!]
	\centering
     \begin{subfigure}[b]{0.49\textwidth}
          \centering
          \includegraphics[width=\textwidth]{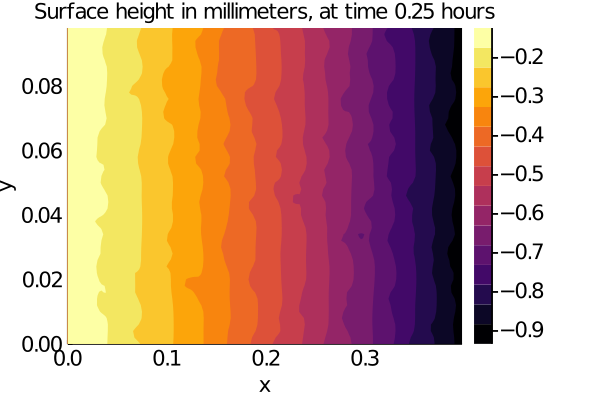}
     \end{subfigure}
     \begin{subfigure}[b]{0.49\textwidth}
          \centering
          \includegraphics[width=\textwidth]{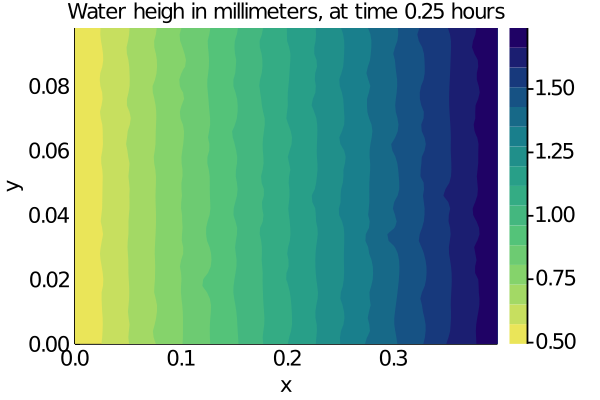}
     \end{subfigure}
     \caption{Results of the simulation when $r=r_1$, after $2000$ iterations.}
     \label{fig_simu_r2_early}
\end{figure}

First, we consider a constant in time, and uniform in space source term $r_1 = 0.005$mm/s. Figure~\ref{fig_simu_r2_early} shows the surface height (left figure) and water height (right figure) computed by the simulation, at time $0.25$ hours. We can observe that the water height is almost four times higher at the bottom of the domain ($x=L_x$), than at the top ($x=0$). As the erosion rate is proportional to a power of the water height $h$, the surface erodes faster at the bottom of the domain.

\begin{figure}[h!]
	\centering
     \begin{subfigure}[b]{0.49\textwidth}
          \centering
          \includegraphics[width=\textwidth]{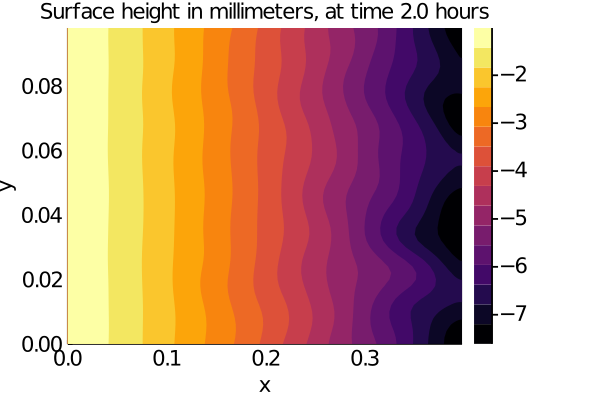}
     \end{subfigure}
     \begin{subfigure}[b]{0.49\textwidth}
          \centering
          \includegraphics[width=\textwidth]{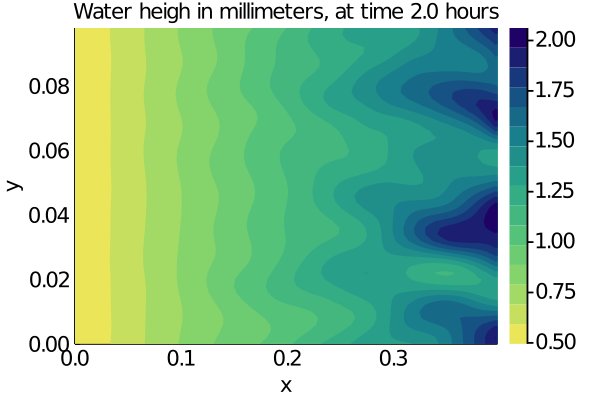}
     \end{subfigure}
     \caption{Results of the simulation when $r=r_1$, after $16000$ iterations.}
     \label{fig_simu_r2}
\end{figure}

Then, Figure~\ref{fig_simu_r2} shows the results of the same simulation, at time $2$ hours. The initial perturbations have almost disappeared, and some transverse perturbations have developed at the bottom of the domain. 

\begin{figure}[h!]
	\centering
     \begin{subfigure}[b]{0.49\textwidth}
          \centering
          \includegraphics[width=\textwidth]{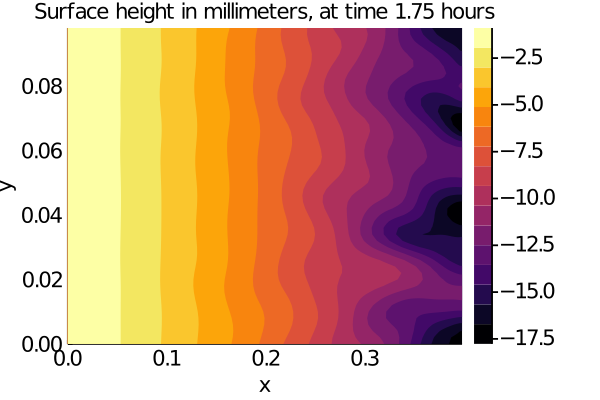}
     \end{subfigure}
     \begin{subfigure}[b]{0.49\textwidth}
          \centering
          \includegraphics[width=\textwidth]{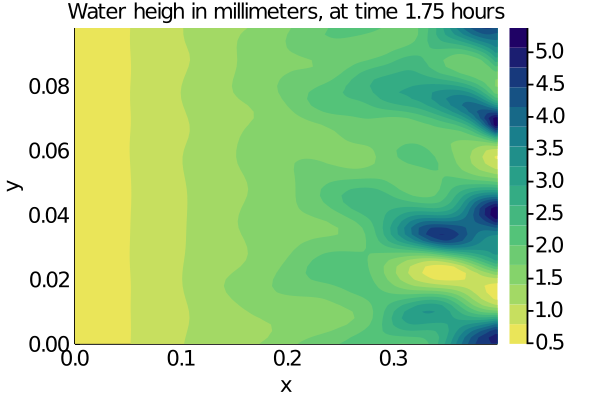}
     \end{subfigure}
     \caption{Results of the simulation when $r=r_1$, after $14000$ iterations.}
     \label{fig_simu_r1}
\end{figure}

Next, we choose a bigger source term: $r_2 = 0.01$m/s, and results of this simulation are shown in Figure~\ref{fig_simu_r1}. We can see that there are bigger perturbations at the bottom of the domain than with the source term~$r_1$. These perturbations are also not totally transverse: they undulate a little bit in the longitudinal direction, which is a different behaviour from the case without source term. 

\begin{figure}[h!]
	\centering
     \begin{subfigure}[b]{0.49\textwidth}
          \centering
          \includegraphics[width=\textwidth]{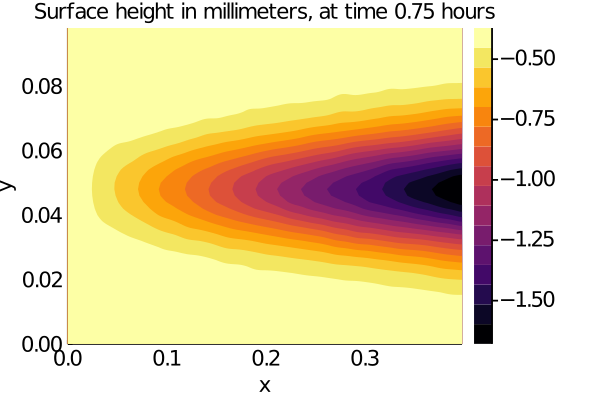}
     \end{subfigure}
     \begin{subfigure}[b]{0.49\textwidth}
          \centering
          \includegraphics[width=\textwidth]{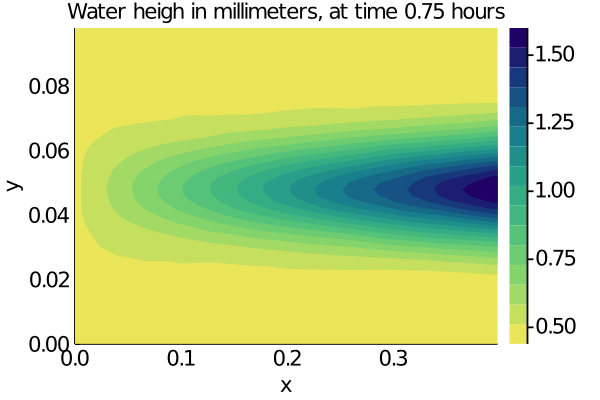}
     \end{subfigure}
     \caption{Results of the simulation when $r=r_3$, after $6000$ iterations.}
     \label{fig_simu_r3}
\end{figure}
Now we take a constant non uniform source term:
$$r_3(x,y) = 0.005 \exp \left(-60* \frac{(y-Ly/2)^2}{Ly^2} \right),$$
that correspond to a positive longitudinal band of rain in the middle of the domain which decrease exponentially fast when approaching $y=0$ or $y=0.1$. At time $0.75$ hours, the surface is highly eroded at the bottom of the band. Here, the evolution of the system is driven by the source term, the erosion landscape depends mainly on this term.

These simulation show that the source term may quantitatively influence the results, although qualitatively they look similar to the case with no source term.

\section{Conclusion and perspectives} \label{sec_ccl}

In this paper, we have considered a model for the evolution of landscape subject to water erosion, in order to study the formation of patterns. This model takes into account the water flow, the dissolved sediments and the main effects of erosion and sedimentation, while remaining simple enough to be studied both theoretically and numerically. We proved that under realistic hypotheses, the system is well posed locally in time. Then a complete spectral analysis of a linearization of the system around stationary solutions has been performed. This analysis highlights the various behaviors of the system depending on the parameters and on the frequencies of the perturbations. A very important parameter of the system is the constant of creep $K$, which controls the stability. The system can become stable if this constant is large enough, and is unstable if this constant is too small. Moreover, the instabilities in the system appears in the transverse direction, and this leads to the formation of channels parallel to the water flow. \smallskip

This analysis is a preliminary step in the study of pattern formation on erodible surfaces. We have shown that an instability mechanism can explain the formation of parallel channels at the early stages of erosion. We plan to perform a more complete parametric study of the system, in order to understand the various behaviors of the model. This analysis should be completed by numerical simulation to illustrate these behaviors.
There are several interesting direction of research. First, the analysis of this model could be extended to more general landscapes like mountains or valley. In these cases, the well-posedness of the system in short time can be still valid, as long as the water level and the water speed does not vanish. A similar stability analysis could be carried out if steady states exist.
One could also consider more involved models. Indeed a lot of factors are not taken into account as weather, vegetation, animal and human activities. It is be quite difficult to include these factors in the model, but a possible amelioration would be to include randomness in the equations. A random term could allow to model these factors which fluctuate over time. Moreover, a non constant source term, that would model the variations of rain in time would make the model more realistic. Another potential improvement of the model would be to consider more complex laws for the fluid velocity or even consider shallow water type models for the evolution of the layer of fluids.

\appendix

\section{Proof of the stability theorem} \label{app_proof th stab}

This section is devoted to the proof of Theorem~\ref{th_stabDomain}. First we state the Routh Hurwitz criteria for complex polynomial of degree $3$, proven in \cite{morris1962routh}:
\begin{proposition}
Let $P(X) = X^3 + (a_1+i b_1) X^2 + (a_2+i b_2) X + a_3+i b_3$ be a polynomial with 
complex coefficients. Then the roots of $P$ have a positive imaginary part if and only if the 
three following conditions are satisfied :
\begin{align*}
\left\{ \begin{array}{ll}
-\Delta_2 = -\left| \begin{array}{ll} 1 & a_1 \\ 0 & b_1 \end{array} \right| > 0 \\
\Delta_4 = \left| \begin{array}{llll} 1 & a_1 & a_2 & a_3 \\ 0 & b_1 & b_2 & b_3\\
     0 & 1 & a_1 & a_2 \\ 0 & 0 & b_1 & b_2 \end{array} \right| > 0, & \quad
-\Delta_6 = -\left| \begin{array}{llllll} 1 & a_1 & a_2 & a_3 & 0 & 0 \\ 
    0 & b_1 & b_2 & b_3 & 0 & 0 \\ 0 & 1 & a_1 & a_2 & a_3 & 0 \\ 
    0 & 0 & b_1 & b_2 & b_3 & 0 \\ 0 & 0 & 1 & a_1 & a_2 & a_3 \\ 
    0 & 0 & 0 & b_1 & b_2 & b_3 \end{array} \right| > 0
\end{array} \right.
\end{align*}
\end{proposition}

This proposition is directly used to prove Theorem~\ref{th_stabDomain}.

\begin{proof}[Proof of Theorem~\ref{th_stabDomain} ]
The characteristic polynomial of $A$, denoted by $P$, is :
\begin{align*}
P(X) &= 
\left| \begin{array}{ccc}
 X+i\xi \tan \theta + \alpha \epsilon^2 \underline{h} & 0 & \alpha \epsilon^2 \underline{h}\\
\displaystyle - \frac{am\underline{c}}{\underline{h}} + \frac{\alpha i\xi a n \underline{c}}{\tan \theta} 
& X + a +  i\xi \tan \theta 
& \displaystyle \frac{\alpha i\xi an\underline{c}}{\tan \theta} \\ 
\displaystyle \frac{am\underline{c}}{\bar \rho_s} - \alpha i\xi \frac{an\underline{h}\underline{c}}{\bar \rho_s \tan \theta} 
& \displaystyle - \frac{a \underline{h}}{\bar \rho_s} 
& \displaystyle X - \alpha i\xi \frac{an\underline{h}\underline{c}}{\bar \rho_s \tan \theta} + \epsilon^2 K
\end{array} \right| \\[1em]
& = 
\left| \begin{array}{ccc}
X+i\xi \tan \theta + \alpha \epsilon^2 \underline{h} & 0 & \alpha \epsilon^2 \underline{h}\\
\displaystyle - \frac{am\underline{c}}{\underline{h}}+ \frac{\alpha i\xi a n \underline{c}}{\tan \theta} & X + a + i\xi \tan \theta 
& \displaystyle \frac{\alpha i\xi an\underline{c}}{\tan \theta}\\
0 & \displaystyle \frac{\underline{h} X+ i\xi \underline{h} \tan \theta}{\bar \rho_s} &  X + \epsilon^2 K
\end{array} \right|.
\end{align*}
where we have denoted $\epsilon^2  = \xi^2 + \eta^2$. In order to simplify the computations, we introduce the variables 
$$
\displaystyle
Y := X+i\xi \tan \theta,\quad \bar \xi = \xi \tan \theta,\quad \bar \epsilon^2 = \alpha \underline{h} \epsilon^2,\quad \bar K \, \alpha \underline{h} = K\quad, \bar \rho_s \bar h = \underline{h}.$$ We also define the constants $N = \alpha an \underline{c} / \tan^2 \theta$, $M = am \underline{c} / \underline{h}$.
Consequently the polynomial writes as:
\begin{align}
P(Y) &=
\left| \begin{array}{ccc}
Y+ \bar \epsilon^2 & 0 & \bar \epsilon^2\\
- M + i\xi N & Y + a & i \bar \xi N\\
0 & \bar h Y & Y + \bar \epsilon^2 K - i \bar \xi
\end{array} \right| \nonumber \\
&= - \hh Y \left(i \bar \xi N Y + \be M \right) + (Y+a)(Y+\be )(Y-i \bar \xi + \be \bK) \nonumber \\
& \begin{array}{ll} 
     \hspace{-0.9mm}= & Y^3 + \Big[ a + \be (1+\bK) - i \bar \xi (1+N \hh) \Big] Y^2 + \Big[\be (a(1+\bK)-\hh M +\bK \be) - i \bar \xi (a+\be) \Big] Y\\ 
     & + a\be (\be \bK - i \bar \xi) .
\end{array} 
\label{polCar2}
\end{align}
The system is stable if and only if the roots of $X \mapsto P(X)$ have a negative real part. As $Y$ have the same real part as $X$, this is equivalent to the fact that the roots of $Y \mapsto P(Y)$ have a negative real part. Denoting $i\lambda := Y$, the system is stable if and only if the roots of $\lambda \mapsto P(i\lambda)$ have a positive imaginary part. Thus we will apply the Routh Hurwitz criteria to the polynomial $ \lambda \mapsto Q(\lambda) := iP(i\lambda)$ where $i\lambda := Y$:
\begin{align*}
     Q(\lambda) = & \lambda^3 - \Big[\bar \xi (1+N \hh) + i \left(a + \be (1+\bK) \right)  \Big]\lambda^2 - \Big[\be (a(1+\bK)-\hh M +\bK \be) - i\bar \xi (a+\be) \Big] \lambda\\
     &+ a\be (\bar \xi + i \be \bK).
\end{align*}
We denote:
\begin{align*}
     \left\{ \begin{array}{l} \ba = 1+N \hh,\\ \bb = a + \be (1+\bK), \end{array} \right. \qquad
     \left\{ \begin{array}{l} \baa = a(1+\bK) - \hh  M + \be \bK, \\ \bbb = a + \be. \end{array} \right.
\end{align*}
Thus $Q(\lambda) = \lambda^3 - \left(\bar \xi \ba + i\bb)  \right) \lambda^2 + \left(-\be \baa + i\bar \xi \bbb \right) \lambda + a\be (\bar \xi + i \be \bK)$.

\paragraph{The first condition : }
The determinant $-\Delta_2$ writes as :
\begin{align*}
     -\Delta_2 = -\left| \begin{array}{cc} 1 & -\bar \xi \ba \\ 0 & -\bb \end{array} \right| = \bb = a + \be (1+\bK) .
\end{align*}
Consequently, as $a > 0$ and $K \geq 0$, the condition $-\Delta_2>0$ is always satisfied.

\paragraph{The second condition : }
The determinant $\Delta_4$ is given by :
\begin{align*}
     \Delta_4 &= \left| \begin{array}{ccc} -\bb & \bar \xi \bbb &  \bar \epsilon^4 a \bK\\
     1 & -\bar \xi \ba & -\be \baa \\ 0 & -\bb & \bar \xi \bbb \end{array} \right| =
     \left| \begin{array}{ccc} -\bb & \bar \xi (\bbb - \ba \bb) &  \be (\be a \bK - \bb \baa)\\
          1 & 0 & 0 \\ 0 & -\bb & \bar \xi \bbb \end{array} \right| \\
     &= \be \left( \frac{\bar \xi^2}{\be} \bbb \left(\ba \bb - \bbb \right) + \bb \left(\baa \bb - \be a \bK \right) \right).
\end{align*}
Let $\displaystyle t = \frac{\bar \xi^2}{\be} \in [0, \frac{\mu \tan^2 \theta}{e \underline{h}}]$, then one finds that
$\Delta_4 > 0 \; \forall \, (\xi, \eta) \in (\RR^2)^*$ if and only if  $$\displaystyle t \bbb \left(\ba \bb - \bbb \right) + \bb \left(\baa \bb - \be a \bK \right) > 0 \; \forall \, t \in [0, \frac{\mu \tan^2 \theta}{e \underline{h}}], \, \forall \be >0.$$
The terms $\bb$ and $\bbb$ are positive, and :
$$\ba \bb - \bbb = a + \be (1+\bK) + N \hh (a+\be \left(1+\bK)\right) - a - \be = \be \bK + N\hh \left(a+\be(1+\bK)\right) > 0.$$
Consequently, one has $\Delta_4 > 0, \; \forall \, (\xi, \eta) \in (\RR^2)^*$ if and only if $$\baa \bb - \be a \bK \geq 0 \; \forall \, \be > 0.$$ 
We compute:
\begin{align*}
     \baa \bb - \be a \bK &= \left( a(1+\bK) - \hh  M + \be \bK \right) \left(a + \be (1+\bK) \right) - \be a \bK \\
     &= a \left(a(1+\bK)-\hh M \right) + \be (1+\bK) \left(a(1+\bK) - \hh  M + \be \bK \right) 
\end{align*}
As a consequence, one has $\Delta_4 > 0 \; \forall \, (\xi, \eta) \in (\RR^2)^*$ if and only if $a(1+\bK)- \hh M \geq 0$. This condition is equivalent to:
$$
\displaystyle
 \frac{\bar \rho_s}{\underline{c}} + \frac{K \bar \rho_s}{\alpha \underline{h} \underline{c}}-m  \geq 0.
$$

\paragraph{The third condition : }
The determinant $\Delta_6$ is given by:
\begin{align*}
     \Delta_6 &=
     -\bb \left| \begin{array}{cccc} \bar \xi (\bbb - \ba \bb) & \be ( \be a \bK -\baa \bb) & \bar \xi \be a \bb & 0\\ -\bb & \bar \xi \bbb &  \bar \epsilon^4 a \bK & 0 \\ 1 & -\bar \xi \ba & -\be \baa &  \bar \xi \be a \\ 0 & -\bb & \bar \xi \bbb &  \bar \epsilon^4 a \bK \end{array} \right|\\[1em]
     & = -\bb^2 \left| \begin{array}{ccc} \be ( \be a \bK -\baa \bb) + \bar \xi^2 \ba (\bbb - \ba \bb) & \bar \xi \be a \bb + \bar \xi \be \baa (\bbb - \ba \bb) & - \bar \xi^2 \be a (\bbb - \ba \bb) \\ \bar \xi (\bbb - \ba \bb) &  \be (\be a \bK - \baa \bb) & \bar \xi \be a \bb \\ -\bb & \bar \xi \bbb &  \bar \epsilon^4 a \bK \end{array} \right|.
\end{align*}
We can factorize out the term $\be$ in the third column. Then we develop this expression with respect to the third row. One has:
$$
\Delta_6=-\overline{\varepsilon}^6\overline{b}_1^2\left(T_0 + t T_1 + t^2 T_2\right)
$$
with
\begin{align*}
     T_0 & = \be a \bK \left(\be a \bK - \baa \bb \right)^2 > 0,\\
     T_1 & = \be a \bK \Big[\be a \bK \ba \left(\bbb - \ba \bb \right) + \left(\ba \bb - \bbb \right) \left(a \bb + \baa \bbb \right) - a \bb \bbb + a \bb \left(\ba \bb - \bbb \right)\Big]\\
     & \qquad + a \baa \bb^2 \bbb - a^2 \bb^3 \\
     & =  \be a \bK \Big[\left(\ba \bb - \bbb \right) \left(2 a \bb + \baa \bbb - \be a \ba \bK \right) - a \bb \bbb \Big] + a \bb^2 \left(\baa \bbb - a \bb \right), \\
     T_2 & = a \bbb^2 \left(\ba \bb - \bbb \right) > 0.
\end{align*}
We thus have to determine the sign of $T_1$. First, one finds that
\begin{align*}
     \ba \bbb - a \bb & = \left(a \left(\cancel{1} +\bK \right) - \hh M + \be \bK \right) \left(a + \be \right) - a \left(\cancel{a} + \be \left(\cancel{1} + \bK \right) \right)\\
     & = \left(a \bK  - \hh M + \be \bK \right) \left(a + \be \right) - \be a \bK \\
     & = a \left(a \bK - \hh M \right) + \be \left(a \bK - \hh M \right) + \bar \epsilon^4 \bK.
\end{align*}
The constant term in $T_1$ is $a^4 \left(a \bK - \hh M \right)$.
When $\be = o(t)$, that is $\epsilon^4 = o(\xi^2)$, the dominant term in $\bar \Delta_6$ is the constant term of $T_1$ times $t$, that is $a^4 \left(a \bK - \hh M \right) t$. Thus a necessary condition for the positivity of $\bar \Delta_6$ is $a \bK \geq \hh M$.

We can write $-\frac{\Delta_6}{\bb^2 \bar \epsilon^6} := -\bar \Delta_6$ as a polynomial of degree two in $a_1$, with a negative coefficient in front of $a_1^2$ :
\begin{align*}
     -\bar \Delta_6 & = - \epsilon^4 a^2 K^2 \bb t \ba^2 + \Big[ \be a \bK \bb t \left(2\ba \bb + \baa \bbb) + \be a \bK \bbb \right) + a \bb \bbb^2 t^2\Big] \ba\\
     & + a \bb^2 \left(\baa \bbb - a \bb \right) t - \be a \bK \Big[\bbb \left(2 a \bb + \baa \bbb \right) + a \bb \bbb \Big] t - a \bbb^3 t^2 + \be a K \left(\baa \bb - \be a K \right)^2.
\end{align*}
The  term $\ba$ writes as $\ba = 1 + N \hh \geq 1$. When $N = 0$ and $a \bK \geq \hh M$, we can verify that $\forall (\xi,\eta) \in (\RR^2)^*$, $-\bar \Delta_6 > 0$. Consequently when $a \bK \geq \hh M$ the polynomial $-\bar \Delta_6 (\ba)$ has two roots, the first one $x_1 < 1$ and the second one $x_2 > 1$. Now we suppose that $a \bK \geq \hh M$, and we define 
$$\delta = \inf \{x_2(\xi,\eta)) ; \, (\xi,\eta) \in (\RR^2)^*\} \quad \geq 1.$$
According to the asymptotic calculus of stability, when $a \bK \geq \hh M$ the system is stable near $0$ and $+\infty$. Consequently $\delta$ is an infimum of a continuous function on a compact set of $\RR^2$, thus it admits a minimum. To conclude, $\Delta_6 > 0 \; \forall (\xi,\eta) \in (\RR^2)^*$ if and only if $\ba < \delta$ and $a \bK \geq \hh M$. That is, denoting $\gamma = \tan^2 \theta (\delta-1)/(\alpha n \underline{h} \underline{c})$, this condition reduces to
\begin{align*}
     \displaystyle
    n < \gamma \quad \text{and} \quad \alpha m \underline{h} \underline{c} \leq \bar \rho_s K .
\end{align*}

\paragraph{The third condition when $K = 0$:}

In this case, the computations on $\Delta_6$ are fully explicit. The third  determinant reads:
\begin{align*}
     -\frac{\Delta_6}{\bb^2 \bar \epsilon^6} &= t a \bb^2 \left(\baa \bbb - a \bb \right) + t^2 a \bbb^2 \left(\ba \bb - \bbb \right)\\
     & = t a \bar h \left(a + \be \right)^3 \left( -M + t N \right).
\end{align*}
Thus the stability is directly related to the  sign of $tN-M$. As $\displaystyle t = \frac{\bar \xi^2}{\be} = \frac{\tan^2 \theta}{\alpha \underline{h}} \frac{\xi^2}{\epsilon^2}$, and $M = \frac{a m \underline{c}}{\underline{h}}$, $N = \alpha \frac{a n \underline{c}}{\tan^2 \theta}$, then this condition reads
\begin{align*}
     m < \frac{\xi^2}{\epsilon^2}  n.
\end{align*}
This concludes the proof of Theorem~\ref{th_stabDomain}.
\end{proof}

\section{Proof of the stability results at low frequencies} \label{app_stab_low}

In this section we prove Proposition ~\ref{prop_lowfr2}. 

\subsection{The first case} \label{stab low1}

In this part, we consider the case $\xi\to 0, \eta\to 0$ with $\xi = O(\eta^2)$. We write the characteristic polynomial of the matrix $A(\xi,\eta)$, and use the fact that $\lambda^i(0,0) = 0$ ($i = 1$, $2$) to calculate the two roots of this polynomial bifurcating from zero. As in the proof of Theorem~\ref{stab RH}, we define the variables $Y := X+i\xi \tan \theta$, $\bar \xi = \xi \tan \theta$, $\bar \epsilon^2 = \alpha \underline{h} \epsilon^2$, $\bar K \, \alpha \underline{h} = K$, and $\bar \rho_s \bar h = \underline{h}$. We also define the constants $N = \alpha an \underline{c} / \tan^2 \theta$, $M = am \underline{c} / \underline{h}$. We recall that the characteristic polynomial of $A$, denoted by $P$ is given by Equation~\eqref{polCar2} :
\begin{align*}
     \displaystyle P(Y) = & Y^3 + \Big[ a + \be (1+\bK) - i \bar \xi (1+N \hh) \Big] Y^2 + \Big[\be (a(1+\bK)-\hh M +\bK \be) - i \bar \xi (a+\be) \Big] Y\\
     & + a\be (\be \bK - i \bar \xi).
\end{align*}
The zero order term of the eigenvalues $\lambda^1$ and $\lambda^2$ is zero, thus $\lambda^i = O(\eta^2)$ for $i=1,2$. Consequently $\lambda^1$ and $\lambda^2$ are solution of the approximate equation:
\begin{align*}
\displaystyle
Y^2 + \Big[- i \bar \xi + \bar \eta^2 (1+\bK-\hh M/a) \Big] Y + \bar \eta^2 (\bar \eta^2 \bK - i \bar \xi)= O(\eta^6),
\end{align*}
where $\bar \eta^2 = \alpha \underline{h} \eta^2$.
Therefore, one has
\begin{align*}
\left\{ \begin{array}{ll}
\displaystyle
\lambda^1(\xi,\eta) = -\frac{i \bar \xi + \bar \eta^2 (1+\bK-\hh M/a)}{2} + \frac{1}{2} \sqrt{\Delta} + O(\eta^3),\\
\displaystyle
\lambda^2(\xi,\eta) = -\frac{i \bar \xi  \bar \eta^2 (1+\bK-\hh M/a)}{2} - \frac{1}{2} \sqrt{\Delta} + O(\eta^3), \vspace{1mm} \\
\displaystyle
\Delta = -\bar \xi^2 + 2i\bar \xi \bar \eta^2 (1+\bK-\hh M/a) + \bar \eta^4 ((1+\bK-\hh M/a)^2 -4 \bar K) + O(\eta^6).
\end{array} \right.
\end{align*}
When $\xi$ is of the order of $\eta^2$, we can write $\bar \xi = q \bar \eta^2$ with $q \neq 0$. Thus,
$$\Delta = \left(-q^2 + 2i q (1+\bK-\hh M/a) + ((1+\bK-\hh M/a)^2 -4 \bar K) \right) \bar \eta^4 + O(\eta^6).$$
Therefore we cannot compute explicitly the expression of its square roots, neither directly determine the sign of the real part of the eigenvalues. Consequently, we compute their sign. First, the real part of the square root of the discriminant $\Delta$ is given by:
\begin{align*}
     \Re(\sqrt{\Delta}) = \sqrt{\frac{Re(\Delta)+|\Delta|}{2}}, \quad \text{where:} \quad
     |\Delta|^2 & = \bar \xi^4 + 2 \bar \xi^2 \bar \eta^4 \left(2(1-\bK+\hh M/a)^2 - (1+\bK-\hh M/a)^2 +4 \bar K \right)\\
     & + \bar \eta^8 \left((1+\bK-\hh M/a)^2 - 4 \bar K \right)^2 + O(\eta^{12}).
\end{align*}

We study the sign of $\Re(\lambda^1)$: since $\Re(\lambda^1) \geq \Re(\lambda^2)$, this will determine the spectral stability of the system. One has:
\begin{align*}
Re(\lambda^1) &= -\frac{\bar \eta^2(1+\bK-\hh M/a)}{2} + \frac{1}{2}\sqrt{\frac{Re(\Delta)+|\Delta|}{2}} + O(\eta^3) > 0\\
&\Leftrightarrow \quad \sqrt{\frac{Re(\Delta)+|\Delta|}{2}} > \bar \eta^2 (1+\bK-\hh M/a) + O(\eta^3).
\end{align*}
We then study two cases, depending on the sign of $1+\bK-\hh M/a$ :
\begin{itemize}
\item If $0\leq \bK \leq \hh M/a - 1$, that vis $1+\bK-\hh M/a \leq 0$, then $Re(\lambda^1)+ O(\eta^3) > 0$. This situation can occurs only if $\hh M \geq a$.

\item If $\bar K > \hh M/a - 1$, then 
\begin{align*}
\Re(\lambda^1) > 0 \quad 
&\Leftrightarrow \quad Re(\Delta) + |\Delta| > 2\bar \eta^4 (1+\bK-\hh M/a)^2 + O(\eta^6) \quad\\
&\Leftrightarrow \quad - \bar \xi^2 - 4 \bar \eta^4 \bar K + |\Delta| > \bar \eta^4 (1+\bK-\hh M/a)^2 + O(\eta^6)\\
&\Leftrightarrow \quad |\Delta|^2 > \left(\bar \xi^2 + \bar \eta^4 ((1+\bK-\hh M/a)^2 +4 \bar K)\right)^2 + O(\eta^{12})\\
&\Leftrightarrow \quad \bar \xi^2 (\hh M/a - \bar K) > \bar \eta^4 \bar K(1+\bK-\hh M/a)^2  + O(\eta^{8}).
\end{align*}
\end{itemize}

Consequently, if $\bar K \geq \hh M/a$ then $Re(\lambda^1) + O(\eta^3) < 0$, and if $\hh M/a - 1 \leq K < \hh M/a$ then $Re(\lambda^1) + O(\eta^3) > 0 \Leftrightarrow \frac{\bar \xi^2}{\bar \eta^4} + O(\eta^4) > f(\bar K)$, with 
\begin{align}
     f(\bar K) = \bar K \frac{(1+\bK-\hh M/a)^2}{\hh M/a - \bK}.
     \label{eq_function_f}
\end{align}

\subsection{The second case} \label{second case}

In this section, we consider values of $\xi$ and $\eta$ which go to zeros, with $\eta^2 = o(\xi)$. The matrix $A$ has no zero order terms in $\xi$ and $\eta$, and no first order terms in $\eta$ so we can write the Taylor expansion of the eigenvalues as $\lambda(\xi, \eta) = \lambda_0 + \lambda_1(\xi) + \lambda_2 (\xi^2, \eta^2) + o(\xi^2, \eta^2)$. The method of the previous section gives only the first order terms of the expansion, but in this case they are imaginary. Consequently, we need to calculate the second order terms and the previous method is much more complicated in this case, so we proceed differently. The equation satisfied by $\lambda$ and $V$ is
$$\forall \xi, \eta \in \RR, \quad A(\xi, \eta) V(\xi, \eta) = \lambda(\xi, \eta) V(\xi, \eta).$$
We compute the terms of the Taylor expansion of $\lambda$ and $V$ step by step using this equation and identifying the terms of same order. The Taylor expansion of $V$ is written as $V(\xi, \eta) = V_0 + V_1(\xi) + V_2 (\xi, \eta) + o(\xi^2 + \eta^2)$ (up to a multiplicative constant). For the following calculations we will solve equations of the form $A_0X=Y$ where $Y \in Im (A_0)$. The matrix $A_0$ has a one dimensional image, generated by the vector $X^3$ (which is defined below). Consequently, the solutions $X$ are of the form $\alpha X^3 + K$, where $K \in \ker A_0$.  We also need to compute the eigenvalues and the eigenvectors of $A_0$. The eigenvalues are $ 0$, $0$, $ -a$, and the associated eigenvectors are
\begin{align*}
X^1 =
\left[ \begin{array}{c}
0 \\
0\\
1
\end{array} \right],
& \quad X^2 =
\left[ \begin{array}{c}
\underline{h} \\
m \underline{c}\\
0
\end{array} \right],
\quad X^3 =
\left[ \begin{array}{c}
0 \\
\bar \rho_s \\
-\underline{h}
\end{array} \right].
\end{align*}
Moreover, the two left eigenvectors of $A_0$ associated to the eigenvalue $0$ are used to cancel some terms in the calculus. These vectors are: 
\begin{align*}
X^{*}_1 = \left[ 1, 0, 0 \right],
& \quad X^{*}_2 = \left[0, \underline{h}, \bar \rho_s  \right].
\end{align*}

\paragraph{Order $0$ : }
When $\xi = \eta = 0$ the system $AV = \lambda V$ reduces to $A_0V_0 = \lambda_0 V_0$, so $\lambda$ and $V$ are respectively the eigenvalues and the eigenvectors of $A_0$. Consequently $\lambda_0^1 = \lambda_0^2 = 0$ and $V_0^1$, $V_0^2 \in vect(X^1,X^2)$.

\paragraph{Order $1$ :}
We identify the first order terms of the equation $A V = \lambda V$ :
\begin{equation}
A_0 V_1 + i\xi A_1 V_0 = \lambda_1 V_0,
\label{ordre1}
\end{equation}
because $\lambda_0 = 0$. In order to compute $\lambda_1^1$ and $\lambda_1^2$ we multiply the system $\eqref{ordre1}$ by $X_1^*$ and $X_2^*$. We obtain the following system, denoting $V_0 = u X^1 + v X^2$ :
\begin{align*}
\left( \begin{array}{cc}
0 & \lambda_1\ + i \xi \tan \theta\\
 \lambda_1 & m \underline{c} (\lambda_1 + \alpha \frac{i \xi}{\tan \theta})
\end{array} \right)
\left( \begin{array}{c}
u \\  v \underline{h} 
\end{array} \right) = 0.
\end{align*}
This system admits two solutions up to a multiplying factor :
\begin{align*}
\lambda_1^1 = 0, \quad V_0^1 = 
\left[ \begin{array}{l} 0 \\ 0 \\ 1 \end{array} \right] 
\quad \text{ and } \quad
\lambda_1^2 = -i \xi \tan \theta, \quad V_0^2 = 
\left[ \begin{array}{l} \underline{h} \\ m \underline{c} \\ 0 \end{array} \right].
\end{align*}
Then, using Equation~\eqref{ordre1} we compute $V_1^2$ and $V_1^2$:
\begin{align*}
V_1^1 = -i\xi \frac{\alpha n \underline{c}}{\bar \rho_s \tan \theta} X^3 + K^1, \quad V_1^2 = -i\xi \frac{\alpha n \underline{h} \underline{c}}{\bar \rho_s \tan \theta} X^3 + K^2, \quad \text{with } K^1,\,K^2 \in \ker A_0 = \text{vect} (X^1,X^2).
\end{align*}
As $X^1 $ appears in the leading order of $V^1$, we can suppose (up to renormalisation of $V^1$) that $K^1 = k_1 X^2$ with $k_1 \in \RR$. Consequently,
$$V_1^1 = -i\xi \frac{\alpha n \underline{c}}{\bar \rho_s \tan \theta} X^3 + k_1 X^2.$$
Similarly, as $V^2 = X^2 + V^1_2 + o (\xi, \eta)$ we can write 
$$V_1^2 = -i\xi \frac{\alpha n \underline{h} \underline{c}}{\bar \rho_s \tan \theta} X^3 + k_2 X^1, \quad \text{with } k_2 \in \RR.$$ 

\paragraph{Order $2$ :}

The second order terms of the system $AV^i = \lambda^i V^i$ for $i \in \{1,2\}$ are:
\begin{equation}
A_0V_2 + i\xi A_1 V_1 - (\xi^2+\eta^2)A_2 V_0 = \lambda_1 V_1 + \lambda_2 V_0.
\label{ordre2}
\end{equation}
For $i=1$, we multiply System~\eqref{ordre2} by the two left eigenvectors of $A_0$ associated to the eigenvalue $0$ and we obtain the system :
\begin{align*}
\left( \begin{array}{cc}
i \xi \tan \theta \underline{h} & 0\\
 i\xi \tan \theta m \underline{h} \underline{c} & \bar \rho_s
\end{array} \right)
\left (\begin{array}{cc} k_1\\ \lambda_1^2 \end{array} \right)
= \left( \begin{array}{c}
-(\xi^2+\eta^2) \alpha \underline{h} \\ -\xi^2 \alpha n \underline{h} \underline{c} - (\xi^2+\eta^2) K \bar \rho_s
\end{array} \right).
\end{align*}
Consequently $\lambda_2^1 = -(K+ \alpha \frac{n\underline{h} \underline{c}}{\bar \rho_s} - \alpha \frac{m\underline{h} \underline{c}}{\bar \rho_s} ) \xi^2 - (K- \alpha \frac{m\underline{h} \underline{c}}{\bar \rho_s}) \eta^2$.\\
For the second eigenvalue we multiply System $\eqref{ordre2}$ by $X_1^*$ and $X_2^*$, and we obtain :
\begin{align*}
\left( \begin{array}{cc}
0& \alpha \underline{h}\\
 -i\xi \bar \rho_s & m \underline{h} \underline{c}
\end{array} \right)
\left (\begin{array}{cc} k_2\\ \lambda_2^2 \end{array} \right)
= \left( \begin{array}{c}
-(\xi^2+\eta^2) \underline{h}^2 \\ -\xi^2 \alpha n \underline{h} \underline{c} 
\end{array} \right).
\end{align*}
One finds $\lambda_2^2 = - \alpha \underline{h} (\xi^2 + \eta^2)$.

\paragraph{Conclusion :}
Finally, the three eigenvalues of $A(\xi,\eta)$ when $\xi, \eta \to 0$ and $\xi / \eta^2 \to +\infty$ are : 
\begin{align}
\left\{ \begin{array}{l}
\displaystyle
\lambda^1 = - \left(K+\frac{c_{sat}}{\rho_s} \alpha n\underline{h}\underline{c} - \frac{c_{sat}}{\rho_s} \alpha m\underline{h}\underline{c} \right) \xi^2 - \left(K-\frac{c_{sat}}{\rho_s}  \alpha m\underline{h}\underline{c} \right) \eta^2 + o \left(\xi^2 + \eta^2 \right), \vspace{0.2cm}\\
\displaystyle
\lambda^2 = -i \xi \tan \theta - \alpha \underline{h} \left(\xi^2 + \eta^2 \right) + o \left(\xi^2 + \eta^2 \right), \vspace{0.2cm}\\
\displaystyle
\lambda^3 = - \frac{s \rho_s}{e \underline{h} c_{sat}} + o(1).
\end{array} \right.
\label{eq_lowDL}
\end{align}
Consequently, at low frequencies when $\eta^2 = o(\xi)$, the eigenvalues $\lambda^2$ and $\lambda^3$ have a negative real part. Thus the instability is given by $\lambda_1$, and depends on the parameters. When $K \bar \rho_s > m \underline{h} \underline{c}$ then $\lambda^1$ has a negative real part so System $\eqref{eq_systlin}$ is stable around the stationary solution $(\underline{h},\underline{c},\underline{z})$.  When $\underline{h} \underline{c} (m  - n) < K \bar \rho_s <m \underline{h} \underline{c}$ then the system is stable in the longitudinal direction but not in the transverse direction. In the last case, when $K \bar \rho_s < \underline{h} \underline{c} (m  - n) $ then the system is unstable.
\qed

\section{Proof of the stability results at high frequencies} \label{app_stab_high}

In this Appendix, we prove the stability result at high frequencies, Proposition ~\ref{prop_stab_high}. We use the same method as in Subsection~\ref{second case}. 

\begin{proof}[Proof of Proposition ~\ref{prop_stab_high}]

The system $AV = \lambda V$ is written as
\begin{equation}
\left( A_2 - \frac{i\xi A_1 + A_0}{\xi^2 + \eta^2} \right) V(\xi, \eta) = -\frac{\lambda(\xi,\eta)}{\xi^2+\eta^2} V(\xi,\eta).
\label{eq_infini}
\end{equation}
We can suppose, up to renormalisation that $V(\xi,\eta) = \underset{+\infty}{O} (1)$. We calculate a Taylor expansion of $$\frac{\lambda(\xi,\eta)}{\xi^2+\eta^2} = \Lambda_0 + \Lambda_1(\xi, \eta) + \Lambda_2 (\xi, \eta) + o(1),$$
and of 
$$V(\xi, \eta) = V_0 + V_{1}(\xi, \eta) + V_{2} (\xi, \eta) + o(\frac{1}{\xi^2 + \eta^2}).$$

\paragraph{Order $0$ : }
When $\xi^2+\eta^2 \to +\infty$, System~\eqref{eq_infini} becomes $A_2V_0 = -\Lambda_0 V_0$. Its solutions are the eigenvalues and eigenvectors of $-A_2$, so $\Lambda_0^1 = -K$, $\Lambda_0^2 = -\alpha \underline{h}$ and $\Lambda_0^3 = 0$. The associated eigenvectors are
\begin{align*}
X^1 =
\left[ \begin{array}{c}
\alpha \underline{h} \\ 0\\ K - \alpha \underline{h}
\end{array} \right],
& \quad X^2 =
\left[ \begin{array}{c}
1 \\ 0\\ 0
\end{array} \right],
\quad X^3 =
\left[ \begin{array}{c}
0 \\ 1\\ 0
\end{array} \right].
\end{align*}
When $K >0$ there is one eigenvalue bifurcating from zero, $\Lambda^3$, and two other eigenvalues have negative real parts. When $K = 0$ the two eigenvalues $\Lambda^1$ and $\Lambda^3$ are bifurcating from zero. Consequently we will study these two cases.

\subsection{The case K > 0}

In this part, we continue the asymptotic expansion of the eigenvalue $\Lambda^3$. The zero order terms of the eigenvectors are $V_0^1 = X^1$, $V_0^2 = X^2$ and $V_0^3 = X^3$.\\

We calculate the first order term of $\Lambda_3$. The first order terms of System~\eqref{eq_infini} when $\xi^2+\eta^2 \to + \infty$ depends on the asymptotic behaviour of $i\xi$. If $\xi \to + \infty$ then the first order terms of the system are:
\begin{equation}
A_2V_{1}^3 - \frac{i\xi A_1}{\xi^2 + \eta^2}  V_0^3 = -\Lambda_1^3 V_0^3.
\end{equation}
And if $\xi$ is bounded, $i\xi A_1$ and $A_0$ are of the same order, thus the first order terms are:
\begin{equation}
     A_2V_{1}^3 - \frac{i\xi A_1+A_0}{\xi^2 + \eta^2}  V_0^3 = -\Lambda_1^3 V_0^3.
     \end{equation}
To compute $\Lambda_1^3$ we multiply the equations by the left eigenvector of $A_2$ associated to the eigenvalue $0$, $X^* = X^3$. Then, after some computations:
\begin{align*}
\begin{array}{lll}
     \text{If } \xi \to + \infty, \quad & \displaystyle \Lambda_1^3 = -\frac{i\xi \tan \theta}{\xi^2+\eta^2}, \quad & V_1^3 = \frac{1}{\xi}
     \left[ \begin{array}{l} 0\\1\\0 \end{array} \right].\\
     \text{If $\xi$ bounded}, \quad & \displaystyle \Lambda_1^3 = -\frac{i\xi \tan \theta}{\xi^2+\eta^2} - \frac{s \rho_s}{e \underline{h} c_{sat}} \frac{1}{\xi^2+\eta^2}.
\end{array}
\end{align*}

When $\xi$ is bounded we already computed the second order term of $\Lambda^3$. When $\xi \to +\infty$, the second order terms of the system are:
\begin{equation*}
A_2V_2^3 - \frac{i\xi A_1}{\xi^2 + \eta^2}  V_1^3 - \frac{A_0}{\xi^2 + \eta^2}  V_0^3 = -\Lambda_1^3 V_1^3 - \Lambda_2^3 V_0^3.
\end{equation*}
Then, one finds:
\begin{align*}
\Lambda_2^3 = -\frac{s \rho_s}{e \underline{h} c_{sat}} \frac{1}{\xi^2+\eta^2}.
\end{align*}

Finally, the asymptotic expansion of the eigenvalues of $A(\xi,\eta)$ when $\xi^2 + \eta^2 \to +\infty$ and $K>0$ is: 
\begin{align*}
\left\{ \begin{array}{l}
\displaystyle
\lambda_1(\xi,\eta) = -K (\xi^2+\eta^2) + o(\xi^2 + \eta^2) \vspace{0.2cm}\\
\displaystyle
\lambda_2(\xi,\eta) = -\alpha \underline{h}(\xi^2+\eta^2) + o(\xi^2 + \eta^2) \vspace{0.1cm}\\
\displaystyle
\lambda_3(\xi,\eta) =  -i\xi \tan \theta - \frac{s \bar \rho_s}{e \underline{h}} + o(1)
\end{array} \right.
\end{align*}

\subsection{The case K = 0}

When $K = 0$, the zero order terms of the eigenvalues $\Lambda_1$ and $\Lambda_3 $ vanish. Thus, in order to prove the second point of Proposition ~\ref{prop_stab_high}, we need to compute the next order terms of  these two eigenvalues.\\ 

The eigenvectors associated to the two zero eigenvalues of $-A_2$ are
\begin{align*}
X^1 =
\left[ \begin{array}{c}
\alpha \underline{h} \\ 0\\ - \alpha \underline{h}
\end{array} \right],
\quad X^3 =
\left[ \begin{array}{c}
0 \\ 1\\ 0
\end{array} \right].
\end{align*}
Thus, there exist constants $u_1,v_1,u_3,v_3$ such that $V_0^1 = u_1 X^1 + v_1 X^3$ and $V_0^3 = u_3 X^1 + v_3 X^3$. The left eigenvectors associated to the zeros eigenvalues of $-C$ are:
\begin{align*}
     X^1_* =
     \left[ \begin{array}{lll}
     0 & 0 & 1
     \end{array} \right],
     \quad X^3_* =
     \left[ \begin{array}{lll}
     0 & 1 & 0
     \end{array} \right].
\end{align*}

If $\xi \to + \infty$ then the first order terms of System~\eqref{eq_infini} are:
\begin{equation}
A_2V_{1}^3 - \frac{i\xi A_1}{\xi^2 + \eta^2}  V_0^3 = -\Lambda_1^3 V_0^3.
\label{eq_DLinfty_K0_1}
\end{equation}
If $\xi \neq 0$ is fixed, the first order terms are:
\begin{equation}
     A_2V_{1}^3 - \frac{i\xi A_1+A_0}{\xi^2 + \eta^2}  V_0^3 = -\Lambda_1^3 V_0^3.
\label{eq_DLinfty_K0_2}
\end{equation}

To compute $\Lambda_1^1$ and $\Lambda_1^3$, we multiply the equations~\eqref{eq_DLinfty_K0_1} and~\eqref{eq_DLinfty_K0_2} by $X^1_*$ and $X^3_*$. Then, when $\xi \to + \infty$, one has:
\begin{align*}
     \left\{ \begin{array}{ll}
          \Lambda_1^1 = 0, & V_0^1 = X^1 ,\\
          \displaystyle \Lambda_1^3 = -\frac{i\xi \tan \theta}{\xi^2+\eta^2}, & V_0^3 = X^3.
     \end{array} \right.
\end{align*}

And when $\xi \neq 0$ is fixed, 
\begin{align*}
     \left\{ \begin{array}{l}
          \displaystyle \Lambda_1^1 = \frac{s \underline{c} m}{e \underline{h}}, \\
          \displaystyle \Lambda_1^3 = -\left(i\xi \tan \theta + \frac{s \bar \rho_s}{e \underline{h}} \right) \frac{1}{\xi^2+\eta^2}.
     \end{array} \right.
\end{align*}

When $\xi \to +\infty$, the second order terms of System~\eqref{eq_infini} are:
\begin{equation*}
A_2V_2 - \frac{i\xi A_1}{\xi^2 + \eta^2}  V_1 - \frac{A_0}{\xi^2 + \eta^2}  V_0 = -\Lambda_1 V_1 - \Lambda_2 V_0.
\end{equation*}
Then, one obtains:
\begin{align*}
\Lambda_2^1 = \frac{s \underline{c}}{e \underline{h}} \left(m - \frac{\xi^2}{\epsilon^2}n \right), \quad \Lambda_2^3 = - \frac{s \bar \rho_s}{e \underline{h}} \frac{1}{\xi^2+\eta^2}.
\end{align*}

As a conclusion, the asymptotic expansion of the eigenvalues of $A(\xi,\eta)$, when $\xi^2 + \eta^2 \to +\infty$ and $\xi \neq 0$, $K = 0$ is: 
\begin{align*}
\left\{ \begin{array}{l}
\displaystyle
\lambda_1(\xi,\eta) = \frac{s \underline{c}}{e \underline{h}} \left(m - \frac{\xi^2}{\epsilon^2}n \right) + o(1), \vspace{0.2cm}\\
\displaystyle
\lambda_2(\xi,\eta) = -\alpha \underline{h}(\xi^2+\eta^2) + o(\xi^2 + \eta^2), \vspace{0.3cm}\\
\displaystyle
\lambda_3(\xi,\eta) =  -i\xi \tan \theta - \frac{s \bar \rho_s}{e \underline{h}} + o(1).
\end{array} \right.
\end{align*}

Finally, when $\xi = 0$ the characteristic polynomial $P$ can be factorised:
$$P(X) = X \left(X^2 + (a + \alpha \epsilon^2 \underline{h}) X + \alpha \epsilon^2 a \underline{h} \left(1-\frac{m \underline{c}}{\bar \rho_s} \right) \right).$$
Thus it has a zero root, $\lambda_1 = 0$, and the two other roots are 
\begin{align*}
     \displaystyle
     \lambda_{2,3} = \frac{-\alpha \epsilon^2 \underline{h} - a \pm \sqrt{(a-\alpha \epsilon^2 \underline{h})^2 + 4 \alpha \epsilon^2 \frac{am\underline{h} \underline{c}}{\bar \rho_s}}}{2}.
\end{align*}
Consequently the asymptotic development of the eigenvalues of the system when $\eta^2 \to + \infty$ and $\xi = 0$, $K=0$ are 
\begin{align*}
     \left\{ \begin{array}{l}
     \displaystyle
     \lambda_1(0,\eta) = 0, \vspace{0.3cm}\\
     \displaystyle
     \lambda_2(0,\eta) = -\alpha \underline{h} \eta^2 + o(\eta^2), \vspace{0.1cm}\\
     \displaystyle
     \lambda_3(0,\eta) =  \frac{s \bar \rho_s}{e \underline{h}} \left( \frac{m \underline{c}}{\bar \rho_s}-1 \right) + o(1).
     \end{array} \right.
\end{align*}

This achieves the proof of Proposition ~\ref{prop_stab_high}.
\end{proof}

\section{Numerical scheme} \label{appendix_scheme}

In this appendix, we give the numerical scheme used for the simulations of the system. 

\paragraph{Stationary regime for the water}

The ratio between the erosion speed and the fluid velocity is small and we will suppose that the time derivatives of $c$ and $h$ are small. In order to justify this assumption, we rescale the equations by introducing the characteristic variables $Z$ the eroded height, $L$ the characteristic length, and $T = Z/e$ the characteristic time. We define the dimensionless variables:
\begin{align*}
&h' := \frac{h}{H}, \quad z' = \frac{z}{Z}, \quad v' := \frac{v}{V}, \quad c' := \frac{c}{c_{sat}},\\
&x' := \frac{x}{L}, \quad t' := \frac{e}{Z} t.
\end{align*}
Therefore, dropping the primes, System~\eqref{syst_complet} is written as:
\begin{align*}
\left\{\begin{array}{llll}
\displaystyle \dt h + \frac{1}{\epsilon} \dive(hv) = \frac{r}{e},\\
\displaystyle \dt (ch) + \frac{1}{\epsilon} \dive(chv) = \frac{\rho_s}{c_{sat}} \left( h^m |v|^n - \frac{s}{e} c \right),\\
\displaystyle \dt z = \widetilde K \Delta z - h^m |v|^n + \frac{s}{e} c,\\
\displaystyle v = -\frac{\mu}{V} \nabla (h+z).
\end{array} \right.
\end{align*}

Here, we defined $\epsilon := \frac{Le}{HV}$. If $\epsilon \ll 1$, as generally $\rho_s \gg c_ {sat}$, we can neglect the terms $\dt h$ and $\dt (ch)$ in the equation  $h$ et $ch$. Thus, assuming enough regularity on the solutions, we can write $\dive (chv) = hv. \nabla c + c \dive(hv) = hv. \nabla c$. Moreover we fix $V = \mu$ for simplicity. We obtain:
\begin{align}
	\left\{\begin{array}{llll}
	\displaystyle \dive(hv) = \widetilde r,\\
	hv. \nabla c = \frac{\rho_s}{c_{sat}} \left(h^m |v|^n - \frac{s}{e} c \right), \\
	\displaystyle \dt z = \widetilde K \Delta z - h^m |v|^n + \frac{s}{e} c, \\
    \displaystyle v = - \nabla (h+z).
	\end{array} \right.
     \label{eq_syst_statio}
\end{align}

\paragraph{Discretisation of the equations}

The constraint on $h$ at each time $\dive(h \nabla(h+z)) = 0$ leads to a fully nonlinear problem that may be hard to solve numerically. Instead, we discretize this equation as
\begin{align}
     \dive(h^{n-1} \nabla h^n) + \dive (h^n \nabla z^n) = -r,
\end{align}
where $h^n$ represents the fluid height at time $t_n=n\delta t$ whereas the surface height at time $n$, $z^n$, has been calculated by an explicit Euler method, using the solutions at time $n-1$, $h^{n-1}$ and $c^{n-1}$: $\forall 1 \leq i \leq N_x$, $1 \leq j \leq N_y$,
\begin{align*}
	\displaystyle \frac{z_{i,j}^n - z_{i,j}^{n-1}}{\ddt} = K \frac{z_{i+1,j}^{n-1}-2 z_{i-1,j}^{n-1}+z_{i,j}^{n-1}}{\ddx} + K \frac{z_{i,j+1}^{n-1}-2 z_{i,j-1}^{n-1}+z_{i,j}^{n-1}}{\ddy} - (E-S)_{i,j}^n.
\end{align*}
We discretise the equation on $h^n$ by a centered finite volume scheme, which writes:
\begin{align*}
	-r_{i,j} = & \frac{h_{i+1/2,j}^{n-1} (h_{i+1,j}^n-h_{i,j}^n) - h_{i-1/2,j}^{n-1} (h_{i,j}^n-h_{i-1,j}^n)}{\ddx^2} + \frac{h_{i,j+1/2}^{n-1} (h_{i,j+1}^n-h_{i,j}^n) - h_{i,j-1/2}^{n-1} (h_{i,j}^n-h_{i,j-1}^n)}{\ddy^2}\\
	&+ \frac{h_{i+1/2,j}^{n} (z_{i+1,j}^n-z_{i,j}^n) - h_{i-1/2,j}^{n} (z_{i,j}^n-z_{i-1,j}^n)}{\ddx^2} + \frac{h_{i,j+1/2}^{n} (z_{i,j+1}^n-z_{i,j}^n) - h_{i,j-1/2}^{n} (z_{i,j}^n-z_{i,j-1}^n)}{\ddy^2},
\end{align*}
where $h_{i+1/2,j}^{n-1} = \frac{h_{i,j}^{n-1}+h_{i+1,j}^{n-1}}{2}$. Finally, Equation~\eqref{eq_c_lin} on $c$ is discretised by an upwind finite volume scheme, considering the variable $x$ as a time variable. The time discretisation is an explicit Euler scheme.
\begin{align*}
	\displaystyle \frac{c_{i+1,j} - c_{i,j}}{dt} + \frac{max(v_x[i,j], 0)}{v_y[i,j]} \displaystyle \frac{c_{i,j}-c_{i,j-1}}{\ddy} + \frac{min(v_x[i,j], 0)}{v_y[i,j]} \displaystyle \frac{c_{i,j}-c_{i,j+1}}{\ddy} = - \frac{\rho_s}{h_{i,j} v_x[i,j]} (E-S)_{i,j}^n.
\end{align*}

\paragraph{Boundary conditions}

We choose periodic boundary conditions in the $y$ direction. There remains two boundaries, the top and the bottom of the tilted plane. The boundary conditions for the soil height are Neumann conditions: $\forall 1 \leq j \leq N_y$, 
\begin{align*}
     z_{0,j}^n = z_{1,j}^n, \quad z_{N_x,j}^n = z_{N_{x+1},j}^n.
\end{align*}
The boundary conditions for the water height and for the concentration of sediments are a Dirichlet condition at the top because the incoming flow is fixed, and a free flow Neumann condition at the bottom.
\begin{align*}
     \left\{ \begin{array}{l}
          \displaystyle h_{0,j}^n = h_0, \quad h_{N_x,j}^n = h_{N_{x+1},j}^n,\\
          \displaystyle c_{0,j}^n = c_0, \quad c_{N_x,j}^n = c_{N_{x+1},j}^n.   
     \end{array} \right.
\end{align*}

\paragraph{Comparison with a scheme that solves the non stationary system}

In this paragraph, we drop the stationary assumption for the water height and sediment concentration, and compare the results to those where the stationary assumption is made. We solve System~\eqref{syst_complet}, with a finite volume scheme. A fully explicit scheme fails to solve the system, because the computed solution quickly blows up, even with a time step smaller than the one given by the CFL condition:
\begin{align*}
    \delta t \leq \frac{\delta x}{v_{water}}.
\end{align*}
Therefore, we uses a semi-explicit scheme: linear terms of the water height and sediment concentration equations are implicit, and we explicit a part of the non linear terms. The discretisation in time of equations on $h$ and $c$ is given by:
\begin{align}
	\left\{\begin{array}{llll}
	\displaystyle \frac{h^{k+1} - h^k}{\delta t} + \mu \displaystyle \dive(h^h \nabla h^{k+1} + h^{k+1} \nabla z^{k+1}) = 0,\\
	\displaystyle h^{k+1} \frac{c^{k+1} - c^k}{\delta t} + h^{k+1} v^{k+1} . \nabla c^{k+1} = e \left( \left(h^{k+1} \right)^m \left|v^{k+1} \right|^n \right) - \rho_s \, s \frac{c^{k+1}}{c_{sat}}, \\
    \displaystyle v^k = - \mu \nabla (h^k+z^k).
	\end{array} \right.
     \label{eq_syst_discr_time}
\end{align}
Discretisation in space is done by a finite volume scheme, the same as for the stationary System~\eqref{eq_syst_statio}. The numerical parameters are given by Table~\ref{tab_param}, and $\delta t = 0.1$s. 

\begin{figure}[h!]
	\centering
     \begin{subfigure}[b]{0.49\textwidth}
          \centering
          \includegraphics[width=\textwidth]{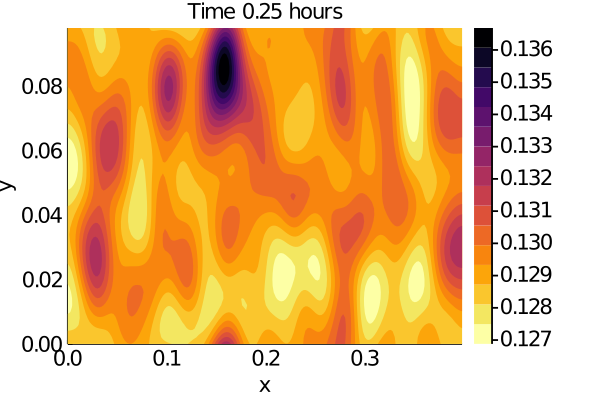}
          \caption{Non stationary scheme, 9000 iterations.}
          \label{stable_nonstatio}
     \end{subfigure}
     \begin{subfigure}[b]{0.49\textwidth}
          \centering
          \includegraphics[width=\textwidth]{sol_VFstatio_K=e_alpha=0_etape2000.png}
          \caption{Stationary scheme, 2000 iterations.}
          \label{stable_statio}
     \end{subfigure}
     \caption{Eroded height of the soil in $mm$, for $K=K_e$, at $T=2$s.}
     \label{simu_stable_nonstatio}
\end{figure}

Figure~\ref{simu_stable_nonstatio} shows the results of two simulations with the same parameters, the non stationary System~\eqref{syst_complet} in Figure~\ref{stable_nonstatio} and the stationary System~\eqref{eq_syst_statio} in Figure~\ref{stable_statio}. At $T=0.25$ hours, we can observe that these two simulations are almost the same. This result supports the fact that the approximation of the complete System~\eqref{syst_complet} by System~\eqref{eq_syst_statio}, where water is in a stationary regime is valid.

\section*{Declarations}

\paragraph{Competing interests}

No funding was received to assist with the preparation of this manuscript.

\paragraph{Availability of Data and Materials}

The code of numerical simulations generated during the current study is available at \url{https://github.com/juliebinard/landscape_evol_finite_volum}.

\paragraph{Acknowledgments}
PD holds a visiting professor association with the Department of Mathematics, Imperial College London.

\bibliographystyle{abbrv}
\bibliography{biblio.bib}

\begin{thebibliography}{10}

\bibitem{adams_sobolev_2003}
R.~A. Adams and J.~J. Fournier.
\newblock {\em Sobolev spaces}.
\newblock Elsevier, 2003.

\bibitem{anand2020linear}
S.~K. Anand, M.~Hooshyar, and A.~Porporato.
\newblock Linear layout of multiple flow-direction networks for landscape-evolution simulations.
\newblock {\em Environmental Modelling \& Software}, 133:104804, 2020.

\bibitem{andrew1993derivatives}
A.~L. Andrew, K.-W.~E. Chu, and P.~Lancaster.
\newblock Derivatives of eigenvalues and eigenvectors of matrix functions.
\newblock {\em SIAM journal on matrix analysis and applications}, 14(4):903--926, 1993.

\bibitem{bonetti2020channelization}
S.~Bonetti, M.~Hooshyar, C.~Camporeale, and A.~Porporato.
\newblock Channelization cascade in landscape evolution.
\newblock {\em Proceedings of the National Academy of Sciences}, 117(3):1375--1382, 2020.

\bibitem{brandle2005viscosity}
C.~Br{\"a}ndle and J.~L. V{\'a}zquez.
\newblock Viscosity solutions for quasilinear degenerate parabolic equations of porous medium type.
\newblock {\em Indiana University mathematics journal}, pages 817--860, 2005.

\bibitem{chen_equations_2014}
A.~Chen, J.~Darbon, G.~Buttazzo, F.~Santambrogio, and J.-M. Morel.
\newblock On the equations of landscape formation.
\newblock {\em Interfaces and Free Boundaries}, 16(1):105--136, 2014.

\bibitem{chen_landscape_2014}
A.~Chen, J.~Darbon, and J.-M. Morel.
\newblock Landscape evolution models: A review of their fundamental equations.
\newblock {\em Geomorphology}, 219:68--86, 2014.

\bibitem{cherrier_linear_2012}
P.~Cherrier and A.~Milani.
\newblock {\em Linear and quasi-linear evolution equations in Hilbert spaces}.
\newblock American Mathematical Society Providence, 2012.

\bibitem{culling1963soil}
W.~Culling.
\newblock Soil creep and the development of hillside slopes.
\newblock {\em The Journal of Geology}, 71(2):127--161, 1963.

\bibitem{culling1960analytical}
W.~E.~H. Culling.
\newblock Analytical theory of erosion.
\newblock {\em The Journal of Geology}, 68(3):336--344, 1960.

\bibitem{davis_convex_1892}
W.~Davis.
\newblock The convex profile of bad-land divides.
\newblock {\em Science}, (508):245--245, 1892.

\bibitem{escalante2021modelling}
C.~Escalante, E.~Fern{\'a}ndez-Nieto, T.~M.~d. Luna, and G.~Narbona-Reina.
\newblock Modelling of bedload sediment transport for weak and strong regimes.
\newblock {\em Numerical Simulation in Physics and Engineering: Trends and Applications: Lecture Notes of the XVIII ‘Jacques-Louis Lions’ Spanish-French School}, pages 179--189, 2021.

\bibitem{evans_partial_2010}
L.~C. Evans.
\newblock {\em Partial differential equations}, volume~19.
\newblock American Mathematical Soc., 2010.

\bibitem{fernandez2017formal}
E.~D. Fern{\'a}ndez-Nieto, T.~M. de~Luna, G.~Narbona-Reina, and J.~de~Dieu~Zabsonr{\'e}.
\newblock Formal deduction of the saint-venant--exner model including arbitrarily sloping sediment beds and associated energy.
\newblock {\em ESAIM: Mathematical Modelling and Numerical Analysis}, 51(1):115--145, 2017.

\bibitem{fernandez2014influence}
E.~D. Fernandez-Nieto, C.~Lucas, T.~M. de~Luna, and S.~Cordier.
\newblock On the influence of the thickness of the sediment moving layer in the definition of the bedload transport formula in exner systems.
\newblock {\em Computers \& Fluids}, 91:87--106, 2014.

\bibitem{gilbert_report_1877}
G.~K. Gilbert.
\newblock {\em Report on the Geology of the Henry Mountains}.
\newblock US Government Printing Office, 1877.

\bibitem{gilbert_convexity_1909}
G.~K. Gilbert.
\newblock The convexity of hilltops.
\newblock {\em The Journal of Geology}, 17(4):344--350, 1909.

\bibitem{guerin_streamwise_2020}
A.~Gu{\'e}rin, J.~Derr, S.~C. Du~Pont, and M.~Berhanu.
\newblock Streamwise dissolution patterns created by a flowing water film.
\newblock {\em Physical Review Letters}, 125(19):194502, 2020.

\bibitem{howard_channel_1983}
A.~D. Howard and G.~Kerby.
\newblock Channel changes in badlands.
\newblock {\em Geological Society of America Bulletin}, 94(6):739--752, 1983.

\bibitem{lebrun_numerical_2018}
M.~Lebrun, M.~Colom, J.~Darbon, and J.-M. Morel.
\newblock Numerical simulation of landscape evolution models.
\newblock {\em Image Processing On Line}, 8:219--250, 2018.

\bibitem{loewenherz_stability_1991}
D.~S. Loewenherz.
\newblock Stability and the initiation of channelized surface drainage: a reassessment of the short wavelength limit.
\newblock {\em Journal of Geophysical Research: Solid Earth}, 96(B5):8453--8464, 1991.

\bibitem{metivier_para-differential_2008}
G.~M{\'e}tivier.
\newblock {\em Para-differential calculus and applications to the Cauchy problem for nonlinear systems}.
\newblock 2008.

\bibitem{morris1962routh}
J.~Morris.
\newblock The routh and routh-hurwitz stability criteria: Their derivation by a novel method using comparatively elementary algebra.
\newblock {\em Aircraft engineering and aerospace technology}, 1962.

\bibitem{smith_stability_1972}
T.~R. Smith and F.~P. Bretherton.
\newblock Stability and the conservation of mass in drainage basin evolution.
\newblock {\em Water Resources Research}, 8(6):1506--1529, 1972.

\end{thebibliography}

\end{document}